\documentclass[a4paper,11pt,reqno]{amsart}

\usepackage{a4wide,amsmath,amssymb,latexsym,amsthm}
\usepackage{eucal}
\numberwithin{equation}{section}

\newcommand{\R}{\mathbb R}
\newcommand{\C}{\mathbb C}
\newcommand{\N}{\mathbb N}
\newcommand{\Z}{\mathbb Z}

\newcommand{\be}{\begin{equation}}
\newcommand{\ee}{\end{equation}}
\newcommand{\ba}{\begin{eqnarray}}
\newcommand{\ea}{\end{eqnarray}}

\renewcommand{\S}{\mathcal S}

\newcommand{\Ms}{\mathcal M}
\newcommand{\Js}{\mathcal J}
\newcommand{\Cs}{\mathcal C}

\newcommand{\norm}[1]{\left\Vert#1\right\Vert}
\newcommand{\D}{\displaystyle}

\newcommand{\n}{\nu}
\newcommand{\e}{e}
\newcommand{\Bi}[2]{\int\limits_{\partial\Omega}#1#2}
\newcommand{\Iomega}[2]{\int\limits_{\Omega}#1#2}
\newcommand{\beq}{\begin{equation}}
\newcommand{\eeq}{\end{equation}}

\newtheorem{theorem}{Theorem}[section]
\newtheorem{proposition}[theorem]{Proposition}
\newtheorem{remark}[theorem]{Remark}
\newtheorem{lemma}[theorem]{Lemma}
\newtheorem{corollary}[theorem]{Corollary}
\newtheorem{definition}[theorem]{Definition}

\usepackage[dvips]{graphicx}
\usepackage{color}

\begin{document}

\title{Control of underwater vehicles in inviscid fluids.\\
 I: Irrotational flows.}

\author{Rodrigo Lecaros$^{1,2}$}

\address{$^1$Centro de Modelamiento Matem\'atico (CMM) and Departamento de Ingenier\'ia Matem\'atica,
Universidad de Chile (UMI CNRS 2807), Avenida Blanco Encalada 2120, Casilla 170-3, Correo 3, Santiago, Chile}
\email{rlecaros@dim.uchile.cl}

\smallskip\

\address{$^2$Basque Center for Applied Mathematics - BCAM, Mazarredo 14, E-48009, Bilbao, Basque Country, Spain}
\email{rlecaros@bcamath.org}

\author{Lionel Rosier$^3$}
\address{$^3$Institut Elie Cartan, UMR 7502 UdL/CNRS/INRIA,
B.P. 70239, 54506 Vand\oe uvre-l\`es-Nancy Cedex, France}
\email{Lionel.Rosier@univ-lorraine.fr}



\begin{abstract} 
In this paper, we investigate the controllability of an underwater vehicle 
immersed in an infinite volume of an inviscid fluid whose flow is assumed to be irrotational.
Taking as control input the flow of the fluid through a part of the boundary of the rigid body, we obtain
a finite-dimensional system similar to Kirchhoff laws in which the control input appears through both linear terms (with time derivative)
and bilinear terms. Applying Coron's return method, we establish some local controllability results for the position 
and velocities of the underwater vehicle. Examples with six, four, or only three controls inputs are given for a vehicule with an ellipsoidal shape.
\end{abstract}

\maketitle
\section{Introduction}


The control of boats or submarines has attracted the attention of the mathematical
community from a long time (see e.g. 
\cite{ACO,BKMS,Chambrion-Sigalotti,fossen-book,Fossen,Lamb,Leonard97,Leonard-Marsden,NS}.)
In most of the papers devoted to that issue, the fluid is assumed 
to be inviscid, incompressible and 
irrotational, and the rigid body (the vehicle) is supposed to have an elliptic shape.   
On the other hand, to simplify the model, the control is often assumed to
appear in a linear way in a finite-dimensional system describing the 
dynamics of the rigid body, the so-called {\em Kirchhoff laws}. 

A large vessel (e.g. a cargo ship) presents often one  tunnel 
thruster built into the bow to make docking easier.
Some accurate model of a boat {\em without rudder} controlled by 
two propellers, the one displayed in a transversal bowthruster at the bow 
of the ship, the other one placed at the stern of the boat,
was derived and investigated in \cite{GR}. A local controllability result for the position and velocity (six coordinates) of a boat surrounded by an inviscid 
(not necessarily irrotational) fluid  was derived in \cite{GR} with only two controls inputs.

The aim of this paper is to provide some accurate model of a neutrally buoyant underwater vehicle immersed in an infinite volume of ideal fluid, without rudder, and
 actuated by a few number of  propellers located into some tunnels inside the rigid body, and to give a rigorous analysis 
of the control properties of such a system. 
We aim to control both the position, the attitude, and the (linear and angular) velocities
of the vehicle by taking as control input the flow of the fluid through a part of the boundary of the rigid body. The inviscid incompressible fluid
is assumed here to have an irrotational  (hence potential) flow, for the sake of simplicity. The case of a fluid with vorticity will be considered elsewhere.

Our fluid-structure interaction problem can be described as follow. The underwater vehicle, represented by a rigid body occupying a connected 
compact set $\S (t)\subset \R ^3$, is surrounded by an homogeneous incompressible perfect fluid filling the open set $\Omega (t) :=\R ^3\setminus \S (t)$ 
(as e.g. for a submarine immersed in an ocean).
We assume that $\Omega (t)$ is $C^\infty$ smooth and connected.
 Let  $\S =\S (0)$ and  
 $$\Omega =\Omega(0)= \R^3\setminus \S (0)$$ 
 denote  the initial configuration ($t=0$). Then, the dynamics of the fluid-structure system are governed by the following system of PDE's 

\begin{eqnarray}\label{euler1}
\D\frac{\partial u}{\partial t}+(u\cdot\nabla)u+\nabla p=0,&& t\in(0,T), \; x\in\Omega(t),\qquad \\ 
\textrm{div } u =0,&& t\in(0,T), \; x\in\Omega(t),\qquad \\ 
u\cdot \n=(h'+\omega  \times  (x-h))\cdot\n+w(t,x),&& t\in(0,T),\; x\in\partial\Omega(t),\qquad \\ 
\lim\limits_{|x|\to+\infty}u(t,x)=0,&& t\in (0,T),  \\ 
m_0h''=\D\int\limits_{\partial\Omega(t)} p\n \,d\sigma,&&t\in(0,T), \\ 
\D\frac{d}{dt}(QJ_0Q^\ast \omega)=\D\int\limits_{\partial\Omega(t)} (x-h)\times p\n \,d\sigma ,&&t\in(0,T),\\ 
\label{eq for Q}
Q'= S(\omega)Q ,&&t\in(0,T),\\ 
u(0,x)=u_0(x),&& x\in \Omega , \\
 \label{original_system_f}
 (h(0),Q(0),h'(0),\omega (0))=(h_0,Q_0,h_1,\omega _0)&\in&\R^3\times \text{SO}(3)\times \R^3\times \R ^3.\label{euler6}
 \end{eqnarray}
In the above equations, $u$ (resp. $p$) is the velocity field (resp. the pressure) of the fluid, $h$ denotes the position of the center of mass of the solid, 
$\omega$ denotes the angular velocity and $Q$ the 3 dimensional rotation matrix giving the orientation of the solid. 
The positive constant $m_0$ and the matrix $J_0$, which denote respectively the mass and the inertia matrix of the rigid body, are defined as
$$m_0= \int\limits_{\S} \rho(x)dx,\;\;\;J_0=\D\int\limits_{\S}\rho(x)(|x|^2 Id- xx^\ast)dx, $$
where $\rho(\cdot)$ represents the density of the rigid body. 
Finally, $\n$  is the outward unit vector to $\partial\Omega(t)$,
$x\times y$ is  the cross product between the vectors $x$ and $y$,  and $S(y)$ is the skew-adjoint matrix such that
$S(y)x=y\times x$, i.e.
$$
S(y)=
\left( \begin{array} {ccc}
0  & -y_3 & y_2 \\
y_3 & 0 & -y_1 \\
-y_2 & y_1 & 0
\end{array}
\right) .
$$
The neutral buoyancy condition reads
\be
\label{buoyancy}
 \int\limits_{\S} \rho(x)dx = \int\limits_{\S} 1dx.
\ee

$f'$ (or $\dot f$) stands for  the derivative of $f$ respect to  $t$, 
$A^\ast$ means the transpose of the matrix $A$, and $Id$ denotes the identity matrix. Finally, the term $w(t, x)$, which 
stands for the flow through the boundary of the rigid body, is taken as control input. Its support will be strictly included in $\partial\Omega(t)$, 
and actually only a finite dimensional control input will be considered here (see below (\ref{form_control}) for the precise form of 
the control term $w(t, x)$).

When no control is applied (i.e. $w(t,x)=0$), then the existence
and uniqueness of strong solutions to \eqref{euler1}-\eqref{euler6} was
obtained first in \cite{ORT1} for a ball embedded in $\R^2$, and next in \cite{ORT2} for a rigid body $\S$ of
arbitrary form (still in $\R ^2$).  The case of a ball in $\R ^3$ was investigated in  \cite{RR2008}, 
and the case of a rigid body of arbitrary form in $\R ^3$ was studied in \cite{WZ}.  
The detection of the rigid body $\S (t)$ from partial measurements of the fluid
velocity has been tackled in \cite{CCOR} when $\Omega (t)
=\Omega _0\setminus \overline{\S (t)}$ ($\Omega_0\subset \R ^2$ being a bounded cavity)
and in \cite{CMM} when  $\Omega (t)=\R ^2\setminus \overline{\S (t})$. 

Here, we are interested in the control properties of 
\eqref{euler1}-\eqref{euler6}. The controllability of Euler equations 
has been established in 2D (resp. in 3D)  
in \cite{Coron1} (resp. in \cite{Glass}). 
Note, however, that there is no hope here to
control the motion of both the fluid and the rigid body. Indeed, $\Omega (t)$ is an
exterior domain, and the vorticity is transported by the flow with a 
finite speed propagation, so that it is not affected 
(at any given time) far from the rigid body. 
Therefore, we will deal with the control of the motion of the rigid body only. 
As the state of the rigid body is described by a vector in $\R ^{12}$,
it is natural to consider a finite-dimensional control input.

Note also that since the fluid is flowing through a part of the boundary of the rigid body, additional  boundary conditions 
are needed to ensure the uniqueness of the solution of \eqref{euler1}-\eqref{euler6} (see \cite{Yudovich64}, 
\cite{Kazhikhov}). In dimension three, one can specify the tangent components of the 
vorticity $\zeta (t, x) := \textrm{curl }v(t, x)$ on the inflow section; 
that is, one can set
\beq
\label{inflow2}
\zeta (t,x)\cdot \tau _i = \zeta_0(t,x)\cdot \tau _i \;\;\textrm{for }w(t,x) < 0,\ i=1,2,
\eeq
where $\zeta_0(t, x)$ is a given function and $\tau_i$, $i=1,2$,  are linearly independent vectors tangent to 
$\partial \Omega (t)$. As we are concerned here with irrotational flows, we choose $\zeta  _0\equiv 0$.

In order to write the equations of the fluid in a {\em fixed frame}, we perform a change of coordinates. We set 
\begin{eqnarray}
x&=&Q(t)y+h(t),\\ 
v(t,y)&=&Q^\ast(t) u(t,Q(t)y+h(t)),\\ 
q(t,y)&=&p(t,Q(t)y+h(t)),\\
\label{def l}
l(t)&=&Q^\ast(t)h'(t), \\
\label{def r} 
r(t)&=&Q^\ast (t) \omega (t).
\end{eqnarray}
Then  $x$ (resp. $y$) represents the vector of coordinates of a point in a fixed frame (respectively in 
a frame linked to the rigid body). We may without loss of generality assume that 
\[
h(0)=0, \qquad Q(0)=Id.
\]
Note that, at any given time $t$, $y$ ranges over the fixed domain 
$\Omega$ when $x$ ranges over $\Omega(t)$. Finally, we assume that the control takes the form

\beq\label{form_control}
w(t, x) =  w(t, Q(t)y + h(t)) =\sum_{j=1}^m w_j (t)\chi_j(y),
\eeq
where $m \in \N^\ast$ stands for the number of independent inputs, and $w_j(t) \in \R$ is the control 
input associated with the function $\chi_j \in C^\infty(\partial\Omega)$. To ensure the conservation of the mass of the fluid, we impose the relation

\beq
\Bi{\chi_j(y)d\sigma}=0\; \textrm{ for }1\leq j\leq m.
\eeq

Then the functions $(v, q, l, r)$ satisfy the following system

 \begin{eqnarray}\label{fixed domain system_i}
\D\frac{\partial v}{\partial t} +((v-l-r\times y)\cdot\nabla)v+r\times v+ \nabla q=0,&& t\in(0,T),\; y\in\Omega ,\\ 
\label{div null} \textrm{div }v=0,&& t\in(0,T),\; y\in\Omega ,\\ 
\label{boundary condition}
\D v\cdot \n=(l+r\times y)\cdot\n+\sum\limits_{1\leq j\leq m}w_j(t)\chi_j(y),&& t\in(0,T),\; y\in\partial\Omega ,\\ 
\label{limit condition}
\lim\limits_{|y|\to+\infty}v(t,y)=0,&& t\in (0,T), \\
\label{eq. of l}
m_0\dot l=\D\int\limits_{\partial\Omega} q\n \,d\sigma -m_0r\times l,&&t\in(0,T),\\ 
\label{eq. of r}
J_0 \dot r=\D\int\limits_{\partial\Omega} q(y\times \n) \,d\sigma-r\times J_0r,&&t\in(0,T),\\
(l(0),r(0))=(h_1,\omega _0),\;v(0,y)=u_0(y).&&
\end{eqnarray}
The paper is organized as follows. In Section 2, we simplify system \eqref{euler1}-\eqref{euler6} by assuming that the fluid is potential. 
We obtain a finite dimensional system (namely \eqref{system pq}) similar to Kirchhoff laws, in which the control input $w$ appears through both linear terms (with time derivative) 
and bilinear terms. The investigation of the control properties  of  \eqref{system pq} is performed in Section 3. After noticing that the controllability of the linearized
system at the origin requires six control inputs, we apply the {\em return method}
due to Jean-Michel Coron  to take advantage of the nonlinear terms  in \eqref{system pq}. (We refer the reader to \cite{coron-book} 
for an exposition of that method for finite-dimensional systems and for PDE's.) We consider the
linearization along a certain closed-loop trajectory and obtain a local controllability
result (Theorem \ref{thm1})
assuming that two rank conditions are fulfilled, by using 
a variant of Silverman-Meadows test for the controllability of a time-varying
linear system. Some examples using symmetry properties of the rigid body are given in Section 4.
 
\section{Equations of the motion in the potential case}\label{potential flows}
In this section we derive the equations describing the motion of the rigid body subject to flow boundary control when the fluid is potential. 
\subsection{Null vorticity}
Let us denote by 
\[
\zeta (t,y)=\textrm{curl}\, v (t,y)  := (\nabla \times v)(t,y)
\]
the {\em vorticity} of the fluid.
Here, we assume that
\beq\label{curl null}
\zeta _0=\textrm{curl}\, v_0 = 0\;\;\; \textrm{ in }\Omega
\eeq
and that the {\em three} components of $\zeta$ are null at the inflow part of $\partial\Omega$, namely
\beq\label{null inflow}
\displaystyle \zeta (t,y) = 0, \;\;\;	\textrm{if }\;	y\in \cup _{1\le j\le m} \text{ Supp } \chi _j
\ \text{ and } \ \sum_{j=1}^m w_j(t)\chi_j(y) \le 0.
\eeq
\begin{proposition} 
\label{prop1}
Under the assumptions (\ref{curl null}) and (\ref{null inflow}), one has
\beq\label{rot cero}
\zeta = \textrm{\rm curl}\, v\equiv 0\;\;\textrm{ in } [0,T]\times\Omega,
\eeq
\end{proposition}
\begin{proof} Let us introduce $\tilde{v}:=v-l-r\times y$. Then it follows  from  (\ref{div null}) that 
\beq\label{div v tilde}
\textrm{div}(\tilde{v})=0,
\eeq
and 
\beq \label{rot v tilde}
\textrm{curl}(\tilde{v})=\zeta-2r.
\eeq
Applying the operator curl in (\ref{fixed domain system_i}) results in 
\beq
\frac{\partial \zeta }{\partial t}+\textrm{curl}((\tilde v\cdot\nabla)\tilde  v)+\textrm{curl}((\tilde v\cdot \nabla) (l+ r\times y)) +\textrm{curl}( r\times v)=0.
\eeq
We note that the following identities hold:
\beq\label{identity 1}
\textrm{curl}(( v\cdot\nabla) v)=( v\cdot\nabla)\textrm{curl}( v)-(\textrm{curl}( v)\cdot\nabla) v+\textrm{div}( v)\textrm{curl}( v)\eeq
and
\beq\label{identity 2}
( v\cdot \nabla) (r\times y)=r\times v,\;\;\;\textrm{curl}(r\times v)=\textrm{div}(v)r-(r\cdot\nabla)v.
\eeq
Using \eqref{div v tilde}-\eqref{identity 2}, we see that $\zeta$ satisfies 
\beq\label{equation rot}
\frac{\partial \zeta }{\partial t}+(\tilde{v}\cdot\nabla)\zeta -(\zeta\cdot\nabla)\tilde{v}=0.
\eeq
Let $\varphi = \varphi(t, s, y)$ denote the flow associated with
$\tilde{v}$, i.e.
\beq\label{flux of v tilde}
\frac{\partial\varphi}{\partial t}=\tilde{v}(t,\varphi),\;\textrm{with}\; \varphi{|_{t=s}}=y.
\eeq
We denote by $G(t,s,y)=\frac{\partial\varphi}{\partial y}(t,s,y)$ the Jacobi matrix of $\varphi$. 
Differentiating in  \eqref{flux of v tilde} with respect to $y_{j}$ ($j=1,2,3$), we see that $G(t,s,y)$ satisfies the following equation:
\beq\label{equation for G}
\frac{\partial G}{\partial t}=\D\frac{\partial\tilde{v}}{\partial y}(t,\varphi(t,s,y))\cdot G(t,s,y),\;\textrm{where}\;\; G(s,s,y)=Id\; \textrm{(identity matrix)}.
\eeq
We infer from  \eqref{div v tilde} and \eqref{equation for G} that 
\beq\label{det G}
\textrm{det } G(t, s,y)=1.
\eeq

Following Yudovich \cite{Yudovich64}, we introduce the time $t^\ast(t,y)\in [0,t ] $ at which the fluid element first appears in $\overline{\Omega}$, 
and set $y^\ast(t,y) = \varphi(t^\ast(t,y),t,y)$. Then either $t^\ast = 0$, or $t^\ast > 0$ and $y^\ast \in\cup_{1\le j\le m} \textrm{ supp } \chi _j \subset \partial \Omega $ with 
$\sum_{j=1}^m w_j(t^\ast)\chi_j(y^\ast) \le 0$. Set $f(s,t,y)=G^{-1}(s,t,y)\zeta(s,\varphi(s,t,y))$. From \eqref{equation rot}-\eqref{det G}, we obtain that
\beq\label{lemma1:eq01}
\frac{\partial f}{\partial s}(s,t,y)=0.
\eeq 
Finally,  integrating with respect to $s$ in (\ref{lemma1:eq01})  yields
\beq
\zeta (t,y)=G^{-1}(t^\ast,t,y)\zeta (t^\ast,y^\ast),
\eeq
which, combined to (\ref{curl null}) and (\ref{null inflow}), gives (\ref{rot cero}). The proof of Proposition \ref{prop1} is complete.
\end{proof}
\begin{remark}
The issue whether the result in Proposition \ref{prop1} still holds with \eqref{null inflow} replaced by 
\[
\displaystyle \zeta (t,y)\cdot \tau _i = 0, \ i=1,2, \;\;\;	\textrm{if }\;	y\in \cup _{1\le j\le m} \text{ Supp } \chi _j
\ \text{ and } \ \sum_{j=1}^m w_j(t)\chi_j(y) \le 0,
\]
 seems challenging.
We notice that the result in \cite{Kazhikhov} was proved solely when $\Omega$ was a cylinder.  
\end{remark}

\subsection{Decomposition of the fluid velocity}
It follows from (\ref{div null}), (\ref{limit condition}) and (\ref{rot cero}) that the flow is potential; that is,
\beq
\label{form v}
v=\nabla \Phi,
\eeq
where $\Phi = \Phi(t, y)$ solves
\beq
\Delta \Phi=0,\;\;\; \textrm{in }(0,T)\times\Omega, 
\eeq
\beq
\frac{\partial \Phi}{\partial \nu}=(l+r\times y)\cdot\n+\sum\limits_{1\leq j\leq m}w_j(t)\chi_j(y)\;\;\; \textrm{on }(0,T)\times\Omega ,
\eeq
\beq
\lim\limits_{|y|\to+\infty}\nabla \Phi(t,y)=0, \;\;\;\textrm{on }(0,T).
\eeq
Actually, $\Phi$ may be decomposed as
\beq\label{form Phi}
\Phi(t,y)=\sum\limits_{1\leq i\leq 3}\big\{l_i\phi_i+r_i\varphi_i\big\}+\sum\limits_{1\leq j\leq m}w_j\psi_j
\eeq
where, for $i=1,2,3$ and $j=1,...,m$,
\beq
\label{eq for phi}
\Delta\phi_i =\Delta\varphi_i =\Delta\psi_j =0 \textrm{ in } \Omega, 
\eeq
\beq\label{eq for phi:01}
\D\frac{\partial \phi_i}{\partial \n} = \n _i,\;\;\D\frac{\partial \varphi_i}{\partial \n} 
= (y\times \n)_i,\;\; \D\frac{\partial \psi_j}{\partial \n} = \chi_j   \textrm{ on }\partial \Omega ,
\eeq
\beq\label{eq for phi:02}
\D\lim\limits_{|y|\to+\infty}\nabla \phi _i(y)=0,\;\; \D\lim\limits_{|y|\to+\infty}\nabla \varphi _i(y)=0,\;\;\D\lim\limits_{|y|\to+\infty}\nabla \psi _j(y)=0 .
\eeq

As the open set $\Omega$ and the functions $\chi_j$, $1 \leq j \leq m$, 
supporting the control are assumed to be smooth, we infer that the functions $\nabla\phi_i$ ($i = 1,2,3$), 
the functions $\nabla \varphi _i$ ($i=1,2,3$) and the functions $\nabla\psi_j$ ($1 \leq j \leq m$) belong to $H^{\infty}(\Omega)$. 
\subsection{Equations for the linear and angular velocities}

For notational convenience, in what follows $\int_{\Omega} f$ (resp. $\int_{\partial \Omega} f$) stands for $\int_{\Omega } f(y)dy$  
(resp. $\int_{\partial \Omega} f(y) d\sigma (y)$).
  
Let us introduce the matrices  $M,J,N\in \R^{3\times 3}$, $C^{M},C^{J}\in \R^{3\times m}$, 
 $L^M_{p},L^J_{p},R^M_{p},R^J_{p}\in\R^{3\times 3}$, and the matrices $W^M_{p},W^J_{p}\in \R^{3\times m}$ for $p\in \{ 1, ... , m \} $ defined by
\beq \label{matrix 1}
M_{i,j} =\int\limits_{\Omega}\nabla\phi_{i} \cdot \nabla\phi_{j}  =\Bi{\n_{i}\phi_{j}}=\Bi{\frac{\partial \phi_{i}}{\partial \n}\phi_{j}},
\eeq
\beq J_{i,j}  =\int\limits_{\Omega}\nabla\varphi_{i} \cdot \nabla\varphi_{j}  =\Bi{(y\times \n)_{i}\varphi_{j}}=\Bi{\frac{\partial \varphi_{i}}{\partial \n}\varphi_{j}},
\eeq

\beq N_{i,j}=\int\limits_{\Omega}\nabla\phi_{i}\cdot \nabla\varphi_{j}=\Bi{\n_{i}\varphi_{j}}=\Bi{\phi_{i}(y\times \n)_{j}},
\eeq

\beq (C^{M})_{i,j}=\int\limits_{\Omega}\nabla\phi_{i} \cdot \nabla\psi_{j}=\Bi{\n_{i}\psi_{j}}=\Bi{\phi_{i}\chi_{j}},
\eeq

\beq (C^{J})_{i,j}=\int\limits_{\Omega}\nabla\varphi_{i} \cdot \nabla\psi_{j}=\Bi{(y\times \n)_{i}\psi_{j}}=\Bi{\varphi_{i}\chi_{j}},
\eeq

\ba
 (L^M_{p})_{i,j}=\D\Bi{(\nabla\phi_{j})_{i}\chi_{p}}, &&(L^J_{p})_{i,j}=\D\Bi{(y\times \nabla\phi_{j})_{i}\chi_{p}},  \\
 (R^M_{p})_{i,j}=\D\Bi{(\nabla\varphi_{j})_{i}\chi_{p}}, && (R^J_{p})_{i,j}=\D\Bi{(y\times \nabla\varphi_{j})_{i}\chi_{p}}, \\ 
 (W^M_{p})_{i,j}=\D\Bi{(\nabla\psi_{j})_{i}\chi_{p}}, && (W^J_{p})_{i,j}=\D\Bi{(y\times \nabla\psi_{j})_{i} \chi_{p}}. \label{matrix 6}
\ea

Note that $M^{\ast}=M$ and $J^{\ast}=J.$ 

Let us now reformulate the equations for the motion of the rigid body. We define the
matrix $\mathcal{J} \in \R^{6\times6}$ by
\beq\label{matrix J}
\mathcal{J}=
\left(\begin{array}{cc}
m_0\,  Id & 0\\
0 & J_{0}
\end{array}\right)
+\left(\begin{array}{cc}
M & N\\
N^{\ast} & J
\end{array}\right).
\eeq
It is easy to see that $\mathcal{J}$  is a (symmetric)  positive definite matrix.
We associate to the (linear and angular) velocity $(l,r)\in \R ^3\times \R ^3$ of the rigid body a momentum-like quantity, the so-called 
{\em impulse} $(P,\Pi)\in\R^{3}\times \R ^3$, defined by
\beq
\mathcal{J}\left(\begin{array}{c}
l\\
r
\end{array}\right)=\left(\begin{array}{c}
P\\
\Pi
\end{array}\right) .
\eeq
We are now in a position to give the equations governing the dynamics of the impulse.
\begin{proposition}
\label{kirchhoff}
\label{system of velocity}
The dynamics of the system are governed by the following  Kirchhoff equations
\beq\label{kirchhoff system}
\begin{array}{rcl}
\D\frac{dP}{dt}+C^{M}\dot w&=&\D (P+C^{M}w)\times r-\sum\limits_{1\leq p\leq m}w_{p}\left\{ L^M_{p}l+R^M_{p}r+W^M_{p}w\right\} ,\\ 
\\
\D\frac{d\Pi}{dt}+C^{J}\dot w&=&\D(\Pi+C^{J} w)\times r+(P +C^{M}w) \times l-\sum\limits_{1\leq p\leq m}w_{p}\left\{ L^J_{p}l+R^J_{p}r+W^J_{p}w\right\}, 
\end{array}
\eeq
where $w(t):=(w_1(t),...,w_m(t))\in\R^m$ denotes the control input.
\end{proposition}

\begin{proof} We first express the pressure $q$ in terms of $l,r,v$ and their derivatives. Using \eqref{rot cero}, we easily obtain 
\beq\label{eq nabla v}
v\cdot\nabla v=\nabla\frac{|v|^2}{2} \;\;\;\textrm{and}\;\;\; (r\times y)\cdot\nabla v-r\times v=\nabla((r\times y)\cdot v)
\eeq
Thus \eqref{fixed domain system_i} gives
$$
\begin{array}{lcl}
\D-\nabla q&=&\D \frac{\partial v}{\partial t} +\nabla \left(\frac{|v|^2}{2} -l\cdot v-(r\times y)\cdot v\right)\\ \\
&=&\D\nabla\left(\sum\limits_{1\leq i\leq 3}\big\{\dot l_i\phi_i+\dot r_i\varphi_i\big\}+\sum\limits_{1\leq j\leq m}\dot w_j\psi_j +\frac{|v|^2}{2} -l\cdot v-(r\times y)\cdot v  \right)
\end{array}
$$
hence we can take 
\beq\label{expresion q}
q=-\bigg\{\sum\limits_{1\leq i\leq 3}\big\{ \dot l_i\phi_i+ \dot r_i\varphi_i\big\}+\sum\limits_{1\leq j\leq m}\dot w_j\psi_j +\frac{|v|^2}{2} -\left(l+(r\times y)\right)\cdot v  \bigg\}
\eeq
Replacing $q$ by its value in  (\ref{eq. of l}) yields 
\beq\label{eq for l in term of v}
m_0\dot l=-m_0r\times l-\bigg\{
\sum\limits_{1\leq i\leq 3}\bigg( \dot l_i\Bi{\phi_i\n}+\dot r_i\Bi{\varphi_i\n}\bigg) 
+\sum\limits_{1\leq j\leq m}\dot w_j\Bi{\psi_j\n} +\Bi{\left(\frac{|v|^2}{2} -\left(l+(r\times y)\right)\cdot v\right)\n}
  \bigg\} .
\eeq 
Using \eqref{eq nabla v} and \eqref{div null}-\eqref{boundary condition}, we
obtain 
\ba
\D\int_{\partial \Omega} {\frac{|v|^2}{2}\nu} &=&\D \Iomega{\nabla\frac{|v|^2}{2}}\nonumber \\
&=& \D\Iomega{v\cdot\nabla v}\nonumber \\
&=&\D-\Iomega{(\textrm{div}\, v)v}+\Bi{(v\cdot\n)v}\nonumber \\
&=&\D\Bi{\left((l+r\times y)\cdot\n\right)v}+\sum\limits_{1\leq j\leq m}w_j(t)\Bi{\chi_j(y)v}.
\label{v in terms of nabla Phi}
\ea
Using Lagrange's formula:
\beq\label{identidad producCruz}
a\times(b\times c)=(a\cdot c)b-(a\cdot b)c,\;\;\forall a,b,c\in\R^3,
\eeq
we obtain that 
\beq\label{eqL l+rxy}
\Bi{\left((l+r\times y)\cdot\n\right)v-\left(\left(l+r\times y\right)\cdot v\right)\n}=\Bi{(l+r\times y)\times(v\times \n)}.
\eeq

Now we claim that
\beq\label{Control:rela f}
\Bi{\n\times\nabla f}=0,\;\;\; \forall f\in C^2(\overline{ \Omega } ).
\eeq 
To prove the claim, we introduce a smooth cutoff function $\rho_a$ such that 
$$
\rho_a(y)=\left\{
\begin{array}{ll}
1 &\text{if } \ |y|<a,\\
0 &\text{if }\  |y|>2a.
\end{array}\right.
$$
Pick a radius $a>0$ such that $\S \subset B(0,a)$, and set   
\beq\label{funcion tilda}
\tilde f(y)=f(y)\rho_a(y).
\eeq
Then 
$$\nabla \tilde f(y)=\nabla f(y),\;\;\forall y\in\partial \Omega,$$
and using the divergence theorem, we obtain 
$$\Bi{\n\times \nabla f}=\Bi{\n\times \nabla\tilde f}=\Iomega{\textrm{curl}( \nabla\tilde f})=0.$$

Therefore, using \eqref{Control:rela f} with $f=\Phi$ where $\nabla \Phi=v$, we obtain
\beq
\label{NNN}
\Bi{l\times(v\times \n)}=0.
\eeq 
Another application of \eqref{Control:rela f}  with $f=y_i\Phi$ yields 
\beq\label{rela yiV}
\Bi{y_iv\times\n}=\Bi{\n\times\e_i \Phi},
\eeq
where $\{ \e _1,\e _2 ,\e _3\}$ denotes the canonical basis in $\R ^3$. 
It follows from \eqref{NNN}, \eqref{rela yiV}, and \eqref{identidad producCruz}
that
\begin{multline}
\label{r wedge v}
\D\Bi{(l+r\times y)\times(v\times \n)}=\D r\times\Bi{\Phi\n}  \\
=\D r(t)\times \bigg( \sum_{i=1}^3\Big\{	 l_i(t)\Bi{\phi_i(y)\n(y)}+r_i(t)\Bi{\varphi_i(y)\n(y)}\Big\}+\sum_{j=1}^mw_j(t)\Bi{\psi_j(y)\n(y)}\bigg).
\end{multline}

Combining \eqref{eq for l in term of v} with 
\eqref{v in terms of nabla Phi}, \eqref{eqL l+rxy}, and \eqref{r wedge v} yelds 
 \ba
m_0\dot l&=&\D-\left\{\sum_{i=1}^3\dot l_i\Bi{\phi_i\n}+\dot r_i\Bi{\varphi_i\n}+\sum_{j=1}^m\dot w_j\Bi{\psi_j\n}\right\} \nonumber \\
& &\D-\sum_{j=1}^m w_j\left\{\sum_{i=1}^3l_i\Bi{\chi_j\nabla\phi_i }+ r_i\Bi{\chi_j\nabla\varphi_i}+\sum_{p=1}^m w_p\Bi{\chi_j\nabla\psi_p}\right\}  \nonumber \\
& &\D -r\times \left\{\sum_{i=1}^3 l_i\Bi{\phi_i\n}+ r_i\Bi{\varphi_i\n}+\sum_{j=1}^m w_j\Bi{\psi_j\n}\right\} \nonumber \\
& & \D-m_0r\times l.
\label{eq l 2}
\ea

Let us turn our attention to the dynamics of $r$. Substituting the expression of $q$ given in \eqref{expresion q} in \eqref{eq. of r} yields


\begin{multline}
\label{eq for r 2}
\D J_0 \dot r=\D-r\times J_0r-\sum\limits_{1\leq i\leq 3}\left\{\dot l_i\Bi{\phi_i (y\times \n)}+\dot r_i\Bi{\varphi_i (y\times \n)}\right\}
-\sum\limits_{1\leq j\leq m}\dot w_j\Bi{\psi_j(y\times \n)}\\
 -\D \Bi{\left(\frac{|v|^2}{2} -\left(l+(r\times y)\right)\cdot v \right)(y\times \n)} 
\end{multline}
From \cite[Proof of Lemma 2.7]{Kikuchi86}, we know that
 \[
 |v(y)|=|\nabla \Phi  (y)| = O(|y|^{-2}),\quad |\nabla v(y) | = O ( |y|^{-3} )\quad \text{ as } |y|\to \infty , 
 \]
 so that
 \be
 \label{L1}
 v \in L^2(\Omega ),  \quad  |y|\cdot |v|\cdot  |\nabla v| \in L^1(\Omega ). 
  \ee
Note that, by \eqref{eq nabla v} and  \eqref{div null}, 
\begin{eqnarray*}
\text{div} (  \frac{|v|^2}{2} (  \hat e _i \times y ) ) 
&=& \nabla (\frac{|v|^2}{2} )\cdot  (  \e _i\times y )  + \frac{|v|^2}{2}\text{div} (\e _i \times y ) \\
&=& (v\cdot \nabla v) \cdot (\e _i\times y )\\
&=& v\cdot \nabla (y\times v)_i \\
&=& \text{div} \big( (y\times v)_i  v\big) ,  
\end{eqnarray*}
and hence, using \eqref{L1} and the divergence theorem,

\ba
\D \int_{\partial \Omega} {  \frac{|v|^2}{2}  (y\times \n)_i} 
&=&\int_{\partial \Omega} {  \frac{|v|^2}{2}  (\e _i \times y )\cdot \nu }  \nonumber \\
&=& \D\Iomega{\textrm{div }\big( \frac{|v|^2}{2} (\e_i\times y)\big) } \nonumber \\
&=& \D\int_{\Omega} \textrm{div }   {\big( (y\times v)_i  v\big) }  \nonumber \\
&=& \D\Bi{(v\cdot\n) (y\times v)_i} \nonumber \\
&=&\D \Bi{(l+r\times y)\cdot\n(y\times v)_i}+\sum\limits_{1\leq j\leq m}w_j(t)\Bi{\chi_j(y\times v)_i}.
\label{eq int por partes vv}
\ea

Furthermore, using \eqref{identidad producCruz} we have that
\ba
&&\D \int_{\partial \Omega } {(l+r\times y)\cdot\n(y\times v)_i-\left(l+(r\times y)\right)\cdot v (y\times \n)_i} \nonumber\\
&&\qquad =\D \Bi{(l+r\times y)\cdot\bigg( \big( (\e_i\times y)\cdot v\big) \n -  \big( (\e_i\times y )\cdot \n\big) v\bigg)}\nonumber \\
&&\qquad =\D \Bi{(l+r\times y)\cdot\left((\e_i\times y)\times( \n\times v)\right)}.
\label{eq2 l+rxy}
\ea

Combining the following identity
\beq
\label{propiedadCruzContrac}
\sum_{j=1}^3(a\times \e_j)\times (\e_j \times b)=-(a\times b),\;\;\;\forall a,b\in\R^3
\eeq
with  \eqref{rela yiV}, we obtain
\ba
\D \int_{\partial \Omega } {l\cdot\left((\e_i\times y)\times( \n\times v)\right)}
&=&\D\sum_{j=1}^3 \ \Bi{l\cdot\left( (\e_i\times \e_j)\times( \n\times y_jv)\right)} \nonumber  \\
&=&\D\sum_{j=1}^3 l\cdot\bigg((\e_i\times \e_j)\times \Bi{(\n\times y_jv)}\bigg) \nonumber \\
&=&\D\sum_{j=1}^3 l\cdot\bigg( (\e_i\times \e_j)\times  \Bi{(\e_j\times \n)\Phi} \bigg)  \nonumber \\
&=&-\D l\cdot\Bi{\left(\e_i\times \n\right)\Phi}=\D \Bi{\left(l\times \n\right)_i\Phi}.
\label{eq 2 de l en J}
\ea

For any given $f\in C^2(\overline{\Omega })$, let 
$$I:= \D \Bi{(r\times y)\cdot\left((\e_i\times y)\times( \n\times \nabla f)\right)}.$$
$\tilde f$ still denoting  the function  defined in \eqref{funcion tilda}, we have that 
\begin{eqnarray*}
I&=& \D \sum_{j=1}^3\ \Bi{(r\times y)\cdot\left((\e_i\times y)\times( \e_j\times \nabla \tilde f)\right)\nu _j}\\
& =&\D \sum_{j=1}^3\bigg\{\Iomega{(r\times \e_j)\cdot\left((\e_i\times y)\times( \e_j\times \nabla \tilde f)\right)}\bigg\}
+\sum_{j=1}^3\bigg\{\Iomega{(r\times y)\cdot\left((\e_i\times \e_j)\times( \e_j\times \nabla \tilde f)\right)}  \bigg\} \\
&& \qquad \D+\sum_{j=1}^3\bigg\{\Iomega{(r\times y)\cdot\left((\e_i\times y)\times( \e_j\times \partial_j\nabla \tilde f)\right)}\bigg\} .
\end{eqnarray*}
Using again \eqref{propiedadCruzContrac}, we obtain
\begin{eqnarray*}
I&=&-\D \sum_{j=1}^3\left\{\Iomega{(\e_i\times y)\cdot\left((r\times \e_j)\times( \e_j\times \nabla \tilde f)\right)}\right\}-\Iomega{(r\times y)\cdot(\e_i\times \nabla \tilde f)}\\
&&\qquad \D+\Iomega{(r\times y)\cdot\left((\e_i\times y)\times \textrm{rot}(\nabla\tilde f)\right)}\\
&=&\D \Iomega{(\e_i\times y)\cdot(r\times  \nabla \tilde f)}-\Iomega{(r\times y)\cdot(\e_i\times \nabla \tilde f)}\\
&=&\D -\Iomega{r\cdot\big( (\e_i\times y)\times  \nabla \tilde f\big) }-\Iomega{r\cdot \big( y\times(\e_i\times \nabla \tilde f ) \big) }\\
&=&- \Iomega{r\cdot\left\{ (\e_i\times y)\times  \nabla \tilde f+ y\times(\e_i\times \nabla \tilde f)\right\}}\\
&=&-\D \Iomega{r\cdot\left\{\e_i\times  (y\times \nabla \tilde f)\right\}}
=\Iomega{\left(r\times  (y\times \nabla \tilde f)\right)_i}=\D\Bi{\left(r\times  (y\times \n )f\right)_i},
\end{eqnarray*}
where we used Jacobi identity
\[
a \times ( b \times c )+ b \times (c \times a ) + c \times ( a \times b) =0\qquad \forall a,b,c\in \R ^3.
\]
Letting $f=\Phi$ in the above expression yields
\be
\label{eq 2 de r en J}
\D \Bi{(r\times y)\cdot\left((\e_i\times y)\times( \n\times v)\right)}
=\D\Bi{\left(r\times  (y\times \n )\Phi\right)_i}.
\ee

Gathering together \eqref{eq for r 2}, \eqref{eq int por partes vv}, \eqref{eq2 l+rxy}, \eqref{eq 2 de l en J}), and \eqref{eq 2 de r en J} yields
\begin{eqnarray}
J_0 \dot r&=&\D \sum_{i=1}^3\Big( \dot l_{i}\Bi{(\n\times y)\phi_{i}}+\dot r_{i}\Bi{ (\n \times y) \varphi_{i}}\Big)+\sum_{j=1}^m \dot w_{j}\Bi{ (\n\times y)\psi_{j}}  \nonumber \\
&+&\D \sum_{j=1}^m w_{j}\left\{\sum_{i=1}^3\Big( l_{i}\Bi{(\nabla \phi_{i}\times y)\chi_{j}}+r_{i}\Bi{(\nabla\varphi_{i}\times y)\chi_{j}}\Big)
+\sum_{p=1}^m w_{p}\Bi{(\nabla\psi_{p}\times y)\chi_{j}}\right\}\nonumber \\
&-&\D l\times\left\{\sum_{i=1}^3\Big( l_{i}\Bi{\phi_{i}\n}+r_{i}\Bi{\varphi_{i}\n}\Big)+\sum_{p=1}^m w_{p}\Bi{\psi_{p}\n}\right\}\nonumber \\
&-&\D r\times\left\{\sum_{i=1}^3\Big( l_{i}\Bi{(y\times \n)\phi_{i}}+r_{i}\Bi{(y\times \n)\varphi_{i}}\Big)+\sum_{p=1}^m w_{p}\Bi{(y\times \n)\psi_{p}}\right\} \nonumber\\
&-&r\times J_0r.\label{eq r 2}
\end{eqnarray}
Combining  \eqref{eq l 2}  and \eqref{eq r 2} with the definitions of the matrices in \eqref{matrix 1}-\eqref{matrix 6}, we obtain 
\ba
\D m_0\dot l &=&-M \dot l- N \dot r -C^{M}\dot w  - \sum_{1\le p\le m} w_{p}\left\{ L^M_{p}l+R^M_{p}r+W^M_{p}w\right\} \nonumber \\
&& -r\times (M l+N r+C^{M}w) -m_0 r\times l,\label{S1}\\
\D J_{0}\dot r &=&- N^\ast\dot l-J \dot r -C^{J}\dot w
-\sum_{1\le p\le m} w_{p}\left\{ L^J_{p}l+R^J_{p}r+W^J_{p}w\right\} \nonumber \\
 && -l\times (M l+N r+C^{M}w) \nonumber \\
 &&- r\times (N^\ast l+J r+C^{J}w)  -r\times J_{0}r. 
\label{S2}
\ea
This completes the proof of Proposition \ref{kirchhoff}.
\end{proof}

\subsection{Equations for the position and attitude}
Now, we look at  the dynamics of the position and attitude of the rigid body. 
We shall use unit quaternions. (We refer the reader to the Appendix for the notations and definitions used in what follows.) 
From  \eqref{eq for Q} and \eqref{def r}, we obtain 
\beq Q'=S(Qr)Q=QS(r),\eeq
with $Q(0)=Id$. 

Assuming that $Q(t)$ is associated with a unit quaternion $q(t)$, i.e. $Q(t)=R(q(t))$, then the dynamics of $q$ are given by
\be
\label{AP2}
\dot q = \frac{1}{2} q * r
\ee 
(see e.g. \cite{ST}). Expanding $q$ as $q=q_0+\vec q =q_0 + q_1 i + q_2 j +q_3 k $, this yields
\be
\dot q_0 + \dot {\vec q}= \frac{1}{2} ( -\vec q\cdot r + q_0 r + \vec q \times r)  
\label{AP3}
\ee
and
\be
\left(
\begin{array}{c}
\dot q_0\\
\dot q_1\\
\dot q_2\\
\dot q_3
\end{array}
\right) 
=
\frac{1}{2}
\left(
\begin{array}{cccc}
q_0 & -q_1 & -q_2 & -q_3 \\
q_1 & q_0  & -q_3 & q_2   \\
q_2 & q_3& q_0 & - q_1 \\
q_3 & -q_2 & q_1& q_0 
\end{array}
\right) 
\left(
\begin{array}{c}
0\\
r_1\\
r_2\\
r_3
\end{array}
\right) \cdot
\label{AP30} 
\ee 

From \eqref{def l}, we see that the dynamics of $h$ are given by 
\be
\dot h(t) = Q(t)\, l(t).\label{AP5}
\ee
Again, if $Q(t)=R(q(t))$, then \eqref{AP5} may be written  as
\be
\label{AP6}
\dot h = q * l * q^*.
\ee 
Expanding $q$ as $q=q_0+\vec q = q_0 + q_1i+q_2j+q_3k$, we obtain 
\[
\dot h = (q_0+\vec q ) * l* (q_0-\vec q) = q_0^2 l + 2q_0 {\vec q}\times l   + (l\cdot \vec q\, ) \vec q  -{\vec q}\times l \times {\vec q}.  
\]
and 
\be
\left(
\begin{array}{c}
\dot h_1\\
\dot h_2\\
\dot h_3
\end{array}
\right) 
=
\left( 
\begin{array}{ccc}
q_0^2 +q_1^2 -q_2^2 -q_3^2   &  2(q_1q_2-q_0q_3)                       &  2(q_1q_3+q_0q_2)                    \\
2(q_2q_1+q_0q_3)                      &  q_0^2 -q_1^2 +q_2^2 -q_3^2   &  2(q_2q_3 -q_0q_1)                    \\
2(q_3q_1-q_0q_2)                       &  2(q_3q_2 + q_0q_1)                    &  q_0^2 -q_1^2 -q_2^2 +q_3^2
\end{array}
\right) 
\left(
\begin{array}{c}
l_1\\
l_2\\
l_3
\end{array}
\right) \cdot
\label{AP40}
\ee

For $q\in S^3_+$, $q$ may be parameterized by $\vec q$, and it is thus sufficient to consider the dynamics of $\vec q$ which read
\be
\dot{\vec q} = \frac{1}{2} (\sqrt{1- ||\vec q\, ||^2} \, r + \vec q \times r ) .
\label{AP33}
\ee
The dynamics of $h$ are then given by 
\be
\dot h = (1-|| \vec q\, ||^2) l+ 2\sqrt{1-||\vec q\, || ^2}\,  \vec q\times l  + (l\cdot \vec q ) \vec q - \vec q \times l\times \vec q.
\label{AP7}
\ee
(Alternatively, one can substitute $\sqrt{1-(q_1^2+q_2^2+q_3^2)}$ to $q_0$ in both \eqref{AP30} and \eqref{AP40}.)\\

\subsection{Control system for the underwater vehicule}
Using \eqref{AP2}, \eqref{AP6}, and Proposition \ref{system of velocity}, we arrive to  
\beq\label{system pq}
\left\{\begin{array}{ccl}
h'&=& q * l * q^* , \\[3mm]
q'&=& \displaystyle \frac{1}{2} q * r, \\[3mm]
\left(\begin{array}{c}
l\\
r
\end{array}\right)'&=&\mathcal{J} ^{-1}(C w'+F(l,r,w)),
\end{array}\right. \eeq
where $(h,q,l,r,w)\in \R ^3\times S^3\times \R ^3\times \R ^3\times \R ^m$, 
\beq\label{non lineal function}
\begin{array}{rcl}
F(l,r,w)&=& -\left(\begin{array}{cc} S(r)  & 0\\ \\ S( l )  & S(r) \end{array}\right)\left(\mathcal{J}\left(\begin{array}{c} l\\r\end{array}\right) -Cw\right) 
-\sum\limits_{p=1}^m w_{p}
\left(\begin{array}{c} \D  L^M_{p}l+R^M_{p}r+W^M_{p}w \\ \\ \D  L^J_{p}l+R^J_{p}r+W^J_{p}w 
\end{array}\right),
\end{array}
\eeq
and 
\beq\label{matrix C}
C=-\left(\begin{array}{c}C^{M}\\C^{J}\end{array}\right).
\eeq

For $q\in  S^3_+$ (i.e. $Q\in {\mathcal O}$), one can replace the two first equations in \eqref{system pq} by \eqref{AP7} 
and \eqref{AP33}, respectively. This results in the system
\beq\label{systempq}
\left\{\begin{array}{ccl}
h'&=&  (1-|| \vec q\, ||^2) l+ 2\sqrt{1-||\vec q\, || ^2}\,  \vec q\times l  + (l\cdot \vec q\, )\vec q - \vec q \times l\times \vec q, \\[3mm]
{\vec q\, }'&=& \frac{1}{2} (\sqrt{1- ||\vec q\, ||^2} \, r + \vec q \times r ), \\[3mm]
\left(\begin{array}{c}
l\\
r
\end{array}\right)'&=&\mathcal{J} ^{-1}(C w'+F(l,r,w)).
\end{array}\right. \eeq

\section{Control properties of the underwater vehicle}
\subsection{Linearization at the equilibrium}
When investigating the local controllability of a nonlinear system around an equilibrium point, it is natural to look first at its linearization
at the equilibrium point.
To linearize the system (\ref{system pq}) at the equilibrium point  $(h,q,l,r,w)=(0,1,0,0,0)$, we use the parameterization of $S^3_+$ by $\vec q$, and consider 
instead the system \eqref{systempq}. 

The linearization of \eqref{systempq} around $(h,\vec q, l, r, w)=(0,0,0,0,0)$ reads
\beq\label{lineal system pq}
\left\{\begin{array}{ccl}
h'&=&l, \\[3mm]
2{\vec q}\, '&=&\displaystyle r, \\[3mm]
\left(\begin{array}{c}
l\\
r
\end{array}\right)'&=&\mathcal{J} ^{-1}C w'.
\end{array}\right.\eeq

\begin{proposition}
\label{prop4}
The linearized system \eqref{lineal system pq} with control  $w'\in\R^{m}$ is controllable if, and only if, rank$(C)=6.$
\end{proposition}
\begin{proof} The proof follows at once from Kalman rank condition, since $(h,2{\vec q} , l,r)\in \R ^{12}$ and 
\[
\textrm{rank } \bigg(
\left(
\begin{array}{c}
0\\ 
\mathcal{J} ^{-1}C
\end{array}
\right) ,
\left(
\begin{array}{cc}
0&Id\\ 
0&0
\end{array}
\right) \,
\left(
\begin{array}{c}
0\\ 
\mathcal{J} ^{-1}C
\end{array}
\right) \bigg) =2\, \textrm{rank} (C).
\]    
\end{proof}
\begin{remark} It is easy to see that the controllability of the linearized system (\ref{lineal system pq}) implies the (local) controllability of the full system
(\ref{systempq}). The main drawback of Proposition \ref{prop4} is that the controllability of the linearized system (\ref{lineal system pq})
 requires at least 6 control inputs ($m\ge 6$).  
\end{remark}
\subsection{Simplications of the model resulting from symmetries}
Now we are concerned with the local  controllability of (\ref{systempq}) with less than 6 controls inputs. To derive tractable geometric conditions, 
we consider rigid bodies with symmetries. Let us introduce the operators $S_{i}(y)=y-2y_{i} \e _{i}$ for $i=1,2,3$, i.e. 
\beq\begin{array}{c}
S_{1}(y)=(-y_1,y_2,y_3),\\
S_{2}(y)=(y_1,-y_2,y_3),\\
S_{3}(y)=(y_1,y_2,-y_3).
\end{array}
\eeq
\begin{definition}
Let $i\in \{ 1,2,3\}$. We say that $\Omega$ is symmetric with respect to the plane $\{y_i = 0\}$ if  $S_i(\Omega)=\Omega$.
Let $f:\Omega\subset\R^3\to \R$. If $f(S_{i}(y))=\varepsilon_{f}^{i} f(y)$ for any $y\in \Omega$ and some number $\varepsilon_{f}^{i}\in\{-1,1\}$, then $f$ is 
said to be even (resp. odd) with respect to $S_i$ if $\varepsilon_f^i=1$ (resp. $\varepsilon_f^i=-1$ ).
\end{definition}

The following proposition gather several useful properties of the symmetries $S_i$, whose proofs are left to the reader.
$\delta _{ip}$ denotes the Kronecker symbol, i.e. $\delta _{ip}=1$ if $i=p$, $\delta _{ip}=0$ otherwise.
\begin{proposition}
\label{prop2}
 Let $i\in\{1,2,3\}$. Then 
\begin{enumerate}
\item $S_{i}S_{i}(a)=a$, $\forall a\in\R^{3}$;
\item $S_{i}(a)\cdot S_{i}(b)=a\cdot b$, $\forall a,b\in\R^{3}$;
\item $S_{i}(a)\times S_{i}(b)=-S_{i}(a\times b)$, $\forall a,b\in\R^{3}$;
\item If $S_{i} (\Omega)=\Omega$, then $\n(S_{i}(y))=S_{i}(\n (y))$, $\forall y \in\partial\Omega$;
\item If $f(S_{i}(y))=\varepsilon f(y)$ with $\varepsilon\in\{\pm1\}$, then $f(S_{i}(y))\n(S_{i}(y))=\varepsilon S_{i}(f(y)\n(y))$,  $\forall y \in\partial\Omega$; 
\item  If $S_{i} (\Omega)=\Omega$, then $S_{i}(y)\times \n (S_{i}(y))=- S_{i}(y\times \n(y))$, $\forall y \in\partial\Omega$;
\item Assume  that $S_{i}(\Omega)=\Omega$, and assume given a function $g:\partial\Omega \to \R$ with 
 $g(S_{i}(y))=\varepsilon g(y)$ for all $y\in \partial \Omega$, where  $\varepsilon \in\{\pm 1\}$. Then the solution $f$ to the system
$$
\left\{
\begin{array}{rll}
\Delta f=0,&&\ \textrm{in } \;\Omega,\\[2mm] 
\D\frac{\partial f}{\partial \n}=g,&&\ \textrm{on } \;\partial\Omega,\\[2mm]
\nabla f(y) \to 0,&&\ \text{as } |y|\to \infty ,
\end{array}
\right.
$$
which is defined up to an additive constant $C$, fulfills for a convenient choice of $C$
\begin{eqnarray*}
f(S_{i}(y)) &=&\varepsilon f(y),\quad \forall y\in \Omega,\\
\nabla f(S_{i}(y)) &=& \varepsilon S_{i}(\nabla f(y)) , \quad \forall y\in \Omega.
\end{eqnarray*}

\item Let $f$  and $g$ be any functions that are even or odd with respect to $S_p$ for some $p\in \{1,2,3\}$,  and let $h(y)= f(y)\partial_{\n} g(y)$. Then 
	\beq
	h(S_{p}(y))=\varepsilon_{f}^{p}\varepsilon_{g}^{p}h(y),
	\eeq
	i.e. $\varepsilon_{f\partial_{\n} g}^{p}=\varepsilon_{f}^{p}\varepsilon_{g}^{p}.$
\item Let $f$  and $g$ be as in \text{\rm (8)}, and let $h_{i}(y)= \partial_{i}f(y)\partial_{\n} g(y)$, where $i\in \{1, 2, 3\}$. Then 
	\beq
	h_{i}(S_{p}(y))=(-1)^{\delta_{ip}}\varepsilon_{f}^{p}\varepsilon_{g}^{p}h_i (y),
	\eeq
	i.e. $\varepsilon_{\partial_{i}f\partial_{\n} g}^{p}=(-1)^{\delta_{ip}}\varepsilon_{f}^{p}\varepsilon_{g}^{p}.$
\item Let $f$  and $g$ be as in \text{ \rm (8)}, and let $h_{i}(y)= (y\times \nabla f(y))_{i}\partial_{\n} g(y)$, where $i\in \{1, 2, 3\}$. Then 
	\beq
	h_{i}(S_{p}(y))=-(-1)^{\delta_{ip}}\varepsilon_{f}^{p}\varepsilon_{g}^{p}h_i (y),
	\eeq
	i.e. $\varepsilon_{(y\times \nabla f)_{i}\partial_{\n} g}^{p}=-(-1)^{\delta_{ip}}\varepsilon_{f}^{p}\varepsilon_{g}^{p}.$
\end{enumerate}
\end{proposition}

Applying Proposition \ref{prop2} to the solutions  $\phi_i,\varphi_i$, $i=1,2,3$,  of  \eqref{eq for phi}-\eqref{eq for phi:02}, we obtain at once the following result.
\begin{corollary} 
Assume that $\Omega$ is symmetric with respect to the plane $\{y_p=0\}$ (i.e. $S_p(\Omega )=\Omega$) for some $p\in\{1,2,3\}$. Then for any $j\in \{1,2,3\}$ 
	\ba
	\phi_{j}(S_{p}(y))&=&\left\{
	\begin{array}{lc}
		\phi_{j}(y) & \text{ if } \ j\neq p, \\
		-\phi_{j}(y) & \text{ if } \ j= p, 
	\end{array}\right. \\
	&=& (-1)^{\delta_{pj}}\phi_{j}(y),
	\ea
	i.e. $\varepsilon_{\phi_{j}}^{p}=(-1)^{\delta_{pj}},$ and

	\ba
	\varphi_{j}(S_{p}(y)) &=& \left\{
	\begin{array}{lc}
		-\varphi_{j}(y) & \text{ if }\  j\neq p, \\
		\varphi_{j}(y) & \text{ if } \ j= p, 
	\end{array}\right. \\
	&=& -(-1)^{\delta_{pj}}\varphi_{j}(y),
	\ea
	i.e. $\varepsilon_{\varphi_{j}}^{p}=-(-1)^{\delta_{pj}}.$
\end{corollary}

The following result shows how to exploit the symmetries of the rigid body and of the control inputs to simplify the matrices in \eqref{matrix 1}-\eqref{matrix 6}   
\begin{proposition} 
\label{prop3}
Assume that $\Omega$ is symmetric with respect to the plane $\{ y_p=0\}$ for some $p\in\{1,2,3\}$. Then 
\begin{enumerate}
\item $M_{ij}=0$ if $\varepsilon_{\phi_{i}}^{p}\varepsilon_{\phi_{j}}^{p}=-1$, i.e.
\beq\label{cond. ML}
\delta_{ip}+\delta_{jp} \equiv 1 \quad \textrm{(mod 2) };
\eeq
\item $J_{ij}=0$ if $\varepsilon_{\varphi_{i}}^{p}\varepsilon_{\varphi_{j}}^{p}=-1$, i.e.
\beq\label{cond. JR}
 \delta_{ip}+\delta_{jp}\equiv 1 \quad \textrm{(mod 2)};
\eeq
\item $N_{ij}=0$ if $ \varepsilon_{\phi_{i}}^{p}\varepsilon_{\varphi_{j}}^{p}=-1$, i.e.
\beq\label{cond. D}
 \delta_{ip}+\delta_{jp}\equiv 0 \quad \textrm{(mod 2)};
\eeq
\item $(C^{M})_{ij}=0$ if $\varepsilon_{\phi_{i}}^{p}\varepsilon_{\chi_{j}}^{p}=-1$, i.e.
\beq\label{cond. CM}
 (-1)^{\delta_{ip}}=-\varepsilon_{\chi_{j}}^{p};
\eeq
\item $(C^{J})_{ij}=0$ if $\varepsilon_{\varphi_{i}}^{p}\varepsilon_{\chi_{j}}^{p}=-1$, i.e.
\beq\label{cond. CJ}
 (-1)^{\delta_{ip}}=\varepsilon_{\chi_{j}}^{p};
\eeq
\item $(L^M_{q})_{ij}=0$ if $(-1)^{\delta_{ip}}\varepsilon_{\phi_{j}}^{p}\varepsilon_{\chi_{q}}^{p}=-1$, i.e.
\beq\label{cond. MLq}
 (-1)^{\delta_{ip}+\delta_{jp}}=-\varepsilon_{\chi_{q}}^{p};
\eeq
\item $(R^M_{q})_{ij}=0$ if  $(-1)^{\delta_{ip}}\varepsilon_{\varphi_{j}}^{p}\varepsilon_{\chi_{q}}^{p}=-1$, i.e. 
\beq\label{cond. MRq}
(-1)^{\delta_{ip}+\delta_{jp}}=\varepsilon_{\chi_{q}}^{p};
\eeq
\item $(W^M_{q})_{ij}=0$ if $ (-1)^{\delta_{ip}}\varepsilon_{\varphi_{j}}^{p}\varepsilon_{\chi_{q}}^{p}=-1$, i.e.
\beq\label{cond. MWq}
(-1)^{\delta_{ip}}=-\varepsilon_{\chi_{j}}^{p}\varepsilon_{\chi_{q}}^{p};
\eeq
\item $(L^J_{q})_{ij}=0$ if  $ -(-1)^{\delta_{ip}}\varepsilon_{\phi_{j}}^{p}\varepsilon_{\chi_{q}}^{p}=-1$, i.e.
\beq\label{cond. JLq}
(-1)^{\delta_{ip}+\delta_{jp}}=\varepsilon_{\chi_{q}}^{p};
\eeq
\item $(R^J_{q})_{ij}=0$ if  $-(-1)^{\delta_{ip}}\varepsilon_{\varphi_{j}}^{p}\varepsilon_{\chi_{q}}^{p}=-1$, i.e.
\beq\label{cond. JRq}
(-1)^{\delta_{ip}+\delta_{jp}}=-\varepsilon_{\chi_{q}}^{p};
\eeq
\item $(W^J_{q})_{ij}=0$ if 
\beq\label{cond. JWq}
(-1)^{\delta_{ip}}=\varepsilon_{\chi_{j}}^{p}\varepsilon_{\chi_{q}}^{p}, 
\eeq
\end{enumerate}
where the matrices   $M,J,N,C^{M},C^{J},L^M_{q},R^M_{q} ,  W^M_{q} , L^J_{q},R^J_{q} $ and $W^J_{q}$
 are defined in \eqref{matrix 1}-\eqref{matrix 6}.
\end{proposition}

From now on, we assume  that $\Omega$ is invariant under the operators $S_2$ and $S_3$, i.e.
\beq\label{cond. symmetric domain} 
S_{p}(\Omega)=\Omega,\;\;\forall p\in\{2,3\},
\eeq 
and that $\varepsilon_{\chi_{1}}^{p}= 1$, i.e. 
\beq\label{cond. symmetric first control}
\chi_{1}(S_{p}(y))=\chi_{1}(y)\quad \forall y\in\partial\Omega, \forall  p\in \{ 2,3\} .\eeq 
In other words,  the set $\S$  and the control $\chi _1$ are  symmetric with respect to the two planes $\{ y_{2} =0 \}$ and $\{ y_{3}=0\}$. 
As a consequence, several coefficients in the matrices in \eqref{matrix 1}-\eqref{matrix 6} vanish.

More precisely, using \eqref{cond. symmetric domain}-\eqref{cond. symmetric first control} and Proposition \ref{prop3}, we see immediately 
that the matrices in \eqref{kirchhoff system} can be written
\beq
\label{E1}
M=\left(\begin{array}{ccc}
M_{11} & 0& 0\\
0 &M_{22} & 0 \\
0 &0 &M_{33}  
\end{array}\right),\;\;\;
J=\left(\begin{array}{ccc}
J_{11} & 0& 0\\
0 &J_{22} & 0 \\
0 &0 &J_{33}  
\end{array}\right) ,
\eeq

\beq
\label{E2}
N=\left(\begin{array}{ccc}
0 & 0& 0\\
0 &0 & N_{23} \\
0 &N_{32} &0  
\end{array}\right),
\eeq

\beq
\label{E3}
C^{M} e_{1}=\left(\begin{array}{c}
(C^{M})_{11} \\
0  \\
0   
\end{array}\right),\;\;\;
C^{J} e_{1}=\left(\begin{array}{c}
 0\\
0 \\
0   
\end{array}\right),
\eeq 

 \beq
L^M_{1} =\left(\begin{array}{ccc}
(L^M_{1})_{11} & 0& 0\\
0 &(L^M_{1})_{22} & 0 \\
0 &0 &(L^M_{1})_{33}  
\end{array}\right),\;\;\;
R^M_{1} =\left(\begin{array}{ccc}
0 & 0& 0\\
0 &0 & (R^M_{1})_{23} \\
0 &(R^M_{1})_{32} &0  
\end{array}\right)
\eeq

\beq
(W^M_{1}) e_1=\left(\begin{array}{c}
 (W^M_{1})_{11}\\
 0 \\
0  
\end{array}\right),\;\;\;
L^J_{1} =\left(\begin{array}{ccc}
0 & 0& 0\\
0 &0 & (L^J_{1})_{23} \\
0 &(L^J_{1})_{32} &0  
\end{array}\right),
\eeq
and
\beq
R^J_{1} =\left(\begin{array}{ccc}
(R^J_{1})_{11} & 0& 0\\
0 &(R^J_{1})_{22} & 0 \\
0 &0 & (R^J_{1})_{33} 
\end{array}\right),\;\;\;
(W^J_{1}) e_1=\left(\begin{array}{c}
 0\\
 0 \\
0  
\end{array}\right).
\eeq

\subsection{Toy problem}
Before investigating the full system \eqref{systempq}, it is very important to look at the simplest situation for which 
$h_i=l_i=0$ for $i=2,3$, $\vec q =0$, $r=0$, and $w_j=0$ for $j=2,..,m$.
\begin{lemma} 
\label{lemma simple} Assume that
 \eqref{cond. symmetric domain}-\eqref{cond. symmetric first control} hold, and
assume given some functions $h_{1},l_{1},w_{1}\in C^1([0,T])$ satisfying
\beq\label{system simple}
\left\{\begin{array}{rcl}
h'_{1}&=&l_{1}\\ \\
l'_{1}&=&\D\alpha w'_{1}+\beta l_{1}w_{1}+\gamma (w_{1})^{2},
\end{array}\right.
\eeq
where 
$$\alpha:= \frac{-(C^M)_{11}}{m_0+M_{11}},\;\;\; \beta:=\frac{-(L^{M}_1)_{11}}{m_0+M_{11}},\;\;\textrm{ and }\gamma:=\frac{-(W^M_{1})_{11}}{m_0+M_{11}} \cdot$$  
Let $h:=(h_{1},0,0)$, $\vec q :=(0,0,0)$, $l:=(l_{1},0,0)$,  $r:=(0,0,0)$, and $w:=(w_{1},0,...,0)$.
Then $(h,\vec q,l,r,w)$  solves \eqref{systempq}.
 \end{lemma}
 
\begin{proof}
  Let us set $h=h_1e_1$, $\vec q =0$, $l=l_{1}e_{1}$, $r=0$ and  $w=(w_{1},0,...,0)$, where $(h_1,l_1,w_1)$ 
 fulfills  \eqref{system simple}. From  \eqref{E1}-\eqref{E3},
   we have that
 \beq 
 \mathcal{J}\left(\begin{array}{c} l\\r\end{array}\right) 
 =l_{1}\mathcal{J} e_1=l_{1}\left(\begin{array}{c}m_0 +M_{11}\\0\\0\\0\\0\\0\end{array}\right),
 \eeq
and 
\beq  Cw
=w_{1}C e_1=-w_{1}\left(\begin{array}{c}(C^{M})_{11}\\0\\0\\0\\0\\0\end{array}\right).
\eeq
This yields
\beq  
\left(\begin{array}{cc} S(r)  & 0\\ \\  S( l )   &  S(r)  \end{array}\right)
\left(\mathcal{J}\big(\begin{array}{c} l\\r\end{array}\big) -Cw\right) =0.
\eeq
Replacing in (\ref{non lineal function}), we obtain   
\beq\begin{array}{rcl}
 F(l,r,w)&=&  -\sum\limits_{p=1}^m w_{p}\left(\begin{array}{c} \D  L^M_{p}l+R^M_{p}r+W^M_{p}w \\ \\
\D  L^J_{p}l+R^J_{p}r+W^J_{p}w 
\end{array}\right)\\ \\
&=&  -w_{1}\left(\begin{array}{c} \D  L^M_{1}l+W^M_{1}w \\ \\
\D  L^J_{1}l+W^J_{1}w 
\end{array}\right)\\ \\
&=&  -w_{1}\bigg( \D  l_{1}\left(\begin{array}{c}
(L^M_{1})_{11} \\
0  \\
0  \\
0\\
0\\
0 
\end{array}\right) +w_{1}\left(\begin{array}{c}
 (W^M_{1})_{11}\\
 0 \\
0  \\
0\\
0\\
0
\end{array}\right)  
\bigg) .
\end{array}
\eeq
We conclude that $(h,\vec q,l,r,w)$ is a solution of  (\ref{systempq}).
 \end{proof}
\begin{remark}
If $\gamma + \alpha\beta=0$, then it follows from \cite[Lemma 2.3]{GR} that for any $T>0$ we may associate with any pair $(h_1^0,h_1^T)$ in $\R ^2$
a control input $w_1\in C_0^\infty (0,T)$ such that the solution $(h_1(t),l_1(t))$ of \eqref{system simple} emanating from 
$(h_1^0,0)$ at $t=0$ reaches $(h_1^T,0)$ at $t=T$. 
\end{remark}
\subsection{Return method}
The main result in this section (see below Theorem \ref{thm1}) is derived in following a strategy developed in \cite {GR} and inspired in part from
Coron's return method.  We first construct a (non trivial) loop-shaped trajectory of the control system \eqref{systempq}, which is based on the computations
performed in Lemma \ref{lemma simple}. (For this simple control system, we can require that $w_1(0)=0$, but we cannot in general require that
$w_1(T)=0$.)
Next, we compute the linearized system along the above reference trajectory. We use a controllability test
from \cite{GR} to investigate the controllability of the linearized system, in which the control appears with its time derivative. Finally, we derive
the (local) controllability of the nonlinear system by a standard linearization argument. 
 
\subsubsection{Construction of a loop-shaped trajectory.}
The construction differs slightly from those in \cite{GR}: indeed, to simplify the computations, we impose here that all the derivatives of $\overline{l}_1$ of order
larger than two vanish at $t=T$. 
 For given $T>0$, let  $\xi\in C^\infty (\R; [0,1])$ be a function such that 
 \[
 \xi (t) = \left\{ 
 \begin{array} {ll}
 0 & \text{ if } \ \displaystyle  t <\frac{T}{3}, \\[3mm]
 1 &\text{ if } \ \displaystyle t>\frac{2T}{3}.
 \end{array}
 \right.
 \]
 Pick any $\lambda_0 > 0$ and let $\lambda \in [-\lambda_0, \lambda_0]$ with $\lambda \ne 0$. 
Set
\be
\overline{h}_1(t) = \lambda \xi (t) (t-T)^2,\;\;\;  \overline{l}_1(t) = \overline{h}_1'(t),\qquad t\in\R.
\label{P1000}
\ee
Note that 
\ba
&&\overline{h}_1(0)=\overline{h}_1(T)=\overline{l}_1(0)=\overline{l}_1(T)=0, \label{P2}\\
&& \overline{l}_1'(T) = 2\lambda \ne 0, \quad \overline{l}_1^{(k)} (T)=0\ \text{ for } k\ge 2. \label{P3}
\ea
Next, define $\overline{w}_1$ as the solution to the Cauchy problem
\begin{eqnarray}
\D\dot{ \overline{w}}_1&=&\D\alpha^{-1}(\dot{\overline{l}}_1-\beta\overline{l}_1\overline{w}_1-\gamma\overline{w}_1^2), \label{X1}\\ 
\overline{w}_1(0)&=&0.\label{X2}
\end{eqnarray}
By a classical result on the continuous dependence of solutions of ODE's with respect to a parameter, we have that the 
solution $\overline{w}_1$ of \eqref{X1}-\eqref{X2} is defined on $[0, T]$ provided that $\lambda_0$ is small enough. 
Set  $\overline{h}=(\overline{h}_{1},0,0)$, $\overline{\vec q}=(0,0,0)$, $\overline{w}=(\overline{w}_{1},0,...,0)$, $\overline{l}=(\overline{l}_{1},0,0)$ and  
$\overline{r}=(0,0,0)$. According to Lemma  \ref{lemma simple},  $(\overline{h},\overline{\vec q},\overline{l},\overline{r},\overline{w})$ is a solution 
of (\ref{systempq}), which satisfies  
$$(\overline{h},\overline{\vec q},\overline{l},\overline{r})(0) = 0 = (\overline{h},\overline{\vec q},\overline{l},\overline{r})(T).$$

\subsubsection{Linearization along the reference trajectory}
Writing  
\beq\begin{array}{ccc}
h&=&\overline{h}+ {\hat h},\\ 
\vec q&=&\overline{\vec q}+ \widehat{\vec q},\\ 
l&=&\overline{l}+ {\hat l},\\ 
r&=&\overline{r}+ {\hat r},
\end{array}
\eeq
expanding in (\ref{systempq}) in keeping only the first order terms in $\hat h,\widehat{\vec q},\hat l$ and
$\hat r$, we obtain 
the following linear system
\beq\label{second lineal system}
\left\{\begin{array}{ccl}
{\hat h}'&=& {\hat l } + 2\widehat{\vec q}  \times \overline{l} , \\[2mm]
{\widehat{\vec q}\,} '&=&\frac{1}{2} \hat{r}, \\[2mm]
\left(\begin{array}{c}
\hat{l}\\
\hat{r}
\end{array}\right)'&=&\mathcal{J} ^{-1}\left( 
A(t) \left( \begin{array}{c} \hat l \\ \hat r \end{array}  \right) 
+ B(t) \hat w +  
C \hat{w}' \right),
\end{array}\right.\eeq
where the matrices  $A(t) \in \R ^{6\times 6}$ and  $B(t)\in \R ^{6\times m}$ are defined as 
\begin{eqnarray}
A(t) &=& \left( \frac{\partial F}{\partial l}(\overline{l}(t),\overline{r}(t),\overline{w}(t)) \quad  \big\vert\quad  
 \frac{\partial F}{\partial r}(\overline{l}(t),\overline{r}(t),\overline{w}(t)) \right) ,\\ 
B(t)&=&\D\frac{\partial F}{\partial w} (\overline{l}(t),\overline{r}(t),\overline{w}(t)).
\end{eqnarray}
Setting 
\beq
{\hat p} = 2\widehat{\vec q},
\eeq
we can rewrite \eqref{second lineal system} as
\beq\label{third lineal system}
\left\{\begin{array}{ccl}
{\hat h}'&=& {\hat l}-\overline{l} \times \hat{p},   \\[2mm]
{\hat p}'&=&\hat{r}, \\[2mm]
\left(\begin{array}{c}
\hat{l}\\
\hat{r}
\end{array}\right)'&=&\mathcal{J} ^{-1}\left( 
A(t) \left( \begin{array}{c} \hat l \\ \hat r \end{array}  \right) 
+ B(t) \hat w +  
C \hat{w}' \right) .
\end{array}\right.\eeq
Obviously, \eqref{second lineal system} is controllable on  $[0,T]$ if, and only if, \eqref{third lineal system} is.
Letting 
\[
 z =\left( \begin{array}{c} \hat h\\ \hat p \end{array} \right), \quad 
k =\left( \begin{array}{c} \hat l\\ \hat r \end{array} \right),\quad  f=\hat w,
\]
we obtain the following control system
\ba
\left( \begin{array}{c} \dot z \\ \dot k \end{array} \right)
&=& 
\left( 
\begin{array}{cc}D(t) & Id \\ 0 & {\mathcal J}^{-1} A(t) \end{array}
\right) 
\left( \begin{array}{c} z \\ k  \end{array} \right)
 + 
 \left( \begin{array}{c} 0 \\ {\mathcal J }^{-1} B(t)  \end{array} \right) f
 + 
 \left( \begin{array}{c} 0 \\ {\mathcal J }^{-1} C  \end{array} \right) \dot f \nonumber \\
 &=:& {\mathcal A} (t)  \left( \begin{array}{c} z \\ k  \end{array} \right)
 +{\mathcal B} (t)  f + {\mathcal C } \dot f.
\label{B251}
\ea

We find that 
\[ D = \left( \begin{array}{cc} 0 & -S(\overline{l})\\0&0 \end{array} \right),\quad \text{ with } 
S(\overline{l})= \left( \begin{array}{ccc} 0&0&0\\ 0&0& -\overline{l}_1 \\ 0 & \overline{l}_1 &0 \end{array} \right) ,\] 
\begin{multline*}
B = \left(\begin{array}{cc} 0&0\\[2mm] S(\overline{l}) &0  \end{array} \right) C 
-\overline{w}_1 \left( \begin{array}{c} W_1^M\\[2mm] W_1^J \end{array} \right)  
- \overline{l}_1 \left( \begin{array}{c|c|c|c} 
L_1^M e_1 & L_2^M e_1 &\cdots   & L_m^M e_1  \\[2mm] 
L_1^J  e_1 & L_2^J  e_1  &\cdots & L_m^J e_1
 \end{array} \right)\, \\
 - \overline{w}_1 \left( \begin{array}{c|c|c|c} 
W_1^M e_1 & W_2^M e_1 &\cdots   & W_m^M e_1  \\[2mm] 
W_1^J  e_1 & W_2^J  e_1  &\cdots & W_m^J e_1
 \end{array} \right)\, ,
\end{multline*} 
and that 
\[
A = 
\left( 
\begin{array}{cccccc}
-(L_1^M)_{11}\overline{w}_1 & 0 & 0 & 0 & 0 & 0\\
 0 & -(L_1^M)_{22}\overline{w}_1 & 0 & 0 & 0 & A_{26}\\
 0 & 0 & -(L_1^M)_{33}\overline{w}_1 & 0 & A_{35} & 0 \\
 0 & 0 & 0 &   -(R_1^J)_{11}\overline{w}_1 & 0 & 0 \\
 0 & 0 & A _{53}  & 0 & N_{32}\overline{l} _1 -(R_1^J)_{22}\overline{w}_1 & 0\\
 0 & A_{62} & 0 & 0 & 0 & -N_{23} \overline{l} _1   -(R_1^J)_{33}\overline{w}_1
\end{array}
\right) 
\]
with 
\begin{eqnarray*}
A_{26} &=& -(m_0+M_{11}) \overline{l_1} -\big( (C^M)_{11} +(R^M_1)_{23}\big) \overline{w}_1, \\
A_{35} &=&  (m_0+M_{11}) \overline{l_1}  + \big( (C^M)_{11} -(R^M_1)_{32}\big) \overline{w}_1, \\
A_{53} &=&  (M_{33} - M_{11}) \overline{l_1} -\big( (C^M)_{11} +(L^J_1)_{23}\big) \overline{w}_1, \\
A_{62} &=&  (M_{11} - M_{22}) \overline{l_1}  + \big( (C^M)_{11}  - (L^J_1)_{32}\big) \overline{w}_1.
\end{eqnarray*}

From now on, we suppose in addition to \eqref{cond. symmetric domain}-\eqref{cond. symmetric first control}
that $\chi_1$ is chosen so that
\beq 
\label{cond. alpha}
\alpha\neq 0.
\eeq

\subsubsection{Linear control systems with one derivative in the control} 
Let us consider any linear control system of the form
\be
\label{XX1}
\dot x = {\mathcal A } (t) x + {\mathcal B } (t) u + C \dot u
\ee
where $x\in \R ^n$ is the state ($n\ge 1$), $u\in \R ^m$ is the control input ($m\ge 1$), ${\mathcal A} \in C^\infty ([0,T];\R ^{n\times n })$,
${\mathcal B}\in C^\infty ([0,T]; \R ^{n\times m})$, and ${\mathcal C} \in \R ^{n\times m}$.   
Define a sequence of matrices ${\mathcal M_i}(t)\in \R ^{n\times m}$ by 
\be
\label{N1}
{\mathcal M}_0(t)={\mathcal B}(t) + {\mathcal A } (t){\mathcal C},
\quad \text{ and }\  {\mathcal M}_i(t) = \dot {\mathcal M}_{i-1} (t) - {\mathcal A}(t)  {\mathcal M} _{i-1} (t), \qquad \forall i\ge 1, \ \forall t\in [0,T].  
\ee
Introduce the reachable set
\begin{multline*}
{\mathcal R }_{u(0)=0} =\{ x_T\in \R ^n; \exists u\in H^1(0,T;\R ^m)\ \text { with }\   u(0)=0 \text { such that  }\\
x_T=x(T), \ \text{ where }\  x(\cdot )\  \text{ solves  \eqref{XX1}  } \text{ and } x(0) =0\}.   
\end{multline*} 
Then the following result holds.
\begin{proposition}
\label{prop1000}
\cite[Propositions 2.4 and 2.5]{GR} 
Let $\varepsilon >0$, ${\mathcal A} \in C^{\omega} ((-\varepsilon , T + \varepsilon ) ; \R ^{n\times n } )$ and 
${\mathcal B} \in C^{\omega} ((-\varepsilon , T + \varepsilon ) ; \R ^{n\times m } )$, and let $({\mathcal M}_i)_{i\ge 0}$
be the sequence defined in \eqref{N1}. Then for all $t_0\in [0,T]$, we have  that
\be
\label{reachable}
{\mathcal R}_{u(0)=0} = {\mathcal C } \R ^m + \text{Span} \{ \phi (T,t_0)\,  {\mathcal M}_i (t_0) u; \ u\in \R ^m, i\ge 0 \},
\ee 
where $\phi$ denotes the fundamental solution associated with the system $\dot x={\mathcal A} (t)x$. 
\end{proposition}
Recall that the fundamental solution associated with  $\dot x={\mathcal A} (t)x$ is defined as the solution to
\begin{eqnarray*}
\frac{\partial \phi } {\partial t } &=& {\mathcal A} (t) \phi (t,s), \\
\phi (s,s) &=& Id. 
\end{eqnarray*} 
For notational convenience, we introduce the matrices
\be
\label{N2}
\hat A (t) ={\mathcal J }^{-1}A(t) ,\quad \hat B (t) = {\mathcal J} ^{-1} B (t), \quad \hat C = {\mathcal J} ^{-1} C, \quad
{\mathcal M} _i(t) = \left( \begin{array}{c} U_i(t)\\ V_i(t) \end{array} \right) ,   
\ee
where $\hat A (t) \in \R ^{6\times 6}$ , $\hat B(t), \hat C , U_i(t),V_i(t)\in \R ^{6\times m}$. Then 
\be
\label{N3}
 \left( \begin{array}{c} U_0(t)\\ V_0(t) \end{array} \right)  = 
  \left( \begin{array}{c} \hat C \\ \hat B(t) + \hat A (t) \hat C  \end{array} \right) , 
\ee
while 
\be
\label{N4}
 \left( \begin{array}{c} U_i(t)\\ V_i(t) \end{array} \right)  = 
  \left( \begin{array}{c} U_{i-1}'(t)  -D(t) U_{i-1} (t) -V_{i-1} (t)  \\ V_{i-1}'(t) - \hat A (t) V_{i-1}(t) \end{array} \right) .
\ee
In certain situations, half of the terms $U_i(t)$ and $V_i(t)$ vanish at $t=T$. The following result, whose proof is given in Appendix, 
will be used thereafter. 
\begin{proposition}
\label{prop100}
If $\hat C\in \R ^{6\times m}$ is given and $\hat A,D$ (resp. $\hat B$) denote some functions in $C^\infty ([0,T]; \R ^{6\times 6} )$ 
(resp. in $C^\infty ([0,T];\R ^{6\times m}$) fulfilling 
\be
\label{N5}
\hat A ^{(2l)} (T) = D ^{(2l)} (T) =0 \ \textrm{ and } \  \hat B ^{(2l)} (T) =0 \qquad \forall l\in \N ,
\ee 
then the sequences $(U_i)_{i\ge 0}$ and $(V_i)_{i\ge 0} $ defined in \eqref{N3}-\eqref{N4} satisfy
\ba
&&V_{2k}^{(2l)} (T) = V_{2k+1}^{(2l+1)} (T) =0, \qquad \forall k,l\in\N , \label{N6}\\
&&U_{2k+1}^{(2l)} (T) = U_{2k} ^{(2l+1 )} (T ) =0, \qquad \forall k,l\in \N . \label{N7}
\ea  
\end{proposition}

The following result, which is one of the main results in this paper,  shows that under suitable assumptions the local controllability of  (\ref{systempq})
holds with less than six control inputs.
\begin{theorem} 
\label{thm1}
Assume that  \eqref{cond. symmetric domain}, \eqref{cond. symmetric first control} and \eqref{cond. alpha} hold. Pick any $T>0$. 
If the rank condition
\beq \label{cond1}
\textrm{rank }({\mathcal C},{\mathcal M}_0(T), {\mathcal M}_1(T), {\mathcal M}_2(T), ...) =12
\eeq
holds, then the system (\ref{systempq}) with state 
$(h,\vec q,l,r) \in \R^{12}$ and control $w \in \R ^m$ is locally controllable
 around the origin in time $T$. We can also impose that the control input 
 $w \in H^2(0, T; \R ^m)$ satisfies $w(0) = 0$. Moreover, for some
 $\eta>0$,	 there is a	 $C^1$ map from $B_{\R^{24}} (0, \eta)$	to $H^2(0, T; \R^m)$, which associates with
 $(h_0,{\vec q}_0,l_0,r_0,h_T,{\vec q}_T,l_T,r_T )$ a control satisfying $w(0) = 0$ and steering the state of the 
 system from $(h_0,{\vec q}_0,l_0,r_0)$ at $t=0$ to $(h_T,{\vec q}_T,l_T,r_T)$ at $t=T$.
\end{theorem}
\begin{proof}
\noindent
{\bf Step 1: Controllability of the linearized system.}\\ 
Letting $t_0=T$ in Proposition \ref{prop1000} yields
\begin{equation*}
{\mathcal R}_{f(0)=0} = {\mathcal C} \R ^m + \sum _{i\ge 0} {\mathcal M}_i(T)\R ^m. 
\end{equation*}
Thus, if the condition \eqref{cond1} is fulfilled, we infer that ${\mathcal R}_{f(0)=0} = \R ^{12}$, i.e. the system  
\eqref{third lineal system} is controllable. The same is true for \eqref{second lineal system}. 

\noindent
{\bf  Step 2: Local controllability of the nonlinear system.}\\
 Let us introduce the Hilbert space
$$\mathcal{H}:=\R^{12}\times\{f\in H^2(0,T;\R^m);\; f(0)=0 \}$$
endowed with its natural Hilbertian norm
$$\norm{(x,f)}^2_\mathcal{H}=||x|| ^2_{\R^{12}}+\norm{f}^2_{H^2(0,T)}.$$
We denote by $B_\mathcal{H}(0,\delta)$ the open ball in $\mathcal{H}$ with center $0$ and radius $\delta$, i.e.
$$B_\mathcal{H}(0,\delta) = \{(x,f) \in \mathcal{H}; \norm{(x,f)}_\mathcal{H} <\delta\}.$$
Let us introduce the map
$$
\begin{array}{rcl}
\Gamma:B_\mathcal{H}(0,\delta)&\to&\R^{24}\\
((h_0,{\vec q}_0,l_0,r_0),f)&\mapsto&(h_0,{\vec q}_0,l_0,r_0,h(T),{\vec q}(T),l(T),r(T)),
\end{array}
$$
where $(h(t),\vec q\, (t)  ,l(t),r(t))$ denotes the solution of
\beq
\left\{
\begin{array}{ccl}
h'&=&  (1- || \vec q\, ||^2) l+ 2\sqrt{1-||\vec q\, || ^2}\,  \vec q\times l  + (l\cdot \vec q \, ) \vec q - \vec q \times l\times \vec q, \\[3mm]
{\vec q\,} '&=& \frac{1}{2} (\sqrt{1- ||\vec q\, ||^2} \, r + \vec q \times r ), \\[3mm]
\left(\begin{array}{c}
l\\
r
\end{array}\right)'&=&\mathcal{J} ^{-1}(C( \overline{w}'+f')+F(l,r,\overline{w}+f)),\\[3mm]
(h(0),{\vec q}(0),l(0),r(0))&=&(h_0,{\vec q}_0,l_0,r_0).
\end{array}
\right.
\eeq
Note that $\Gamma$ is well defined for $\delta > 0$ small enough (provided that $\lambda _0$ has been taken sufficiently small). 
Using the Sobolev embedding $H^2(0,T;\R ^m ) \subset C^1([0,T];\R ^m )$, 
we can prove as in \cite[Theorem 1]{sontag-book} that $\Gamma$ is of class $C^1$ on $B_\mathcal{H}(0,\delta)$ and that 
its tangent linear map at the origin is given by
$$d\Gamma(0)((\hat{h}_0,\widehat{\vec q}_0,\hat{l}_0,\hat{r}_0),f)=(\hat{h}_0,\widehat{\vec q}_0,\hat{l}_0,\hat{r}_0,\hat{h}(T),\widehat{\vec q\,} (T),\hat{l}(T),\hat{r}(T)),$$
where $(\hat{h}(t),\widehat{\vec q\, }(t),\hat{l}(t),\hat{r}(t))$ solves the system
(\ref{second lineal system}) with  the initial conditions 
$$(\hat{h}(0),\widehat{\vec q}(0),\hat{l}(0),\hat{r}(0))=(\hat{h}_0,\widehat{\vec q}_0,\hat{l}_0,\hat{r}_0).$$
We know from Step 2 that (\ref{second lineal system}) is controllable, so that $d\Gamma(0)$ is onto.
 Let $V := (\textrm{ker } d \Gamma(0))^\perp$ 
denote the orthogonal complement	of $\textrm{ker	}d \Gamma(0)$  in $\mathcal{H}$.
Then $d \Gamma (0) |_V$ is invertible, and therefore it follows from the inverse function theorem that the map $\Gamma|_V : V \to \R^{24}$
 is locally invertible at the origin. More precisely, there exists a number
 $\delta>0$ and an open set $\omega  \subset \R^{24}$
  containing $0,$ such that the map $\Gamma: B_\mathcal{H}(0,\delta)\cap V \to \omega$ is well-defined, of class $C^1$, invertible, 
  and with an inverse map of class $C^1$. Let us denote this inverse map by $\Gamma^{-1}$, and let us write
   $\Gamma^{-1}(x_0,x_T) = (x_0,f(x_0,x_T))$. Finally, let us set $w = \overline{w} + f$. Then, for $\eta > 0$ small enough, we have that 
\beq
w\in C^1(B_{\R^{24}}(0,\eta),H^2(0,T; \R ^m)),
\eeq
and that for $||(h_0,\vec q_0,l_0,r_0,h_T,\vec q_T, l_T,r_T)|| _{\R ^{24}} <\eta$,  the solution $(h(t),{\vec q}(t),l(t),r(t))$ of system (\ref{systempq}), with the initial conditions 
$$(h(0),{\vec q}(0),l(0),r(0))=(h_0,{\vec q}_0,l_0,r_0),$$
satisfies 
$$(h(T),{\vec q}(T),l(T),r(T))=(h_T,{\vec q}_T,l_T,r_T).$$
The proof of Theorem \ref{thm1} is complete.
\end{proof}
We now derive two corollaries of Theorem \ref{thm1}, that will be used in the next section. 
We introduce the matrices
\beq
{\bf A}=\Big(A_L \Big| A_R\Big),
\eeq
where
\beq
A_L=\left(\begin{array}{ccc}
-(L^M_{1})_{11} & 0& 0\\
0 &-(L^M_{1})_{22} & 0 \\
0 &0 &-(L^M_{1})_{33}  \\
0 & 0& 0\\
0 &0 & \alpha(M_{33}-M_{11})-((L^J_1)_{23}+(C^{M})_{11}) \\
0 &\alpha(M_{11}-M_{22})-((L^J_{1})_{32}-(C^{M})_{11}) &0  
\end{array}\right),
\eeq
\beq
A_R=\left(\begin{array}{ccc}
0 & 0& 0\\
0 &0 & -(R^M_{1})_{23}\\
0 &-(R^M_{1})_{32} &0 \\
-(R^J_{1})_{11} & 0& 0\\
0 &\alpha N_{32}-(R^J_{1})_{22} & 0 \\
0 &0 & -\alpha N_{23}-(R^J_{1})_{33} 
\end{array}\right),
\eeq
\beq
\begin{array}{c}
\begin{array}{rcl}
{\bf B}&=& \left(\begin{array}{c}0 \\ \\
\D -\alpha S( e_1) C^{M} 
\end{array}\right)
-\alpha\left(\begin{array}{c|c|c|c}L_1^M e_{1}& L_2^M e_{1}&\cdot\cdot\cdot &L_m^M e_{1}\\ & & &\\ 
\D L_1^J e_{1} & L^J_2 e_{1}&\cdot\cdot\cdot &L_m^J e_{1}
\end{array}\right)-
\left(\begin{array}{c} \D  W^M_{1} \\ \\
W^J_{1} 
\end{array}\right)
\end{array}\\ \\
-\left(\begin{array}{c|c|c|c} \D  W_1^M e_{1}&W_2^M e_{1} &\cdot\cdot\cdot &W_m^M e_{1} \\ & & &\\ 
W^J_1 e_{1} &W^J_2 e_{1}&\cdot\cdot\cdot &W^J_m e_{1}
\end{array}\right),
\end{array}
\eeq
and 
\beq
{\bf D}=\left(\begin{array}{cc}
0 &-\alpha S(  e_1 ) \\
0 &0  
\end{array}\right).
\eeq

The first corollary  will be used later to derive a controllability result  with only four control inputs.
\begin{corollary}
\label{cor1}
If both rank conditions
\beq \label{cond rank 1}
\textrm{rank }(C, {\bf B} + {\bf A}\mathcal{J}^{-1}C) = 6
\eeq
and
\beq \label{cond rank 2}
\textrm{rank }(C, \frac{1}{2}\mathcal{J}{\bf D}\mathcal{J}^{ -1}C + {\bf B} + {\bf A}\mathcal{J}^{ -1}C) = 6
\eeq
are fulfilled, then the condition \eqref{cond1} is satisfied for any $T>0$, so that the conclusion of Theorem \ref{thm1} is valid
for any $T>0$.
\end{corollary}
\begin{proof}
We distinguish two cases.\\
{\sc Case 1: $\gamma + \alpha \beta =0$.} \\
We begin with the ``simplest'' case when $\gamma + \alpha \beta =0$. Pick any $T>0$ and let
$\overline{l}_1,\overline{w}_1$ be as in \eqref{P1000} and \eqref{X1}-\eqref{X2}. 
 Let $\overline{g}_1:=\overline{l}_1 -\alpha\overline{w}_1$. It is clear that 
$\dot {\overline{g}}_1 = \beta \overline{w}_1\overline{g}_1$, hence $\overline{g}_1\equiv 0$. We infer that 
\begin{eqnarray*}
&&\overline{w}_1^{(k)}(T) = \alpha ^{-1} \overline{l}_1^{(k)}(T)=0\quad \text{ for }\  k\in \N \setminus \{1\}, \\
&&\overline{w}_1'(T) = \alpha ^{-1} \overline{l}_1' (T)=2\lambda/\alpha \ne 0.
\end{eqnarray*}
It follows that 
\begin{eqnarray}
&&A^{(k)}(T) =0,\quad  B^{(k)}(T)=0,\quad  D^{(k)} (T) =0\qquad   \text{ for }\  k\in \N \setminus \{1\}, \label{P4} \\
&&A'(T) = \overline{w}_1'(T){\bf A} ,\quad  B'(T)=\overline{w}_1'(T){\bf B} ,\quad  D'(T) =\overline{w}_1'(T){\bf D}. \label{P4000}
\end{eqnarray}
Applying Proposition \ref{prop100}, we infer that 
\begin{multline*}
\textrm{rank } \big( {\mathcal C} , {\mathcal M}_0(T),  {\mathcal M}_1 (T),  {\mathcal M}_2(T) \big)  \\
=\textrm{rank } \big( 
\left(\begin{array}{c} 0\\ \hat C \end{array} \right), 
 \left(\begin{array}{c} \hat C\\ \hat B(T) + \hat A(T) \hat C  \end{array} \right), 
 \left(\begin{array}{c} 0\\ V_1(T) \end{array} \right), 
 \left(\begin{array}{c} U_2(T) \\  0 \end{array} \right) 
\big) .
\end{multline*}
On the other hand, it is easily seen that 
\begin{eqnarray*}
&&V_1(T) = V_0'(T) = {\mathcal J}^{-1}B'(T) +  {\mathcal J}^{-1} A'(T) {\mathcal J}^{-1} C 
=  \overline{w}_1'(T) \left(  {\mathcal J}^{-1}{\bf B}  +  {\mathcal J}^{-1} {\bf A} {\mathcal J}^{-1} C \right) ,\\
&&U_2(T) = -D'(T) U_0(T) -2V_0'(T) =-\overline{w}_1'(T) [ {\bf D} {\mathcal J}^{-1} C + 2{\mathcal J} ^{-1} ({\bf B} + {\bf A} {\mathcal J } ^{-1} C) ].
\end{eqnarray*}
It follows that
\begin{eqnarray*}
&&\textrm{rank}(\hat C, V_1(T))= \textrm{rank }(C, {\bf B} + {\bf A}\mathcal{J}^{-1}C) = 6,\\
&&\textrm{rank}(\hat C, U_2(T)) = \textrm{rank }(C, \frac{1}{2}\mathcal{J}{\bf D}\mathcal{J}^{ -1}C + {\bf B} + {\bf A}\mathcal{J}^{ -1}C) = 6,
\end{eqnarray*}
and 
\[
\textrm{rank } \big( {\mathcal C} , {\mathcal M}_0(T),  {\mathcal M}_1 (T),  {\mathcal M}_2(T) \big) =12.
\]
Thus \eqref{cond1} is satisfied, as desired.\\
{\sc Case 2. $\gamma + \alpha\beta \ne 0$.} 
We claim that for $T>0$ arbitrary chosen and $\lambda_0$ small enough, we have for 
$0<\lambda <\lambda _0$,
\[
\text{rank } (\Cs, \Ms _0(T), \Ms _1(T), \Ms _2 (T))=12.
\]
First, $\|\overline{l}_1\|_{W^{2,\infty}(0,T)}=O(\lambda)$ still with
$\overline{l}_1(T)=\ddot{\overline{l}}_1(T)=0$. From \eqref{X1}-\eqref{X2}, we
infer with Gronwall lemma (for $\lambda _0$ small enough) that 
$\overline{w}_1$ is well defined on $[0,T]$ and that  
$\|\overline{w}_1\|_{L^\infty (0,T)}=O(\lambda )$. This also yields
(with \eqref{X1}) $||\overline{w}_1||_{W^{2,\infty}(0,T)}=O(\lambda )$. 
Next, integrating
in \eqref{X1} over $(0,T)$ yields $\overline{w}_1(T)=O(\lambda ^2)$. Finally,
derivating in \eqref{X1} gives $\ddot{\overline{w}}_1(T)=O(\lambda ^2)$.
We conclude that 
\[
(A(T),B(T),\ddot A(T),\ddot B(T))=O(\lambda ^2),\quad D(T)=0,
\]
while 
\[
(\dot A(T),\dot B(T),\dot D(T))=  (2 \lambda / \alpha )   ({\mathbf A}, {\mathbf B},{\mathbf D})
+O(\lambda ^4),
\]
for  $\dot{\overline{l}}_1(T)=\alpha \dot{\overline{w}}_1(T)  + O(\lambda ^4)$. It follows that 
\begin{align*}
\text{\rm rank} \,& (\Cs,\Ms _0(T),\Ms _1(T),\Ms _2(T)) \quad\\
&=\text{\rm rank } \Bigg[
\left( \begin{array}{c}
0\\
{\Js} ^{-1} C
\end{array} \right) , 
\left( \begin{array}{c}
{\Js} ^{-1} C\\
0
\end{array} \right) ,\\
&\hskip 1cm
\left( \begin{array}{c}
0\\
{\Js} ^{-1} ({\mathbf B} + {\mathbf A}{\Js}^{-1}C) 
\end{array} \right) , 
\left( \begin{array}{c}
{\Js} ^{-1} [\Js {\mathbf D}{\Js}^{-1}C + 2 ({\mathbf B} + {\mathbf A}{\Js}^{-1}C)]\\
0
\end{array} \right) \Bigg] \\
&\quad =12,
\end{align*}
for $0< \lambda  <\lambda _0$ with $\lambda _0$ small enough, as desired.
\end{proof}

The second one is based on the explicit computations of ${\mathcal M}_i(T)$ for $i\le 8$. It will be used later to derive a controllability 
result with only three controls inputs.
\begin{corollary}
\label{cor2}
Let ${\bf E}:= {\bf B} + {\bf A} {\mathcal J}^{-1} C$. 
If the conditions
\beq \label{newcond1}
\textrm{rank }(C, {\bf E}, {\bf A}{\mathcal J}^{-1} {\bf E},({\bf A}{\mathcal J}^{-1})^2 {\bf E},({\bf A}{\mathcal J }^{-1})^3 {\bf E} ) = 6,
\eeq
and
\begin{multline}
\label{newcond2}
\textrm{rank }(C, \frac{1}{2}\mathcal{J}{\bf D}\mathcal{J}^{ -1}C + {\bf E}, 
({\mathcal J} {\bf D} {\mathcal J}^{-1} + 2{\bf A} {\mathcal J} ^{-1} ) {\bf E} ,\\ 
(8{\mathcal J} {\bf D} {\mathcal J}^{-1} + 11{\bf A} {\mathcal J} ^{-1} ) {\bf A}{\mathcal J }^{-1} {\bf E} , 
(17{\mathcal J} {\bf D} {\mathcal J}^{-1} + 64{\bf A} {\mathcal J} ^{-1} ) ( {\bf A} {\mathcal J }^{-1})^2 {\bf E}
) = 6,
\end{multline}
are fulfilled, then the condition \eqref{cond1} is satisfied, so that the conclusion of Theorem \ref{thm1} is valid.
\end{corollary}
\begin{proof}
The proof is almost the same as those of  Corollary \ref{cor1}, the only difference being that we
need now to compute ${\mathcal M}_i(T)$ for $i\le 8$. In view of Proposition  \ref{prop100}, it is sufficient in Case 1 ($\gamma + \alpha \beta =0$) to compute 
$V_i(T)$ for $i\in \{ 1,3,5,7\}$ and $U_i(T)$ for $i\in \{2,4,6,8\}$. The results are displayed in two propositions, whose proofs
are given in Appendix.
\begin{proposition}
\label{prop200}
Assume that the pair $(\overline{h}_1, \overline{l}_1)$ is as in \eqref{P1000},  that $\overline{w}_1$ is as in 
\eqref{X1}-\eqref{X2}, and that $\gamma + \alpha \beta =0$. Then we have 
\ba
V_1(T) &=& V_0'(T), \label{P5}\\
V_3(T) &=& -3 \hat A '(T) V_0'(T), \label{P6} \\
V_5(T) &=& 15 \hat A'(T) ^2 V_0'(T), \label{P7}\\
V_7(T) &=& -105 \hat A'(T) ^3 V_0'(T). \label{P8}  
\ea
\end{proposition}

\begin{proposition}
\label{prop300}
Assume that the pair $(\overline{h}_1, \overline{l}_1)$ is as in \eqref{P1000},  that $\overline{w}_1$ is as in 
\eqref{X1}-\eqref{X2}, and that that $\gamma + \alpha \beta =0$. Then we have  
\ba
U_2(T) &=& -D'(T) U_0(T) -2V_0'(T),  \label{P21}\\
U_4(T) &=&  4 \big( D'(T) +  2 \hat A '(T) \big) V_0'(T), \label{P22}\\
U_6(T) &=& -3 (8 D'(T) + 11 \hat A'(T) ) \hat A'(T) V_0'(T), \label{P23} \\
U_8(T) &=& 6\big( 17 D'(T) + 64 \hat A '(T) \big) \hat A'(T) ^2 V_0'(T). \label{P24} 
\ea 
\end{proposition} 
\end{proof}
\section{Examples}
This section is devoted to examples of vehicles with ``quite simple'' shapes, for which the coefficients in the matrices in 
\eqref{matrix 1}-\eqref{matrix 6} can be computed explicitly.
We begin with the case of a vehicle with one axis of revolution, for which the controllability fails for any choice of the flow controls.   
\subsection{Solid of revolution}  
Let $f\in\mathcal{C}^{1}([a,b]; \R)$ be a nonnegative function such that $f(a)=f(b)=0$, and let 
$$\S =\Big\{ \Big(y_1, s f(y_1)\cos \theta, s f(y_1)\sin \theta \Big); \ y_1\in [a,b],\; s \in [0,1], \; \theta\in[0,2\pi ]\Big\} .$$
In other words, $\S$ is a solid of revolution (see Figure \ref{fig01:revolution}).

Assume that the density $\rho $ depends on $y_1$ only, i.e. $\rho =\rho (y_1)$. 
Clearly  
$J_{0}=\textrm{diag}(J_{1},J_{2},J_{2})$. On the other hand, 
$$\partial \Omega=\Big\{ \Big(y_1,f(y_1)\cos \theta ,f(y_1)\sin \theta \Big) ; \; y_1\in [a,b],\; \theta\in[0,2\pi ]\Big\},$$
and the normal vector $\n$ to $\partial \Omega$ is given by 
$$\n(y_1 ,\theta)=\frac{1}{\sqrt{1+(f'(y_1 ))^{2}}}\Big(f'(y_1 ), - \cos \theta , - \sin \theta \Big)^\ast,$$
so that $$(y\times \n)(y_1,\theta)=\frac{(y_1 +f(y_1 )f'(y_1 ))}{\sqrt{1+(f'(y_1 ))^{2}}}\Big(0, \sin \theta ,- \cos \theta \Big)^\ast .$$
It follows that $(y\times \n) \cdot e_{1}=0$. Replacing in (\ref{eq. of r}), we obtain  
$$J_{1}\dot r_{1}=(J_0 \dot r) \cdot e_{1}=-(r\times J_0r) \cdot e_{1}=J_{2}r_{2}r_{3}-J_{2}r_{2}r_{3}=0,$$
which indicates that the angular velocity $r_{1}$ is not controllable.

\begin{figure}[h] 
\begin{center}
\includegraphics[width=.5\textwidth]
{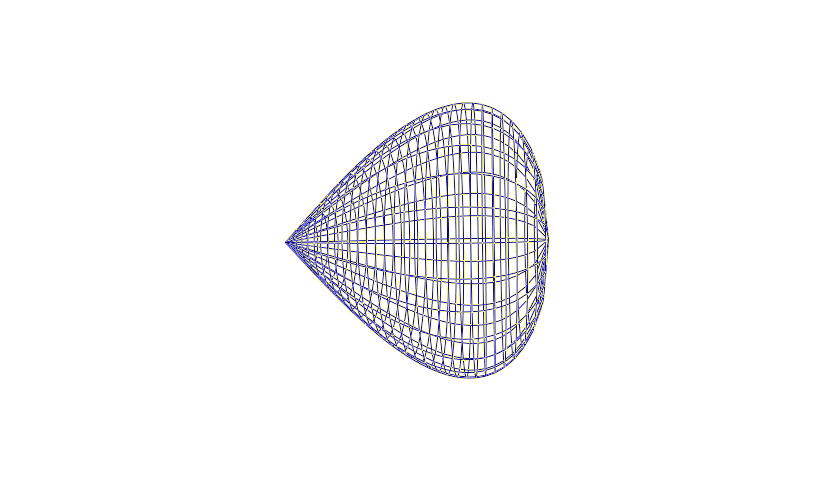} 
\caption{A solid of revolution.}
\label{fig01:revolution}
\end{center}
\end{figure}

\subsection{Ellipsoidal vehicle.}

We assume here that the vehicle fills the ellipsoid
\be
\S =\Big\{ y\in \R^3;\ \  (y_1/c_1)^2 + (y_2/c_2)^2 +(y_3/c_3)^2 \le 1\}
\label{WW1}
\ee
where  $c_1>c_2>c_3>0$ denote some numbers. Our first aim is to compute explicitly the functions $\phi _i$ and $\varphi _i$ for $i=1,2,3$, which solve
\eqref{eq for phi}-\eqref{eq for phi:02} for
$$\Omega =\Big\{ y\in \R^3;\ \  (y_1/c_1)^2 + (y_2/c_2)^2 +(y_3/c_3)^2 > 1\} .$$

\subsubsection{Computations of the functions $\phi_i$ and $\varphi_i$.}

We follow closely \cite[pp.148-155]{Lamb}. We introduce a special system of orthogonal curvilinear coordinates, denoted by $(\lambda ,\mu ,\nu )$, which
are defined as the roots of the equation 
\be
\label{cubic}
\frac{y_1^{2}}{c_1^{2}+\theta}+\frac{y_2^{2}}{c_2^{2}+\theta}+\frac{y_3^{2}}{c_3^{2}+\theta}-1=0
\ee
viewed as a cubic in $\theta$. It is clear that \eqref{cubic} has three real roots:
$\lambda \in(-c_3^{2},+\infty)$, $\mu \in (-c_2^{2},-c_3^{2})$, and  
$\nu \in (-c_1^{2},-c_2^{2})$. 

It follows immediately from the above definition of $\lambda,\mu,\nu$, that 
$$\frac{y_1^{2}}{c_1^{2}+\theta}+\frac{y_2^{2}}{c_2^{2}+\theta}+\frac{y_3^{2}}{c_3^{2}+\theta}-1
=\frac{(\lambda-\theta)(\mu-\theta)(\nu-\theta)}{(c_1^{2}+\theta)(c_2^{2}+\theta)(c_3^{2}+\theta)}.$$ 
This yields
\beq\begin{array}{ccccc}
y_1^{2}&=&\D\frac{(c_1^{2}+\lambda)(c_1^{2}+\mu)(c_1^{2}+\nu)}{(c_2^{2}-c_1^2)(c_3^2-c_1^2)},
& \;\;\;&\D\partial_\lambda y_1=\D\frac{1}{2}\frac{y_1}{(c_1^2+\lambda)},\\ \\
y_2^{2}&=&\D\frac{(c_2^{2}+\lambda)(c_2^{2}+\mu)(c_2^{2}+\nu)}{(c_1^{2}-c_2^2)(c_3^2-c_2^2)},
& \;\;\;&\D\partial_\lambda y_2=\D\frac{1}{2}\frac{y_2}{(c_2^2+\lambda)},\\ \\
y_3^{2}&=&\D\frac{(c_3^{2}+\lambda)(c_3^{2}+\mu)(c_3^{2}+\nu)}{(c_1^{2}-c_3^2)(c_2^2-c_3^2)},
& \;\;\;&\D\partial_\lambda y_3=\D\frac{1}{2}\frac{y_3}{(c_3^2+\lambda)}.
\end{array}\eeq
We introduce the scale factors
\beq\begin{array}{ccc}
h_{\lambda}&=&\D\frac{1}{2}\sqrt{\frac{(\lambda-\mu)(\lambda-\nu)}{(\lambda+c_1^2)(\lambda+c_2^2)(\lambda+c_3^2)}},\\ \\
h_{\mu}&=&\D\frac{1}{2}\sqrt{\frac{(\mu-\nu)(\mu-\lambda)}{(\mu+c_1^2)(\mu+c_2^2)(\mu+c_3^2)}},\\ \\
h_{\nu}&=&\D\frac{1}{2}\sqrt{\frac{(\nu-\lambda)(\nu-\mu)}{(\nu+c_1^2)(\nu+c_2^2)(\nu+c_3^2)}},\\ \\
\end{array}\eeq
and the function  $$f(\lambda )=\sqrt{(\lambda +c_1^2)(\lambda +c_2^2)(\lambda +c_3^2)}.$$ 
If $\xi$ is any smooth function of $\lambda$, then its Laplacian is given by 
\beq\label{laplacian of lambda}
\Delta\xi=\frac{4}{(\lambda-\mu)(\lambda-\nu)}f(\lambda)\partial_\lambda(f(\lambda)\partial_\lambda \xi).
\eeq
according to \cite[(7) p. 150]{Lamb}.
We search $\phi _i$ in the form  $\phi_i(y_1,y_2,y_3)=y_i\xi _i(y_1,y_2,y_3)$. Then
\beq\label{desc of phi}
0=\Delta \phi _i=y_i\Delta \xi _i+2\partial_i\xi_i.
\eeq
Assuming furthermore that  $\xi_i$ depends only on $\lambda$, we obtain that
\be
\label{W1}
\frac{2\partial_i\xi_i}{y_i}=\frac{2\partial_\lambda y_i}{y_i}\frac{\partial_\lambda\xi_i}{h_\lambda^2}=\frac{1}{c_i^2+\lambda}\frac{\partial_\lambda\xi_i}{h_\lambda^2}=\frac{4f^2(\lambda)}{c_i^2+\lambda}\frac{\partial_\lambda\xi_i}{(\lambda-\mu)(\lambda-\nu)}. 
\ee
Combining \eqref{desc of phi} with \eqref{laplacian of lambda} and \eqref{W1}, we arrive to
$$0=\partial_\lambda(f(\lambda)\partial_\lambda\xi_i)+\frac{1}{c_i^2+\lambda}f(\lambda)\partial_\lambda\xi_i,$$ 
which is readily integrated as 
$$ \xi_i=-\hat C_i\int\limits_\lambda^{+\infty}\frac{ds}{(c_i^2+s)f(s)} +\hat C.$$

We choose the constant $\hat C=0$ for \eqref{eq for phi:02}
to be fulfilled. As $\partial \Omega$ is represented by the equation $\lambda=0$,  
then \eqref{eq for phi:01} reads
$$\partial_{\n}\phi_i=\n_i\Leftrightarrow \xi_i \frac{\partial_\lambda y_i}{y_i}+\partial_\lambda\xi_i=\frac{\partial_\lambda y_i}{y_i}.$$
We infer that $\hat C_i =c_1c_2c_3 /(2-\alpha_i),$ where 
$$\alpha_i= c_1c_2c_3\int\limits_0^{+\infty}\frac{ds}{(c_i^2+s)f(s)}.$$
It is easy seen that 
$$\frac{2c_2c_3}{3c_1^2}\leq \alpha_i\leq \frac{2c_1c_2}{3c_3^2}.$$ 
It follows that if $c_1,c_2,c_3$ are sufficiently close, then $\alpha_i$ is different from $2$, so that  $\hat C_i$ is well defined. We conclude that
\beq\label{form of phi}
\phi_i(y)=-\frac{\alpha_i}{2-\alpha_i}y_i,\;\;\forall y\in\partial\Omega.
\eeq

Let us now proceed to the computation of $\varphi _i$.
We search $\varphi _i$ in the form  $\varphi _i(y)=\frac{y_1y_2y_3}{y_i}\xi_i(y)$, where $\xi_i$ depends only on $\lambda$.  We obtain 
$$\Delta \xi_i +2\sum_{j=1,j\neq i}^3\frac{\partial_{y_j}\xi_i}{y_j}=0
\Leftrightarrow \partial_\lambda (f(\lambda)\partial_\lambda \xi_i)
+\Big(\sum_{j=1,j\neq i}^3\frac{1}{(c^2_j+\lambda)}\Big)f(\lambda)\partial_\lambda\xi_i=0,$$
and hence
$$  \xi_i=-\tilde C_i\int\limits_\lambda^{+\infty}\frac{c_i^2+s}{f^3(s)}ds.$$
From \eqref{eq for phi:01}-\eqref{eq for phi:02}, we infer that
$$\begin{array}{ccc}
\D \tilde C_1=\frac{c_1c_2c_3(c_2^2-c_3^2)}{2-\beta_1},& \;\;\;&\D  \beta_1=c_1c_2c_3(c_2^2+c_3^2)\int\limits^{+\infty}_0\frac{ds}{(c_2^2+s)(c_3^2+s)f(s)},\\ \\ 
\D \tilde C_2=\frac{c_1c_2c_3(c_3^2-c_1^2)}{2-\beta_2},& \;\;\;& \D \beta_2=c_1c_2c_3(c_3^2+c_1^2)\int\limits^{+\infty}_0\frac{ds}{(c_3^2+s)(c_1^2+s)f(s)},\\ \\
\D \tilde C_3=\frac{c_1c_2c_3(c_1^2-c_2^2)}{2-\beta_3},& \;\;\;&\D  \beta_3=c_1c_2c_3(c_1^2+c_2^2)\int\limits^{+\infty}_0\frac{ds}{(c_1^2+s)(c_2^2+s)f(s)}.
\end{array}
$$
Note that at the limit case  $c_1=c_2=c_3$, we obtain $\beta_1=\beta_2=\beta_3=4/5$.
 Therefore, if $c_1,c_2$ and $c_3$ are near but different, then $\beta_i$ is different  from $2$, and therefore $\tilde C_i$ is well defined. We conclude that
\beq\label{form of varphi}
 \varphi_i=-\left(\tilde C_i \int\limits_0^{+\infty}\frac{c_i^2+s}{f^3(s)}ds\right)\frac{y_1y_2y_3}{y_i},\;\forall y \in\partial\Omega .
\eeq

\subsubsection{Controllability of the ellipsoid with six controls}
Assume still that $\S$ is given by \eqref{WW1}. Note that $\S$ is symmetric with respect to the plane $\{y_p=0\}$ for $p=1,2,3$.  
Assume given six functions $\chi_j$, $j=1,\ldots , 6$, each being  symmetric with respect to the plane $\{y_p=0\}$ for $p = 1,2,3$, with
\ba
\varepsilon_{\chi_1}^p=\left\{
\begin{array}{cc}
-1& p=1\\
1 & p=2\\
1 & p=3
\end{array}\right.
,
&\varepsilon_{\chi_2}^p=\left\{
\begin{array}{cc}
1 & p=1\\
-1 & p=2 \\
1 & p=3
\end{array}\right.
,
&\varepsilon_{\chi_3}^p=\left\{
\begin{array}{cc}
1 & p=1\\
1 & p=2 \\
-1 & p=3
\end{array}\right., \nonumber\\
\varepsilon_{\chi_4}^p=\left\{
\begin{array}{cc}
1 & p=1\\
-1 & p=2 \\
-1 & p=3
\end{array}\right.
,
&\varepsilon_{\chi_5}^p=\left\{
\begin{array}{cc}
-1 & p=1\\
1  & p=2\\
-1 & p=3
\end{array}\right.
,
&\varepsilon_{\chi_6}^p=\left\{
\begin{array}{cc}
-1 & p=1\\
-1 & p=2 \\
1 & p=3
\end{array}\right. .\label{WW2}
\ea
To obtain this kind of controls in practice,  we can proceed as follows:
\begin{itemize}
\item We build six tunnels in the rigid body, as drawn in  Figure \ref{fig11}.
\begin{figure}[h] 
\begin{center}
\includegraphics[width=.5\textwidth]
{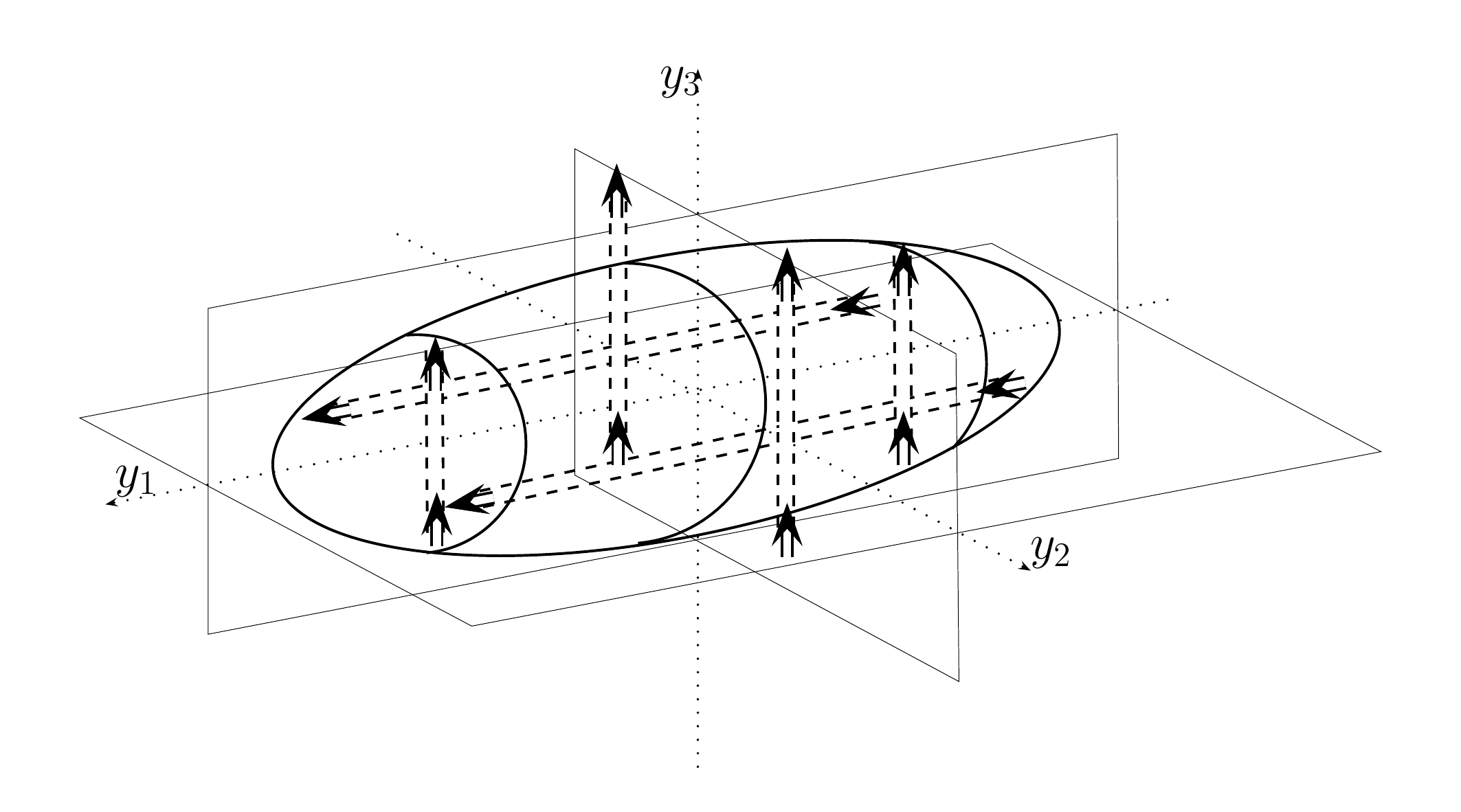} 
\caption{Ellipsoid with six controls.}
\label{fig11}
\end{center}
\end{figure}
\item We divide the six tunnels in three groups of two parallel tunnels; that is, we put together the tunnels located in the same plane
(see Figure \ref{fig12}).
\begin{figure}
\begin{minipage}[t]{.45\textwidth}
\begin{center}
\includegraphics[width=1\textwidth]{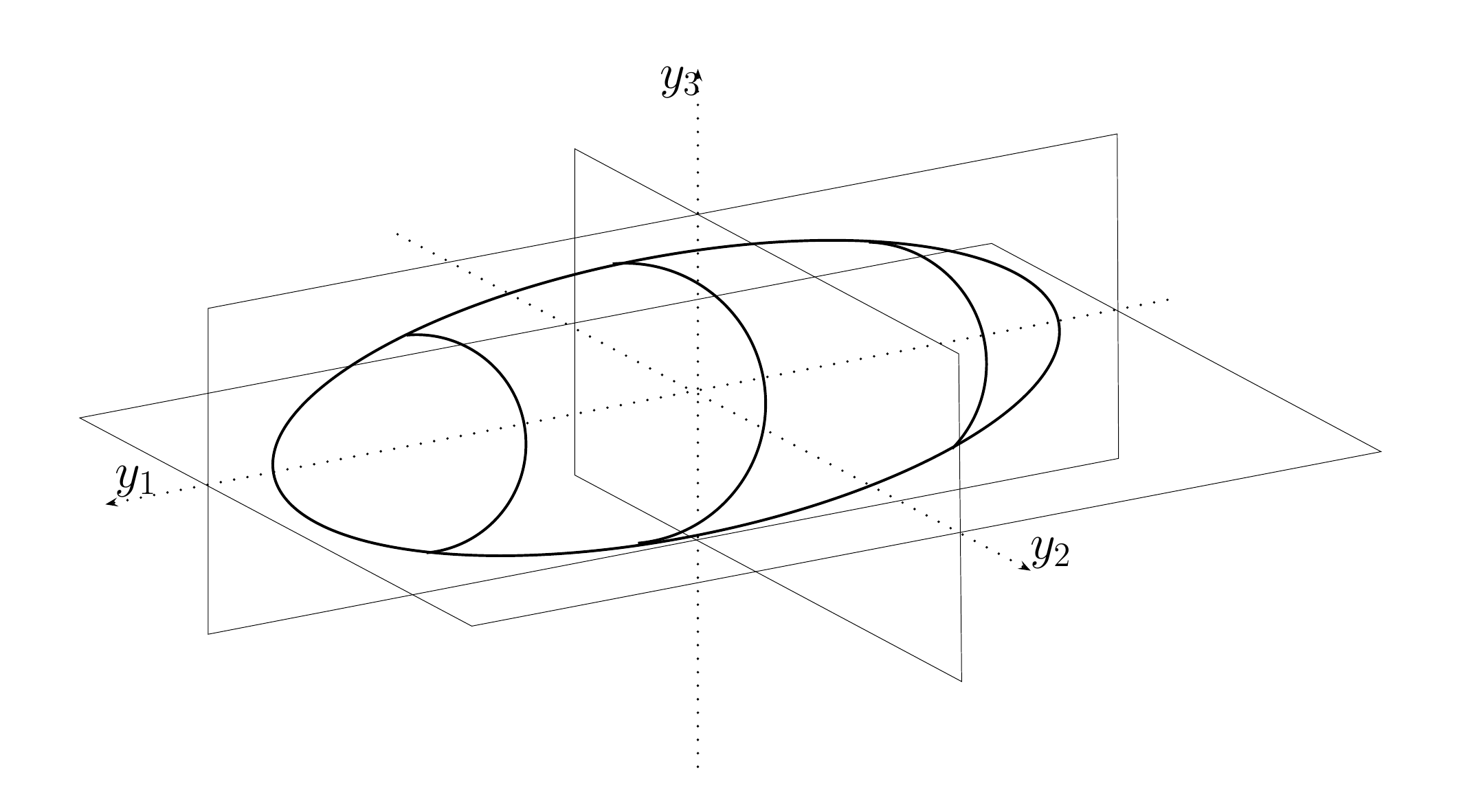} 
\end{center}
\end{minipage}
\begin{minipage}[t]{.45\textwidth}
\begin{center}
\includegraphics[width=1\textwidth]{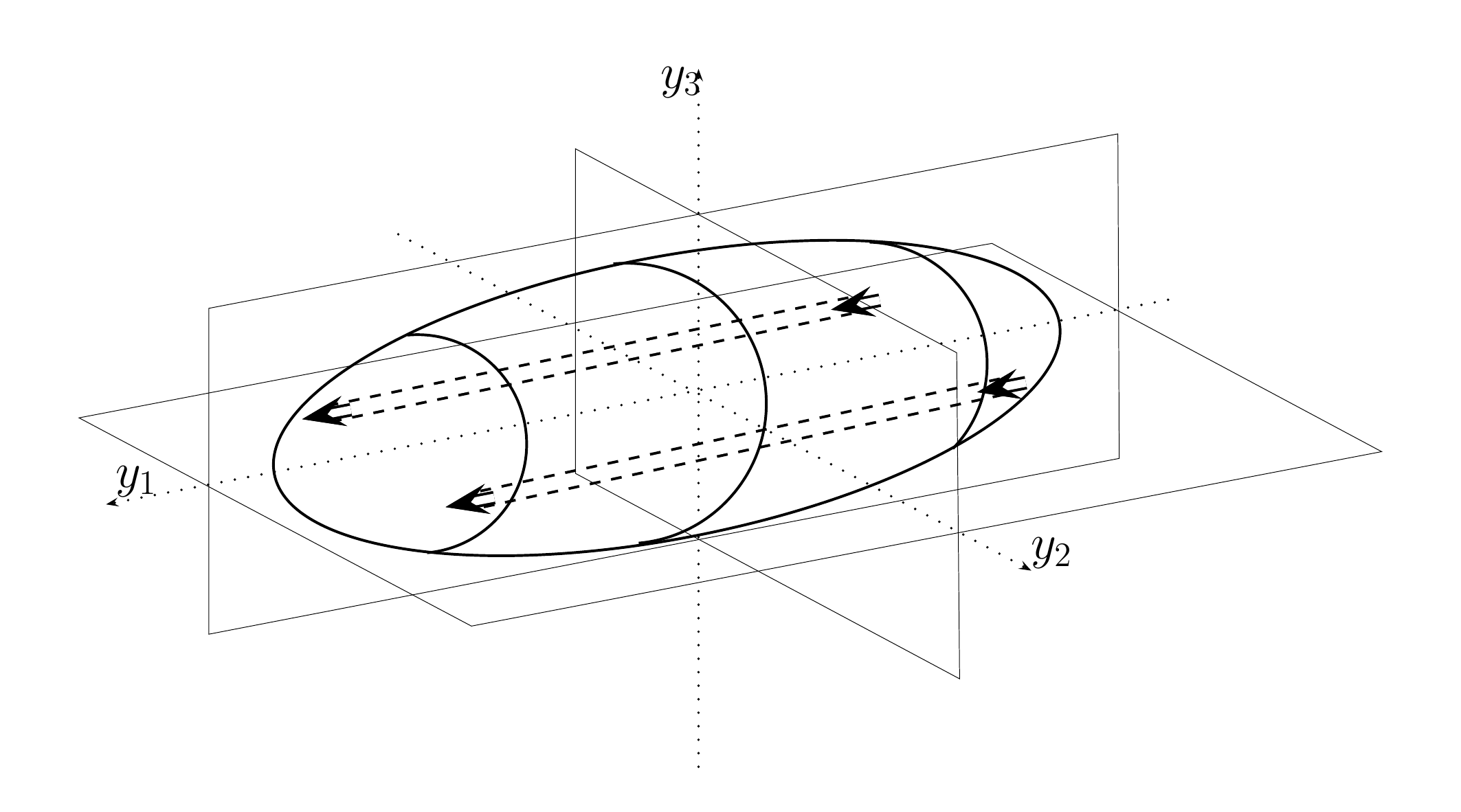} 
\end{center}
\end{minipage}
\begin{minipage}[t]{.45\textwidth}
\begin{center}
\includegraphics[width=1\textwidth]{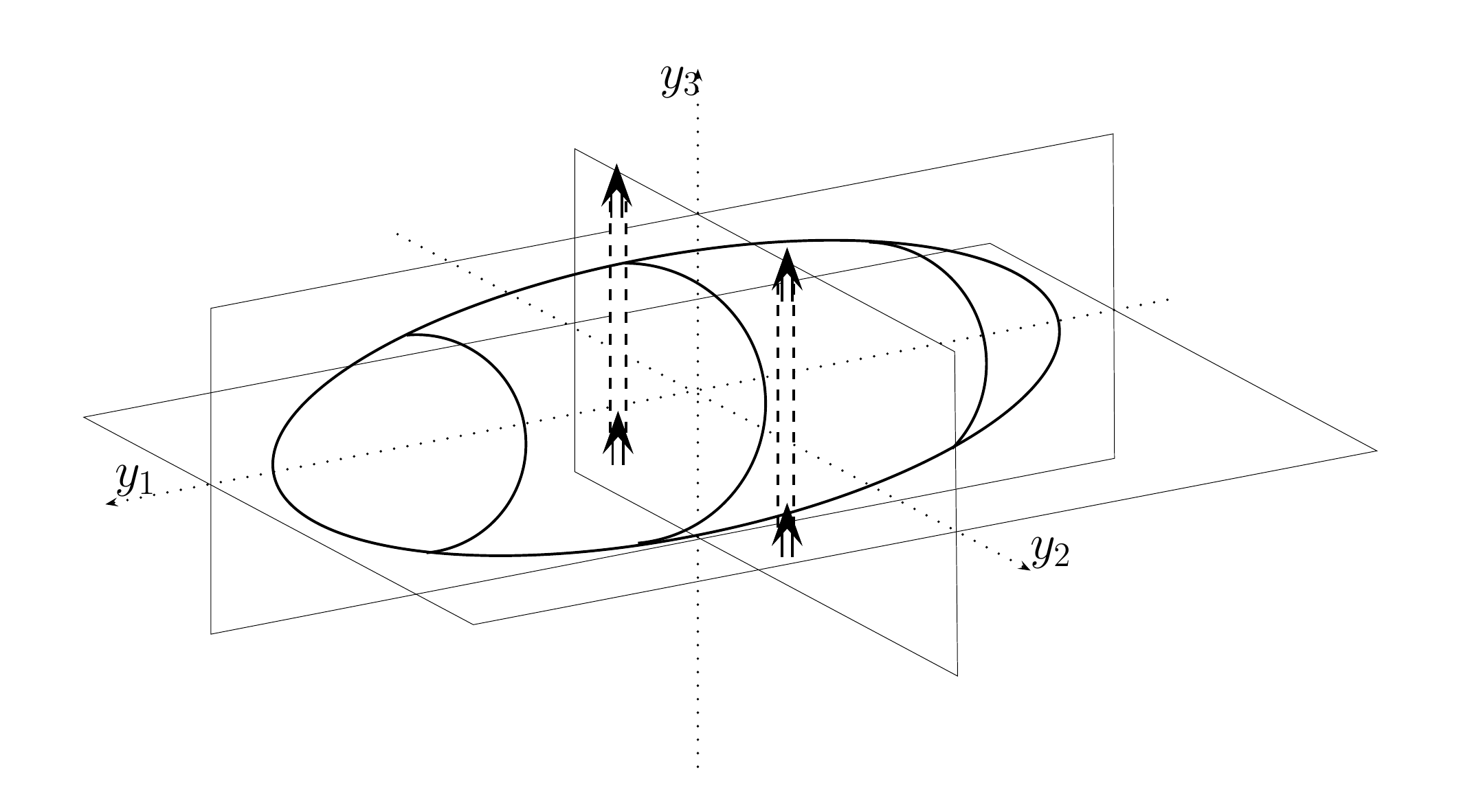} 
\end{center}
\end{minipage}
\begin{minipage}[t]{.45\textwidth}
\begin{center}
\includegraphics[width=1\textwidth]{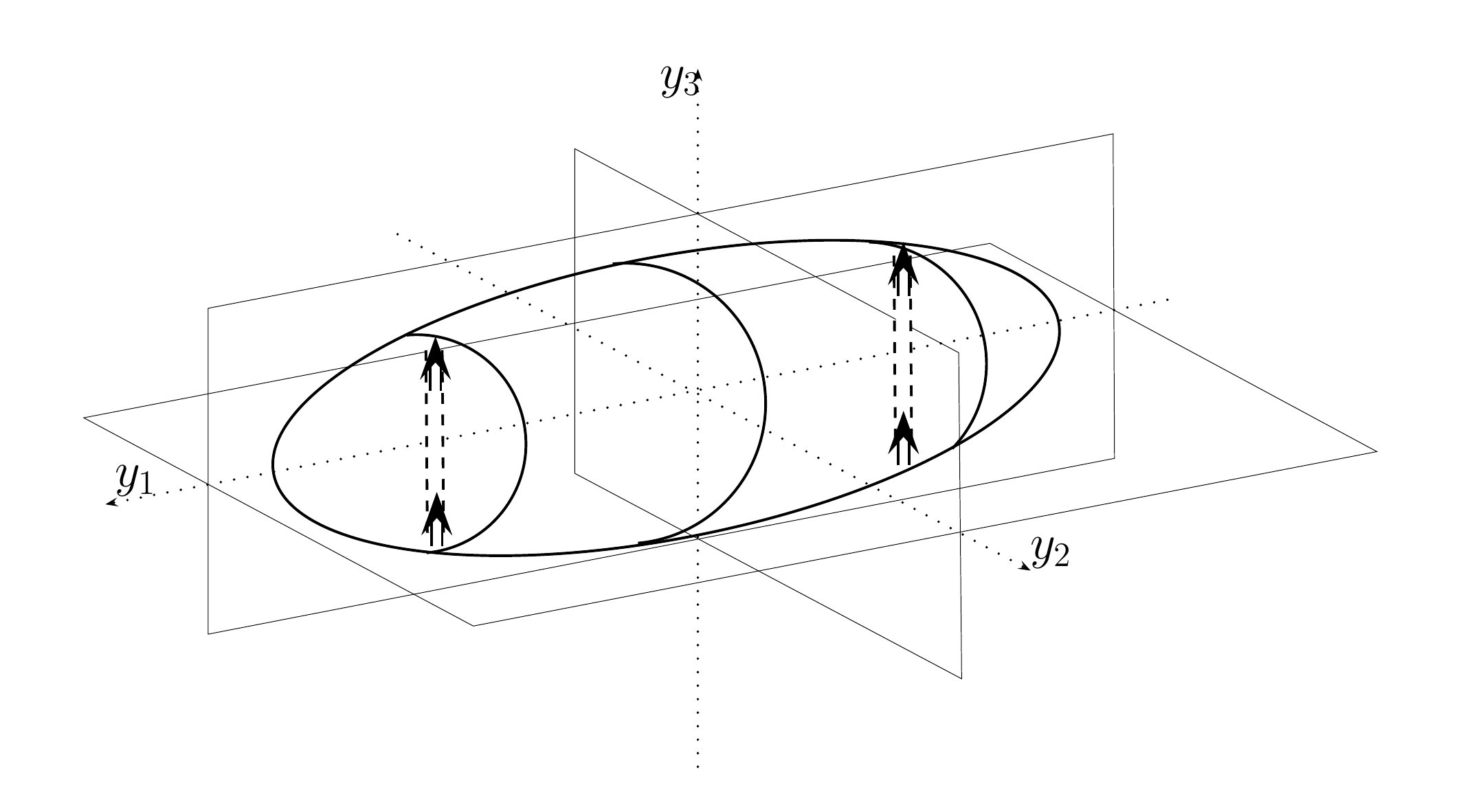} 
\end{center}
\end{minipage}
\caption{Independent controls in each plane.}
\label{fig12}
\end{figure}
\item  Let $\tilde w_1$ and $\tilde w_2$  denote the effective flow controls in the two tunnels located  in the plane $\{ y_3=0 \}$. They may 
appear together in \eqref{boundary condition} as
$\tilde w_1\chi(y_1,y_2,y_3) + \tilde w_2\chi(y_1,-y_2,y_3)$, where $\chi \in C^\infty (\partial \Omega)$ is some function with
 \[
 \textrm{Supp } \chi  \subset \{ y_2 >0 \} , \quad \varepsilon _{\chi}^1= -1, \quad \textrm{and} \quad \varepsilon _{\chi}^3=1. 
 \] 
 We introduce the (new) support functions 
 \begin{eqnarray*}
 \chi_1(y_1,y_2,y_3) &=& \chi (y_1,y_2,y_3) + \chi (y_1,-y_2,y_3),\\
 \chi_6(y_1,y_2,y_3) &=& \chi (y_1,y_2,y_3)  - \chi (y_1,-y_2,y_3)
 \end{eqnarray*}
 and the (new) control inputs
 \begin{eqnarray*}
 w_1&=&\frac{\tilde w_1+\tilde w_2}{2}, \\
 w_6&=&\frac{\tilde w_1-\tilde w_2}{2}.
 \end{eqnarray*}
 (See Figure \ref{fig13}.) 
 Then \eqref{WW2} is satisfied for $\chi_1$ and $\chi _6$, and 
 \[
 \tilde w_1\chi(y_1,y_2,y_3) + \tilde w_2\chi(y_1,-y_2,y_3) = w_1 \chi_1(y_1,y_2,y_3) + w_6 \chi _6(y_1,y_2,y_3).
 \]
The same can be done in the other planes $\{y_1 =0 \}$ and $\{ y_2=0 \}$. 
\begin{figure}[h] 
\begin{center}
\includegraphics[width=.9\textwidth]
{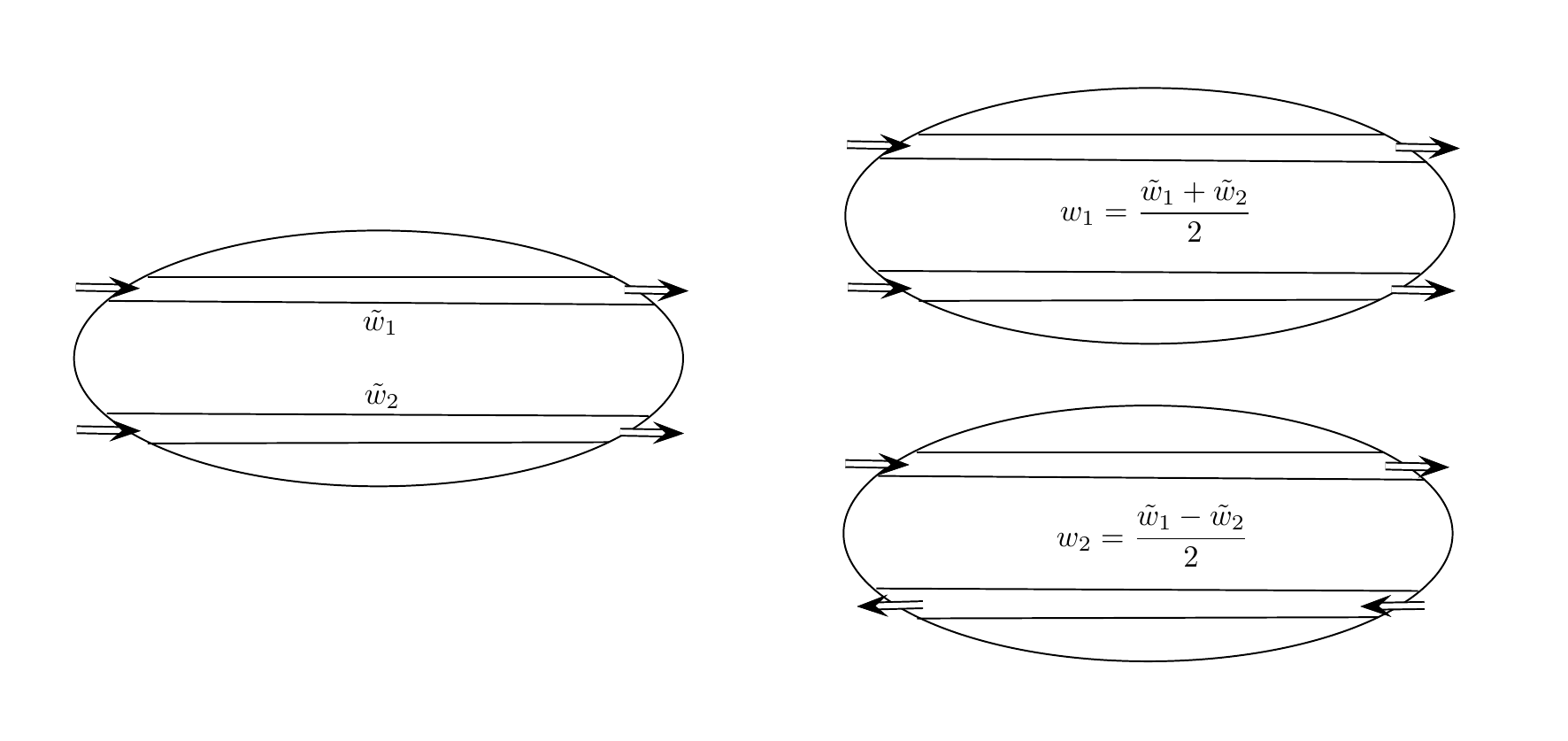} 
\caption{Definition of the new controls in the plane $\{ y_3=0 \}$. }
\label{fig13}
\end{center}
\end{figure}
\end{itemize}

We notice that  $C$ is a diagonal matrix:
\[ C=-\textrm{diag}(C_1,C_2,C_3,C_4,C_5,C_6),\]
with
$$C_i=\Bi{\phi_i\chi_i},\; i=1,2,3, \quad \text{and} \quad C_{i+3}=\Bi{\varphi_{i}\chi_{i+3}},\; i=1,2,3.$$
From \eqref{form of phi}-\eqref{form of varphi}, there are some constants $\bar C_i \neq 0$, $i=1,...,6$, which depend only on $c_1,c_2$ and $c_3$, 
such that
\be
C_i=\bar C_i \int\limits_{\partial\Omega \cap (0,+\infty )^3 }{y_i \chi _i}(y),\;\;\;\;C_{i+3}=\bar C_{i+3} \int\limits_{\partial\Omega\cap (0,+\infty ) ^3}
\left(\frac{y_1y_2y_3}{y_{i}}\right) \chi_{i+3}(y),\;\;\; i=1,2,3.
\label{WW4}
\ee

By \eqref{WW4}, we have that $C_i\ne 0$ for $i=1,...,6$, and hence rank$(C)=6$
if, in addition  to \eqref{WW2}, it holds
\ba
\label{WW5}
&&\chi_i \not\equiv 0,\ i=1,...,6,\\
\label{WW6}
&&\chi _i \ge 0 \text{ on } \partial \Omega  \cap (0,+\infty )^3 , \ i=1,...,6.
\ea
By Proposition \ref{prop4} and Theorem \ref{thm1}, it follows that both the linearized system (\ref{lineal system pq}) and the nonlinear system 
\eqref{systempq} are (locally) controllable.
\begin{remark}
Since $\varepsilon _{\chi _1}^1=(-1)^{\delta _{11}} = -1$, we have that $(L_1^M)_{11}=(W_1^M)_{11}=0$, and hence $\beta=\gamma =0$. 
Thus $\gamma + \alpha\beta=0$. Proceeding as in \cite[Theorem 2.2]{GR}, one can prove that, under certain rank conditions, two arbitrary states 
of the form $(h,\vec q, 0,0)$ can be connected by trajectories of the ellipsoid in (sufficiently) large time.  
\end{remark}

In the following sections, we shall be concerned with the controllability of the ellipsoid with less controls (namely, $4$ controls  and $3$ controls). 
If, in the pair $(\chi _1, \chi _6)$,  only $\chi _6$ is available, then $\chi _6$ can be generated as above by two propellers controlled  in the same way
(Figure \ref{fig30} left), or by only one propeller by choosing an appropriate scheme for the tunnels (Figure \ref{fig30} middle). 
In what follows, to indicate that the flows in the two tunnels are {\em linked}, we draw a transversal line in bold between the two tunnels (Figure \ref{fig30} right). 

\begin{figure}[h] 
\begin{center}
\includegraphics[width=.9\textwidth]
{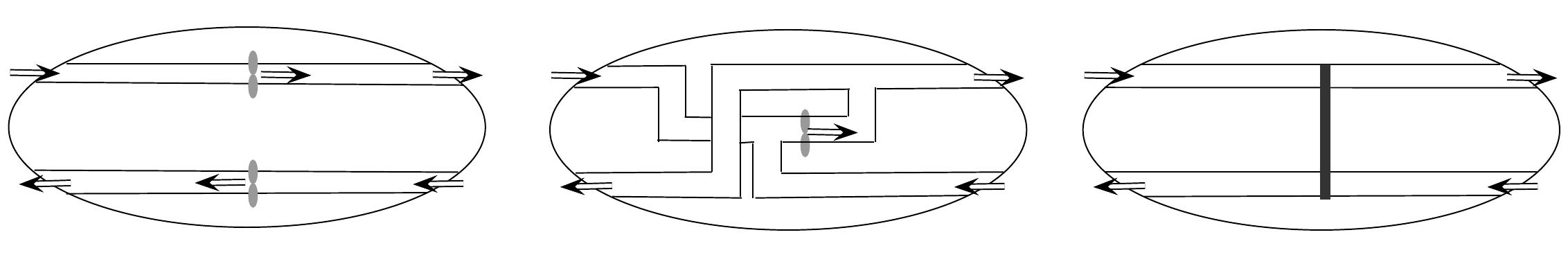} 
\caption{Two ways to generate $\chi_6$.}
\label{fig30}
\end{center}
\end{figure}

\subsubsection{Controllability of the ellipsoid with four controls}

We consider the same controllers $\chi_1,\chi_4,\chi_5$ and $\chi_6$ as above, still satisfying \eqref{WW2}, \eqref{WW5}, \eqref{WW6}.
(See Figure \ref{fig20}.)  
\begin{figure}
\begin{minipage}[t]{.45\textwidth}
\begin{center}
Control $\chi_1$
\includegraphics[width=1\textwidth]{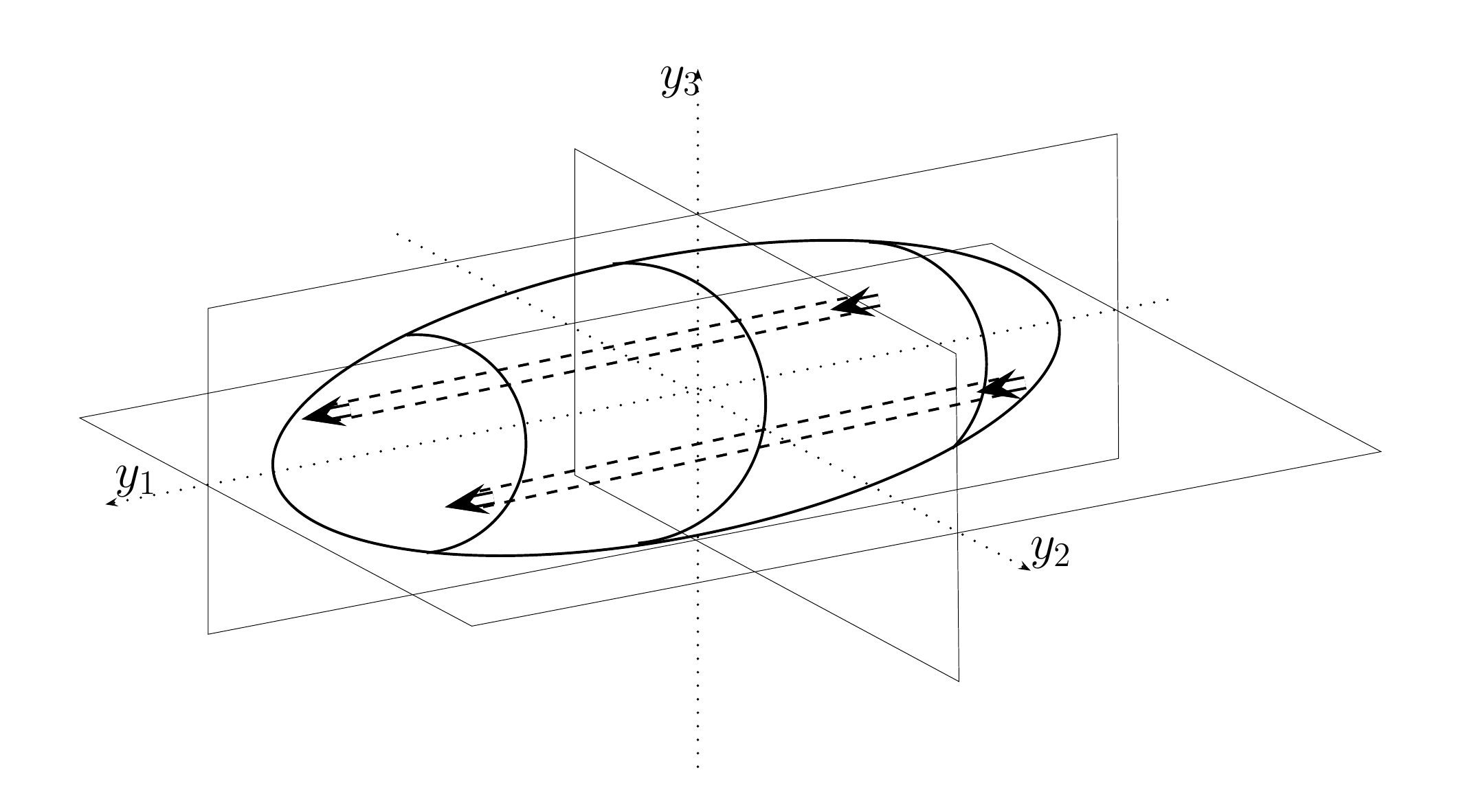} 
\end{center}
\end{minipage}
\begin{minipage}[t]{.45\textwidth}
\begin{center}
Control $\chi_4$
\includegraphics[width=1\textwidth]{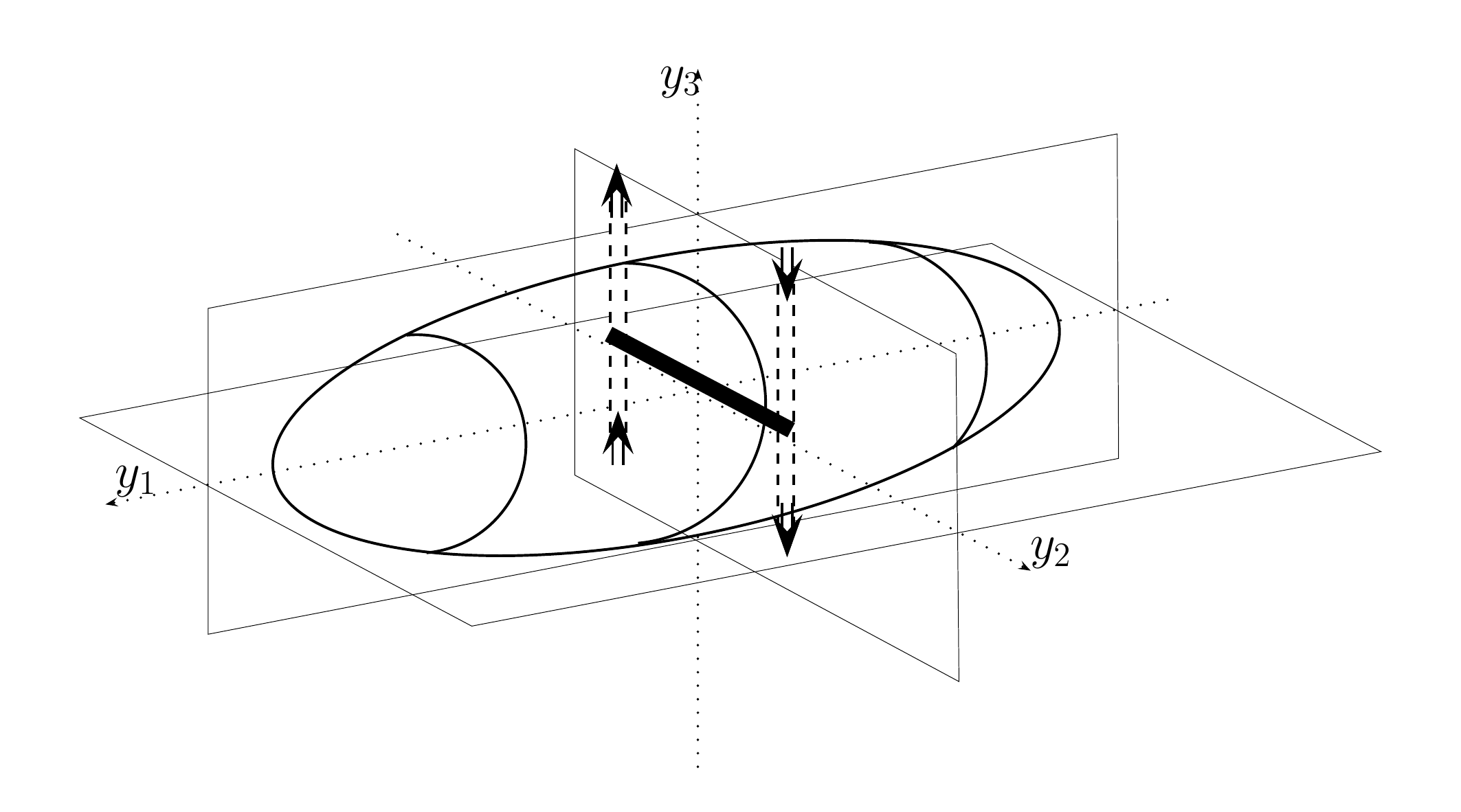} 
\end{center}
\end{minipage}
\begin{minipage}[t]{.45\textwidth}
\begin{center}
Control $\chi_5$
\includegraphics[width=1\textwidth]{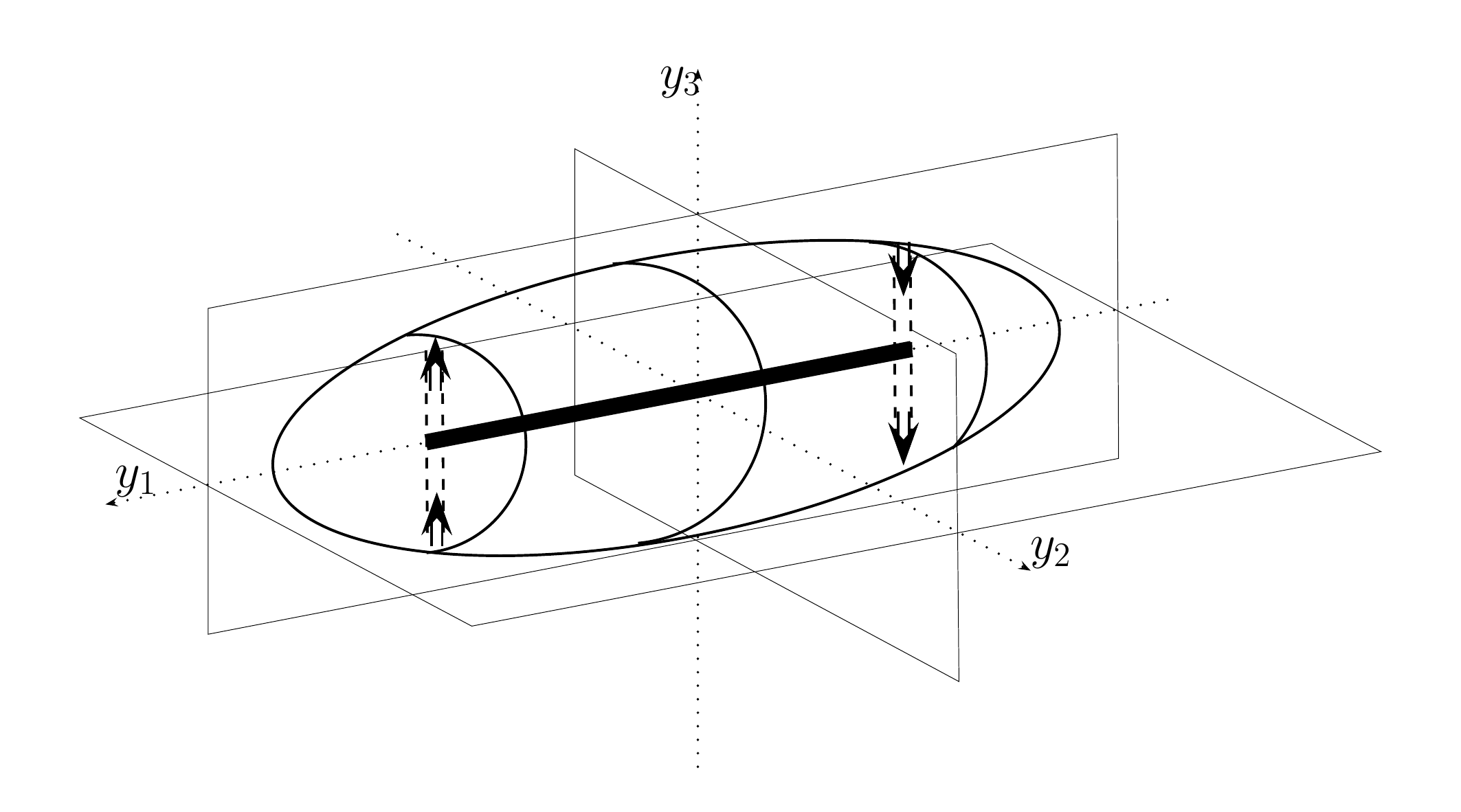} 
\end{center}
\end{minipage}
\begin{minipage}[t]{.45\textwidth}
\begin{center}
Control $\chi_6$
\includegraphics[width=1\textwidth]{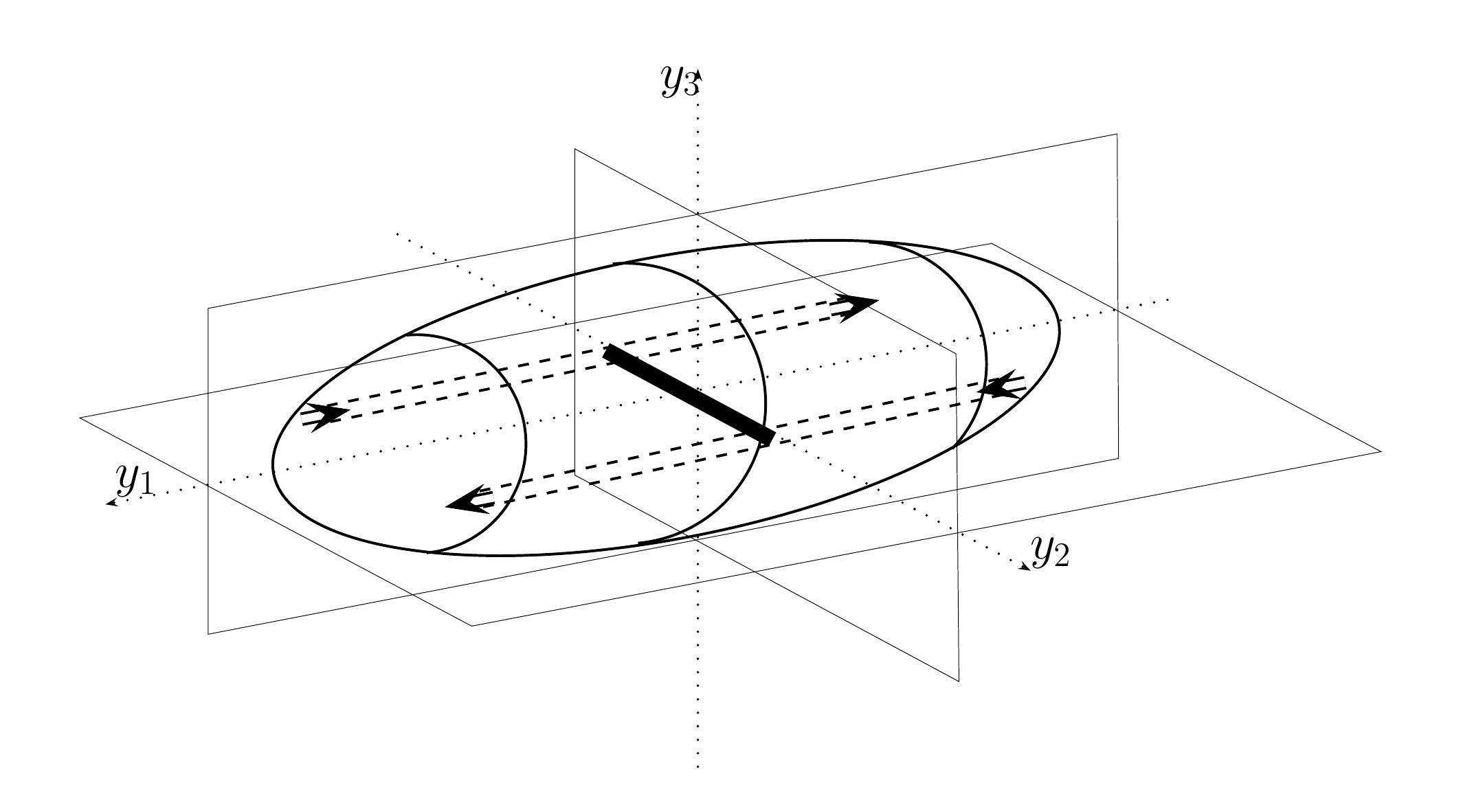} 
\end{center}
\end{minipage}
\caption{Ellipsoid with four controls.}
\label{fig20}
\end{figure}
If the density $\rho$ is scaled by a factor $\lambda$, i.e. $\rho (x)$ is replaced by $\rho ^\lambda (x)=\lambda \rho (x)$ where $\lambda >0$, 
then the mass and the inertia matrix are scaled  in the same way; that is, $m_0$ and $J_0$ are replaced by
\[
m_0^\lambda = \lambda m_0,\quad J_0^\lambda = \lambda J_0.
\] 
Thus, if $\lambda \to \infty$, then $m_0^\lambda \to \infty $, $[J_0^\lambda]^{-1}  \to 0$, and $[ {\mathcal J}^\lambda ]^{-1}\to 0$.
(Note that large values of $\lambda$ are not compatible with the neutral buoyancy, but they prove to be useful to identify geometric configurations
leading to controllability results with less than six control inputs.)

Note that the matrices $M,J,N,C^M,C^J,L_p^M,R_p^M,W_p^M,L_p^J,R_p^J,W_p^J$ keep constant when $\lambda \to \infty$.
In particular,
\[
\lim_{\lambda \to \infty} {\bf A} {\mathcal J}^{-1} C=0,\quad \lim \frac{1}{2} {\mathcal J} {\bf D} {\mathcal J}^{-1} C=0.
\]

Let $B^\infty=\lim\limits_{\lambda \to \infty}{\bf B}$. Then $B^\infty$ and $C$ are given by
$$
\begin{array}{cc}
B^\infty=-
\left(\begin{array}{cccc}
0&0&0&0\\
0&0&0&B_6\\
0&0&B_5&0\\
0&0&0&0\\
0&0&0&0\\
0&0&0&0
\end{array}\right),
&C=-
\left(\begin{array}{cccc}
C_1&0&0&0\\
0&0&0&0\\
0&0&0&0\\
0&C_4&0&0\\
0&0&C_5&0\\
0&0&0&C_6
\end{array}\right), 
\end{array}
$$
with 
$$B_5=\Bi\left(\nabla\psi_1\cdot \nabla\psi_5\right)\n_3,\;\;\;B_6=\Bi\left(\nabla\psi_1\cdot \nabla\psi_6\right)\n_2.$$
Thus, if $B_5\ne 0$ and $B_6\ne 0$, we see that \eqref{cond rank 1} and \eqref{cond rank 2} are fulfilled, so that 
the local controllability of \eqref{systempq} is ensured by Corollary \ref{cor1} for $\lambda $ large enough.
We note then that  the matrix in $\R ^{6\times 6}$ obtained by gathering together the four columns of $C$ and the last  two columns of $B^\infty $ is invertible.
Let  $R_1(\lambda)\in \R ^{6\times 6}$ (resp.  $R_2(\lambda)\in \R ^{6\times 6}$) denote the matrix obtained by gathering  together the four columns of 
$C$ with the last two columns of ${\bf B} + {\bf A } {\mathcal J} ^{-1}C$ (resp. with the  last two   columns of 
$ \frac{1}{2} {\mathcal J} {\bf D} {\mathcal J}^{-1} C + {\bf B} + {\bf A } {\mathcal J} ^{-1}C$). Then for $\lambda \gg 1$, we have 
\[
\textrm{det }R_1(\lambda )\ne 0\quad\textrm { and }\quad   \textrm{det }R_2(\lambda )\ne 0.
\]
Since the coefficients of $R_1(\lambda ),R_2(\lambda )$ are rational functions of $\lambda$, 
we infer that the equation
\[
\textrm{det }R_1(\lambda )\cdot   \textrm{det }R_2(\lambda )  =0
\] 
is an {\em algebraic} equation in $\lambda$.  Therefore, it has at most a {\em finite} set of roots in $(0,+\infty )$, 
 that we denote by $\Lambda _{critical}$. 
 We conclude that for any $\lambda \in (0,+\infty) \setminus\Lambda_{critical}$, the local controllability of   \eqref{systempq} still holds.
 In particular, we can consider values of $\lambda$ arbitrary close to the value $\lambda=1$ imposed by \eqref{buoyancy}.
 The issue whether $1\in \Lambda _{critical}$ seems hard to address without computing numerically all the coefficients in our system.

\subsubsection{Controllability of the ellipsoid with three controls}

Assume that $\chi_1, \chi _4, \chi_5$ and $\chi _6$ are as above  (satisfying \eqref{WW2}, \eqref{WW5}, \eqref{WW6}), and consider now the controls supported by 
$\chi _1$, $\chi _4$ and $\tilde \chi _5=\chi _5 + \chi _6$ (see Figure \ref{fig21}). 
\begin{figure}
\begin{minipage}[t]{.45\textwidth}
\begin{center}
Control $\chi_1$
\includegraphics[width=1\textwidth]{elipce_4Control1.pdf} 
\end{center}
\end{minipage}
\begin{minipage}[t]{.45\textwidth}
\begin{center}
Control $\chi_4$
\includegraphics[width=1\textwidth]{elipce_4Control3.pdf} 
\end{center}
\end{minipage}
\begin{minipage}[t]{.45\textwidth}
\begin{center}
Control $\tilde \chi_5$
\includegraphics[width=1\textwidth]{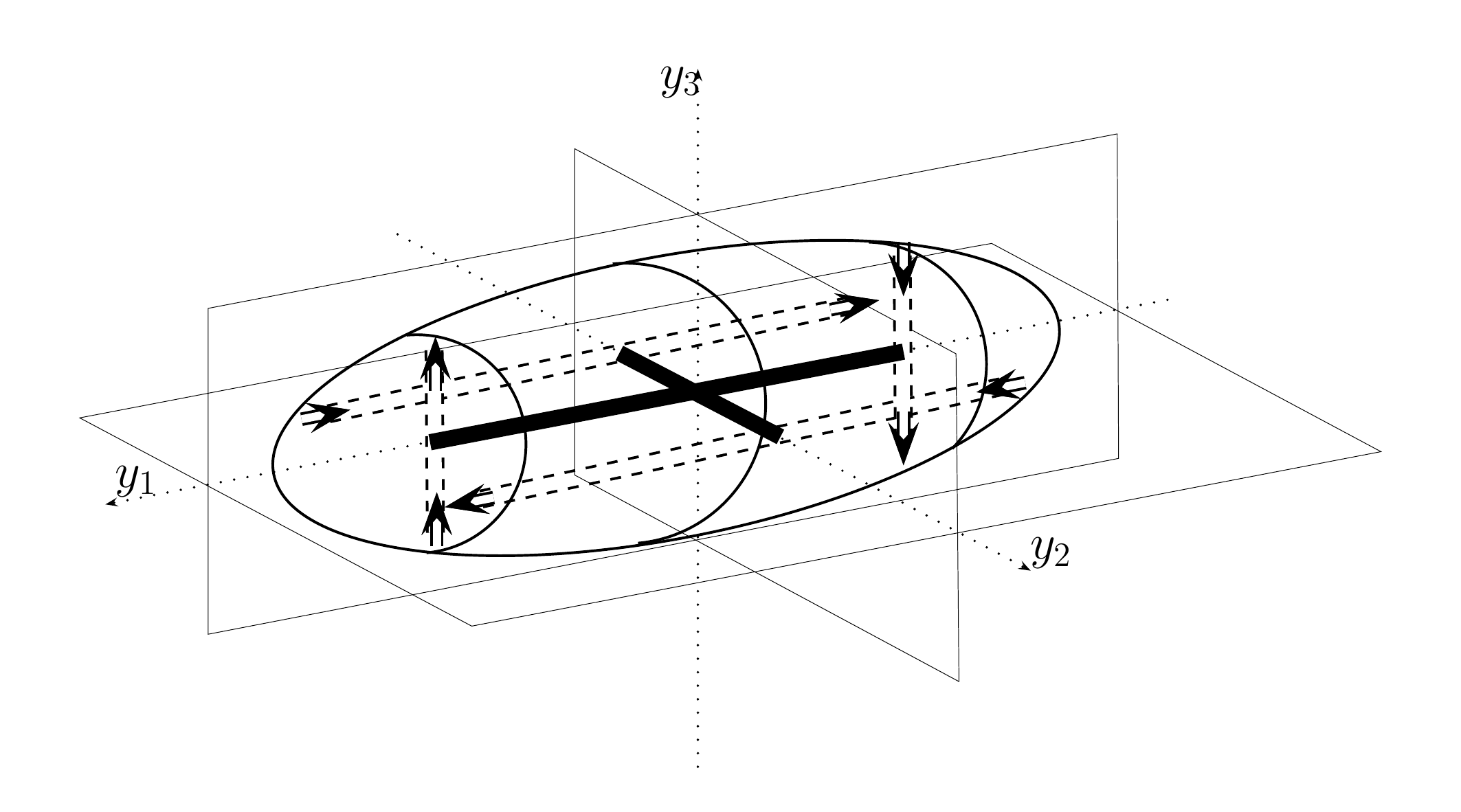} 
\end{center}
\end{minipage}
\caption{Ellipsoid with three controls.}
\label{fig21}
\end{figure}

Doing the same scaling for the density, and letting $\lambda \to \infty$,
we see that the matrices $B^\infty$ and $C$ read
$$
\begin{array}{cc}
B^\infty=-
\left(\begin{array}{cccc}
0&0&0\\
0&0&B_6\\
0&0&B_5\\
0&0&0\\
0&0&0\\
0&0&0
\end{array}\right),
&C=-
\left(\begin{array}{ccc}
C_1&0&0\\
0&0&0\\
0&0&0\\
0&C_4&0\\
0&0&C_5\\
0&0&C_6
\end{array}\right), 
\end{array}
$$
where the coefficients $B_5,B_6,C_1,C_4,C_5,C_6$ are as above.
For simplicity, we assume that the principal axes of inertia of the vehicule coincide with the axes of the ellipsoid. Then the matrix $J_0$ is diagonal 
(see \cite{Chambrion-Sigalotti}) with entries $J_1,J_2,J_3$.
Notice that the first and fourth coordinates are well controlled (using $\chi_1$ and $\chi _4$), and that the other coordinates are decoupled from them, at  least
asymptotically (i.e. when $\lambda \to \infty$). Let $A^\infty=\lim_{\lambda \to \infty}{\bf A}$ (i.e. $A^\infty$ is obtained by letting $\alpha =0$ in ${\bf A}$). 
 Let $K\in \R ^{4\times 4}$ denote the matrix obtained from $A^\infty$ by removing the first
and fourth lines (resp. columns), and let $b\in \R ^4$ (resp. $c\in \R ^4$) denote the vector obtained from the last column of $B^\infty $ (resp. $C$) 
by removing the first and fourth coordinates, namely 
\begin{eqnarray*}
&&K=
\left( 
\begin{array}{cccc}
-(L_1^M)_{22} &0&0& -(R_1^M)_{23}\\
0 & -(L_1^M)_{33} & -(R_1^M)_{32} & 0 \\
0& -((L_1^J)_{23} +(C^M)_{11}) & -(R_1^J)_{22} & 0 \\
-((L_1^J)_{32} -(C^M)_{11}) &0 &0 & -(R_1^J)_{33} 
\end{array}
\right) , \\
&&b=
\left( 
\begin{array}{cccc}
B_6\\B_5\\0\\0
\end{array}
\right) ,\qquad 
c=
\left( 
\begin{array}{cccc}
0\\0\\C_5\\C_6
\end{array}
\right).
\end{eqnarray*}
Let finally 
\[
 F=
\left( 
\begin{array}{cccc}
0 & 0 & 0 & -1\\
0 & 0 & 1& 0 \\
0& 0 & 0 & 0 \\
0 &0 &0 & 0 
\end{array}
\right) \quad \text{ and } 
G=
\left( 
\begin{array}{cccc}
m_0^{-1} & 0 & 0 & 0\\
0 & m_0^{-1} & 0& 0 \\
0& 0 & J_2^{-1} & 0 \\
0 &0 &0 & J_3^{-1} 
\end{array}
\right) .
\]
Then, keeping only the leading terms  as $\lambda \to \infty$, we see that 
\eqref{newcond1} holds if  
\be
\label{WWW1}
\textrm{rank }(c,b,KGb,(KG)^2b, (KG)^3b)=4 
\ee 
while \eqref{newcond2} holds if 
\begin{multline}
\label{WWW2}
\textrm{rank }\bigg( c,b, [ (C^M)_{11}F+ 2K ] Gb,
[ 8 (C^M)_{11}F+ 11 K ] GK G b,\\
[ 17 (C^M)_{11}F+ 64 K ]G(KG)^2b\bigg)
=4. 
\end{multline}
Note that \eqref{WWW1} is satisfied whenever 
\be
\textrm{rank} (b,KGb,(KG)^2b,(KG)^3b)=4,
\label{kalman}
\ee
which is nothing but the Kalman rank condition for the system 
$\dot x = KGx+bu$. However, it is clear that we should take advantage of the presence $c$ in \eqref{WWW1}. 
As previously, this gives a controllability result for $\lambda \gg1$, but such a result
is  also valid for all the positive $\lambda$'s except those in a finite set defined by an algebraic equation.

\section{Appendix}
\subsection{Quaternions and rotations.}
Quaternions are a convenient tool for representing rotations of objects in three dimensions. For that reason, they are
widely used in robotic, navigation, flight dynamics, etc. (See e.g. \cite{altmann,ST}). We limit ourselves to introducing the few
definitions and properties needed to deal with the dynamics of $h$ and $Q$. (We refer the reader to \cite{altmann} for more details.)

The set of quaternions, denoted by $\mathbb H$, is a noncommutative field containing  $\mathbb C$ and which is a $\mathbb R$-algebra 
of dimension 4. Any quaternion $q\in \mathbb H$ may be written as
\[
q=q_0 +q_1 i + q_2 j + q_3 k,  
\]
where $(q_0,q_1,q_2,q_3)\in \R^4$ and $i,j,k\in \mathbb H$ are some quaternions whose products will be given later. We say that
$q_0$ (resp. $q_1i+q_2j+q_3k$)  is the {\em real part} (resp. the {\em imaginary part}) of $q$. Identifying the imaginary part 
$q_1i+q_2j+q_3k$ with the vector $\vec q = (q_1,q_2,q_3)\in \R ^3$, we can represent the quaternion $q$ as 
 $q=[q_0,\vec q\, ]$, where $q_0\in \R$ (resp. $\vec q\in \R ^3$) is the {\em scalar part} (resp. the {\em vector part})  of $q$. 
The addition, scalar multiplication and quaternion multiplication are defined respectively by 
\begin{eqnarray*}
&&[p_0,\vec p\, ]+[q_0,\vec q\, ]=[p_0+q_0,\vec p + \vec q\,],\\
&&t[q_0,\vec q\, ]=[tq_0,t\vec q\, ],\\
&&[p_0,\vec p\, ] *[q_0,\vec q\, ]=[p_0q_0 -\vec p \cdot \vec q , p_0\vec q + q_0\vec p +\vec p \times \vec q\, ],
\end{eqnarray*}
where ``$\cdot$'' is the dot product and ``$\times$'' is the cross product. We stress that the quaternion multiplication $*$ is not 
commutative. Actually, we have that
\begin{eqnarray*}
&&i* j = k,\quad j*i =-k,\\
&&j*k=i,\quad k*j=-i,\\
&&k*i=j,\quad i*k=-j,\\
&&i^2=j^2=k^2=-1.
\end{eqnarray*} 
Any pure scalar $q_0$ and any pure vector $\vec q$ may be viewed as quaternions
\[
q_0 = [q_0, \vec 0 \, ], \qquad \vec q = [0, \vec q \,],
\]
and hence any quaternion $q=[q_0,\vec q\, ]$ can be written as the sum of a scalar and a vector, namely
\[
q=q_0 + \vec q.
\] 
The cross product of vectors extends to quaternions by setting
\[
p\times q = \frac{1}{2} ( p*q - q*p) = [0,\vec p \times \vec q \,].
\]
The {\em conjugate} of a quaternion $q=[q_0,\vec q \, ]$ is $q^*=[q_0, -\vec q \, ]$. The {\em norm} of $q$ is 
\[
||q|| = ( |q_0|^2 + ||\vec q \, ||^2) ^{\frac{1}{2}}.
\] 
From 
\[
q*q^*=q^* *q=||q||^2,
\]
we infer that 
\[
q^{-1}= \frac{q^*}{||q||^2}\cdot
\]
A {\em unit quaternion} is a quaternion of norm 1. The set of unit quaternions may be identified with $S^3$. 
It is a group for $*$. 

Any unit quaternion $q=[q_0,\vec q\, ]$ can be written in the form 
\be
\label{AP20}
q= \cos \frac{\alpha}{2} + \sin \frac{\alpha }{2} \vec u,
\ee
where $\alpha \in \R$ and $\vec u\in \R^3$ with $||\vec u ||=1$. Note that the writing is not unique: if the pair $(\alpha ,\vec u\, )$ 
is convenient, the same is true for the pairs $(-\alpha, -\vec u\,  )$ and $(\alpha + 4k\pi, \vec u \, )$ $(k\in \Z)$, as well. However, if we impose 
that $\alpha \in [0,2\pi ]$, then $\alpha $ is unique, and $\vec u$ is unique for $|q_0|<1$. (However, any $\vec u\in S^3$ is convenient for $|q_0|=1$.)

For any unit quaternion $q$, let the matrix $R(q)\in \R ^{3\times 3}$ be defined by 
\be
\label{AP0}
R(q){ \vec v} = q * {\vec v} * q^*\quad \forall {\vec  v}\in \R ^3. 
\ee 
Then $R(q)$ is found to be
\[
R(q) = \left( 
\begin{array}{ccc}
q_0^2 +q_1^2 -q_2^2 -q_3^2   &  2(q_1q_2-q_0q_3)                       &  2(q_1q_3+q_0q_2)                    \\
2(q_2q_1+q_0q_3)                      &  q_0^2 -q_1^2 +q_2^2 -q_3^2   &  2(q_2q_3 -q_0q_1)                    \\
2(q_3q_1-q_0q_2)                       &  2(q_3q_2 + q_0q_1)                    &  q_0^2 -q_1^2 -q_2^2 +q_3^2
\end{array}
\right) .
\]
For $q$ given by \eqref{AP20}, then $R(q)$ is the {\em rotation around the axis $\R \vec u$ of angle $\alpha$}. 

Note that $R(q_1 * q_2) = R(q_1)R(q_2)$ (i.e. $R$ is a {\em group homomorphism}), hence 
\[
R(1)=Id,\quad R(q^*)= R(q)^{-1}. 
\]   
We notice that the map $q\to R(q)$ from the unit quaternions set $S^3$ to $SO(3)$ is onto, but not one-to-one, for
$R(-q)=R(q)$. It becomes one-to-one when restricted to the open set
\[S^3_+:=\{ q=[q_0,\vec q \, ]\in {\mathbb H}; \ ||q||=1 \text{ and } q_0>0\}.\]
Furthermore, the map $R$ is a smooth invertible map from $S^3_+$ onto an open neighbourhood $\mathcal O$ of $Id$ in $SO(3)$. On the other hand, the map 
\[
\vec q\to q = [q_0,\vec q \, ]= [\sqrt{1-||\vec q \, ||^2}, \vec q\, ]
\] 
is a smooth invertible map from the unit ball $B_1(0)=\{\vec q\in \R ^3;\ ||\vec q \, || <1 \}$ onto $S^3_+$. Thus the rotations in $\mathcal O$ can be parameterized
by $\vec q \in B_1(0)$. 

\subsection{Proof of Proposition \ref{prop100}.}
Let us prove by induction on $k\in \N$ that 
\be
\label{N8}
V_{2k}^{(2l)}(T)=0 \qquad \forall l\in \N.
\ee
The property  is clearly true for $k=0$, since
\[
V_0^{(2l)} (T) = \hat B ^{(2l)} (T) + \hat A ^{(2l)} (T) \hat C = 0,
\] 
by \eqref{N5}. Assume that \eqref{N8} is established  for some $k\in \N$. Then by \eqref{N4} applied twice, we have 
\[
V_{2k+2}  = V_{2k}'' -2\hat A V_{2k}' -\hat A 'V_{2k} +\hat A^2 V_{2k},
\] 
hence 
\be
\label{N9}
V_{2k+2}^{(2l)} (T)= V_{2k}^{(2l+2)} -2(\hat AV_{2k}') ^{(2l)} (T) - (\hat A'V_{2k})^{(2l)} (T) + (\hat A^2 V_{2k}) ^{(2l)} (T).
\ee
The first term in the r.h.s. of \eqref{N9} is null by \eqref{N8}. The second one is also null, for by Leibniz' rule
\[
(\hat AV'_{2k}) ^{(2l)} (T) = \sum_{p=0}^{2l} C_{2l}^{p} \hat A^{(p)} (T) V_{2k}^{(2l-p+1)} (T)  
\]
and $\hat A^{(p)}(T)=0$ if $p$ is even, while $V_{2k}^{(2l-p+1)} (T)=0$ if $p$ is odd. One proves in a similar way
that the third and fourth terms in the r.h.s. of \eqref{N9} are null, noticing that for $p$ odd we have 
\be
\label{N9bis}
(\hat A^2)^{(p)} (T)= 2(\hat A\hat A')^{(p-1)} (T)=0.
\ee
From \eqref{N8}, we infer that 
\[
V_{2k+1}^{(2l+1)} (T)= V_{2k}^{(2l+2)} (T) - (\hat AV_{2k})^{(2l+1)} (T)=0.
\]
Let us proceed to the proof of \eqref{N7}. Again, we first prove by induction on $k\in\N$ that 
\be
\label{N10}
U_{2k+1}^{(2l)}(T)=0 \qquad \forall l\in \N .
\ee
It follows from \eqref{N3}, \eqref{N4} and \eqref{N6} that 
\[
U_1^{(2l)} (T) = U_0^{(2l+1)} (T) -(DU_0)^{(2l)} (T) -V_0^{(2l)} (T) =0\qquad \forall k\in\N . 
\]
Assume that \eqref{N10} is true for some $k\in\N$. Then, by \eqref{N4} applied twice, 
\be
U_{2k+3}^{(2l)}(T) = U_{2k+1}^{(2l+2)} (T)  - (DU_{2k+1}) ^{(2l+1)}  (T) 
-V_{2k+1} ^{(2l+1)} (T)  - (DU_{2k+2}) ^{(2l)} (T) -V_{2k+2}^{(2l)} (T). 
\label{N11}
\ee
Using \eqref{N5}, \eqref{N6} and \eqref{N10}, we see that all the terms in the r.h.s. of \eqref{N11},
except possibly $(DU_{2k+2})^{(2l)}(T)$, are null. Finally,
\[
(DU_{2k+2})^{(2l)} (T) = (DU_{2k+1}')^{(2l)} (T) - (D^2 U_{2k+1})^{(2l)} (T)   - (DV_{2k+1}) ^{(2l)} (T).
\]
Using Leibniz' rule for each term, noticing that in each pair $(p,q)$ with $p+q=2l$, $p$ and $q$ are simultaneously
even or odd, and using \eqref{N5}, \eqref{N6},  \eqref{N9bis} (with $\hat A$ replaced by $D$), and \eqref{N10}, we conclude 
that $(DU_{2k+2})^{(2l)}(T)=0$, so that $U_{2k+3}^{(2l)} (T)=0$. 

Finally, $U_{2k}^{(2l+1)}(T)=0$ is obvious for $k=0$, while for $k\ge 1$
\[
U_{2k}^{(2l+1)}(T)=U_{2k-1}^{(2l+2)} (T) - (DU_{2k-1}) ^{(2l+1)} (T) -V_{2k-1}^{(2l+1)} (T)=0
\]
by \eqref{N5}, \eqref{N6} and \eqref{N10} (with $2k+1$ replaced by $2k-1$). The proof of Proposition
\ref{prop100} is complete.
\subsection{Proof of Proposition \ref{prop200}}
From \eqref{N2}, \eqref{N3} and \eqref{P4}, we obtain successively 
\begin{eqnarray*}
V_1(T) &=& V_0'(T) = \hat B'(T) + \hat A '(T) \hat C = \overline{w_1} '(T) \big( 
{\mathcal J}^{-1} {\bf  B} + {\mathcal J}^{-1} {\bf  A} {\mathcal J}^{-1} C \big) \\
V_3(T) &=& V_2'(T) \\
&=& (V_1'-  \hat A V_1) '(T) \\
&=& (V_0' -\hat A V_0)''(T) - (\hat A V_1 )'(T) \\
&=& V_0'''(T) -2\hat A '(T) V_0'(T) -\hat A '(T) V_1 (T) \\
&=& -3 \hat A '(T) V_0'(T). 
\end{eqnarray*}
Successive applications of \eqref{N3} yield
\ba
V_5(T) &=& V_0^{(5)} (T) - \sum_{i=0}^3 (\hat A V_i)^{(4-i)} (T), \label{P9} \\
V_7(T) &=& V_0^{(7)} (T) - \sum_{i=0}^5 (\hat A V_i) ^{(6-i)} (T).  \label{P10}
\ea
Since $V_0^{(k)} (T)=0$ for $k\ge 2$, it remains to estimate the terms 
$(\hat A V_i)^{(4-i)} (T)$ and $ (\hat A V_i) ^{(6-i)} (T)$. Notice first that by \eqref{P4} and Leibniz' rule 
\[
(\hat A V_i) ^{(k)} (T)= k\hat A '(T) V_i^{(k-1)} (T). 
\]
Thus, from \eqref{P4} and \eqref{P6}, we have that
\ba
&&(\hat A V_0)^{(4)} (T) = 0 \label{P11}, \\
&&(\hat A V_1)^{(3)} (T) = 3 \hat A '(T) V_1''(T) = 3 \hat A' (T) 
\big( V_0^{(3)} (T) - (\hat A V_0)''(T) \big) =-6\hat A '(T) ^2 V_0'(T) , \qquad \label{P12}\\
&&(\hat A V_2) ''(T) = 2\hat A '(T) V_2 ' (T) =2 \hat A '(T) V_3(T) =-6 \hat A'(T)^2 V_0'(T), \label{P1}\\
&&(\hat A V_3) '(T) = \hat A '(T) V_3(T) =-3 \hat A' (T) ^2 V_0'(T).  \label{P14}
\ea
This yields \eqref{P7}. On the other hand,
\ba
(\hat A V_0)^{(6)} (T) &=& 0, \label{P15} \\
(\hat A V_1) ^{(5)} (T) &=& 5 \hat A ' (T) V_1 ^{(4)}(T)  = 5 \hat A'(T)\big( V_0^{(5)} 
-(\hat A V_0) ^{(4)} \big) (T) =0, \label{P16} \\
(\hat A V_2) ^{(4)} (T) &=& 4 \hat A'(T) V_2^{(3)} (T).
\ea
Since 
\[
V_2= V_1' -\hat A V_1 = V_0'' -(\hat A V_0)'-\hat A V_1,
\]
we obtain with \eqref{P4} and \eqref{P12}  that
\[
V_2^{(3)} (T) = V_0^{(5)} (T) - (\hat A V_0) ^{(4)} (T) - (\hat A V_1 ) ^{(3)} (T) = 6 \hat A '( T) ^2 V_0'(T) , 
\]
hence 
\be
\label{P17}
(\hat A V_2 ) ^{(4)} (T) = 24 \hat A' (T) ^3 V_0'(T). 
\ee
On the other hand, 
\ba
(\hat A V_3 ) ^{(3)} (T) &=& 3\hat A'(T) V_3 ''(T) \nonumber \\
&=& 3 \hat A '(T) \big(  V_4'(T) + ( \hat A V_3 )'(T) \big) \nonumber \\
&=& 3 \hat A '(T) \big( V_5(T) + \hat A'(T) V_3(T) \big) \nonumber \\
&=& 36 \hat A'(T) ^3 V_0'(T)  \label{P18} 
\ea
where we used \eqref{N3} and \eqref{P6}-\eqref{P7}. Finally, 
\be
(\hat A V_4) ''(T) = 2\hat A '(T) V_4'(T) = 2\hat A '( T)V_5(T) = 30 \hat A '(T) ^3 V_0'(T)  \label {P19} 
\ee
and 
\be
(\hat A V_5) '(T) = \hat A'(T) V_5 (T) = 15 \hat A'(T) V_0'(T). \label{P20} 
\ee
Gathering together \eqref{P10} and \eqref{P15}-\eqref{P20}, we obtain \eqref{P8}. 
The proof of Proposition \ref{prop200} is complete.
\subsection{Proof of Proposition \ref{prop300}}
From \eqref{N3}-\eqref{N4}, we have that 
\be
\label{P25}
U_0\equiv \hat C, \qquad U_i=U_{i-1}' -D U_{i-1} -V_{i-1}, \quad \forall i\ge 1. 
\ee
Thus
\begin{eqnarray*}
U_2(T) 
&=& (U_1'-DU_1 - V_1 )(T) \\
&=& (0-(DU_0)' -V_0') (T) - V_1 (T) \\
&=& -D'(T) U_0 -2 V_0'(T) 
\end{eqnarray*}
where we used successively \eqref{P25}, \eqref{N7} and \eqref{P5}. 

Successive applications of \eqref{P25} yield
\be
U_4(T) =-\sum_{i=0}^3 [(DU_i) ^{ (3-i) } + V_i ^{ (3-i) } ] (T). \label{P26}
\ee
Using \eqref{P4}, we obtain that
\begin{eqnarray}
\sum_{i=0}^3 (DU_i) ^{(3-i)} (T) &=& \sum_{i=0}^2 (3-i) D'(T) U_i ^{(2-i)} (T)\nonumber  \\
&=& 2D'(T) \big( U_2(T) +V_1(T)  \big) + D'(T) V_0'(T) \nonumber \\
&=& -3D'(T) \big( D'(T) U_0 + 2V_0'(T)  \big) + 2D'(T) V_0'(T) \nonumber\\  
&=& -4 D'(T) V_0'(T) \label{P27}  
\end{eqnarray}
where we used \eqref{P5}, \eqref{P21} and the fact that $D'(T)^2=0$. 

On the other hand,
\ba
\sum_{i=0}^3 V_i ^{(3-i)}(T) &=& (V_0'-\hat A V_0) ''(T) + V_2'(T) + V_3(T)\nonumber \\
&=& -2 \hat A' (T) V_0'(T) + 2V_3 (T) \nonumber\\
&=& - 8 \hat A '(T) V_0'(T)\label{P28} 
\ea 
by \eqref{P6}. Combining \eqref{P26}-\eqref{P28}, we obtain \eqref{P22}. 

Let us now compute $U_6(T)$. Successive applications of \eqref{P25} 
yield
\be
\label{P29}
U_6(T) = -\sum_{i=0}^5 [(DU_i) ^{ (5-i) } + V_i^{ (5-i) } ] (T). 
\ee 
We have that 
\[
\sum_{i=0}^5 (DU_i)^{(5-i)} (T) = \sum_{i=0}^4  (5-i) D'(T) U_i ^{(4-i)} (T). 
\]
Let us estimate the terms $U_i^{(4-i)}(T)$ for $i=0,...,4$. Obviously, $U_0^{(4)} (T) =0$ by \eqref{P25}, while by \eqref{P4} 
\be
U_1^{(3)} (T) = -(DU_0)^{(3)} (T) - V_0^{(3)} (T) =0. 
\label{P30}
\ee
Next we use \eqref{P25} to obtain successively 
\ba
 U_3'(T ) &=& U_4 (T) + V_3(T), \label{P30bis}\\
U_2 ''(T) &=& U_3'(T) + (DU_2)'(T) + V_2'(T) \nonumber \\
&=& U_4(T) + V_3(T) + D'(T) U_2(T) +V_3(T). \nonumber  
\ea
It follows that 
\ba
&&\sum_{i=0}^4 (DU_i) ^{(5-i)} (T) \nonumber \\
&&\qquad = 
3D'(T) \big( U_4(T) + 2 V_3(T) + D'(T) U_2(T)  \big) + 2D'(T) \big( U_4 (T) + V_3(T) \big)  + D'(T) U_4(T) \nonumber  \\
&&\qquad = D'(T) \big( 6 U_4(T) + 8 V_3(T)\big)  \nonumber \\ 
&&\qquad = 24 D'(T) (D'(T) + 2\hat A '(T)) V_0'(T)  - 24 D'(T) \hat A '(T) V_0'(T) \nonumber  \\
&&\qquad = 24 D'(T) \hat A '(T) V_0'(T). \label{P30ter}         
\ea
On the other hand, using \eqref{P16}-\eqref{P18} and \eqref{P6}-\eqref{P7}, we have that
\ba
\sum_{i=0}^4 V_i^{(5-i)} (T) 
&=& V_1 ^{(4)} (T)  + V_2 ^{(3)} (T) + V_3 ^{(2)} (T) + V_4 '(T) \nonumber \\
&=& 6\hat A'(T) ^2 V_0'(T) + 2V_5(T) + \hat A'(T) V_3(T) \nonumber \\
&=& 33 \hat A'(T) ^2 V_0'(T).  \label{P31}  
\ea 
\eqref{P23} follows from \eqref{P29}-\eqref{P31}.

Finally, we compute $U_8(T)$. We see that 
\be
\label{P32}
U_8(T) = -\sum_{i=0}^7 [(DU_i)^{(7-i)} + V_i ^{(7-i)} ](T). 
\ee
Then 
\begin{eqnarray*}
\sum_{i=0}^7 (DU_i)^{ (7-i) } (T) &=& \sum_{i=0}^6 (7-i) D'(T) U_i ^{ (6-i) } (T) \\
&=& 6D'(T) U_1^{ (5) } (T) + 5D'(T) U_2 ^{ (4) } (T) + 4 D'(T) U_3 ^{ (3) } (T) \\
&&\quad + 3D'(T) U_4 ''(T) + 2D '(T) U_5'(T) + D'(T) U_6(T). 
\end{eqnarray*}
Using \eqref{P4}, \eqref{P25} and  \eqref{P30}, we readily see that 
\[
U_1^{(5)}(T) = U_2^{(4)} (T) =0.
\]
Next, successive applications of \eqref{P25} give
\begin{eqnarray*}
U_5'(T)  &=& U_6(T) + V_5(T), \\
U_4''(T) &=& U_5'(T) + (DU_4)' (T) + V_4'(T) \\
&=& U_6(T) + D'(T) U_4(T) + 2V_5(T) .\\
U_3^{(3)} (T) &=& U_4''(T) + (DU_3)''(T) + V_3''(T) \\
&=& \big(  U_6(T) +D'(T) U_4(T) + 2V_5(T)  \big) + 2D'(T) \big(U_4(T) + V_3(T)\big)  \\
&& \quad + V_5(T) + \hat A' (T) V_3(T).  
\end{eqnarray*}
Thus
\begin{eqnarray}
\sum_{i=0}^7 (DU_i)^{(7-i)} (T) 
&=& 4D'(T) \big( U_6(T) + 3D'(T) U_4(T) + 2V_5(T) + 2D'(T) V_3(T) + \hat A'(T) V_3(T) \big) \nonumber \\
&&\quad + 3D'(T) \big( U_6(T) + D'(T) U_4(T) +  2V_5(T)   \big) \nonumber \\
&&\quad +2D'(T) \big(  U_6(T) + V_5(T) \big) +D'(T) U_6(T)\nonumber \\
&=& D'(T) [10\, U_6(T) +16\, V_5(T) + 4 \hat A'(T) V_3(T) ]\nonumber\\
&=& D'(T) [ -240\,  D'(T) \hat A '(T) V_0'(T) -330\,  \hat A' (T) V_0'(T) \nonumber \\
&& \quad + 240\,  \hat A '(T) ^2 V_0'(T) -12 \hat A '(T)^2 V_0'(T) ] \nonumber \\
&=& -102\,  D'(T) \hat A'(T) ^2 V_0'(T). \label{P33} 
\end{eqnarray}
It remains to compute $\sum_{i=0}^7 V_i^{(7-i)} (T)$. 
It is easy to see that 
\[
V_0^{ (7) } (T) = V_1 ^{ (6) }(T) = V_2 ^{ (5) }(T) =0. 
\]
Successive applications of \eqref{N3} give
\begin{eqnarray*}
V_6'(T) &=& V_7(T), \\
V_5''(T) &=& V_6'(T) + (\hat A V_5) '(T) = V_7 (T) + \hat A'(T) V_5(T), \\
V_4^{(3)}  (T) &=& V_5''(T) + (\hat A V_4) ''(T) \\    
&=& V_7(T) + 3 \hat A'(T) V_5(T), \\
V_3^{(4)} (T) &=& V_4 ^{(3)} (T) + (\hat A V_3) ^{(3)} (T) \\
&=& V_7 (T) + 3 \hat A '(T) V_5(T) + 3 \hat A '(T) V_3 ''(T) \\
&=& V_7(T) + 6\hat A'(T)V_5(T) + 3 \hat A '(T) ^2 V_3(T),  
\end{eqnarray*}
where we used \eqref{P18}. Thus
\begin{eqnarray}
\sum_{i=0}^7 V_i^{(7-i)} (T) 
&=& 5 V_7 (T) + 10 \hat A'(T) V_5(T) + 3 \hat A'(T) ^2 V_3 (T) \nonumber \\
&=& -384\,  \hat A'(T) ^3 V_0'(T).   \label{P34}
\end{eqnarray}
Then \eqref{P24} follows from \eqref{P32}-\eqref{P34}. The proof of Proposition \ref{prop300} is achieved.
\section{Acknowledgements}
The authors wish to thank Philippe Martin (Ecole des Mines, Paris) who brought the reference \cite{ST} to their attention. 
The authors were partially supported by the Agence Nationale de la Recherche, Project CISIFS, grant ANR-09-BLAN-0213-02.
The authors wish to thank  the Basque Center for Applied Mathematics-BCAM,  Bilbao (Spain), where part of this work was developed.

\bibliographystyle{abbrv}
\bibliography{Control}

\def\cprime{$'$}
\begin{thebibliography}{10}

\bibitem{altmann}
S.~L. Altmann.
\newblock {\em Rotations, quaternions, and double groups}.
\newblock Oxford Science Publications. The Clarendon Press Oxford University
  Press, New York, 1986.

\bibitem{ACO}
A.~Astolfi, D.~Chhabra, and R.~Ortega.
\newblock Asymptotic stabilization of some equilibria of an underactuated
  underwater vehicle.
\newblock {\em Systems Control Lett.}, 45(3):193--206, 2002.

\bibitem{BKMS}
A.~M. Bloch, P.~S. Krishnaprasad, J.~E. Marsden, and G.~S{\'a}nchez~de Alvarez.
\newblock Stabilization of rigid body dynamics by internal and external
  torques.
\newblock {\em Automatica J. IFAC}, 28(4):745--756, 1992.

\bibitem{Chambrion-Sigalotti}
T.~Chambrion and M.~Sigalotti.
\newblock Tracking control for an ellipsoidal submarine driven by {K}irchhoff's
  laws.
\newblock {\em IEEE Trans. Automat. Control}, 53(1):339--349, 2008.

\bibitem{CCOR}
C.~Conca, P.~Cumsille, J.~Ortega, and L.~Rosier.
\newblock On the detection of a moving obstacle in an ideal fluid by a boundary
  measurement.
\newblock {\em Inverse Problems}, 24(4):045001, 18, 2008.

\bibitem{CMM}
C.~Conca, M.~Malik, and A.~Munnier.
\newblock Detection of a moving rigid body in a perfect fluid.
\newblock {\em Inverse Problems}, 26:095010, 2010.

\bibitem{Coron1}
J.-M. Coron.
\newblock On the controllability of {$2$}-{D} incompressible perfect fluids.
\newblock {\em J. Math. Pures Appl. (9)}, 75(2):155--188, 1996.

\bibitem{coron-book}
J.-M. Coron.
\newblock {\em Control and nonlinearity}, volume 136 of {\em Mathematical
  Surveys and Monographs}.
\newblock American Mathematical Society, Providence, RI, 2007.

\bibitem{fossen-book}
T.~I. Fossen.
\newblock {\em Guidance and Control of Ocean Vehicles}.
\newblock New York: Wiley, 1994.

\bibitem{Fossen}
T.~I. Fossen.
\newblock A nonlinear unified state-space model for ship maneuvering and
  control in a seaway.
\newblock {\em Internat. J. Bifur. Chaos Appl. Sci. Engrg.}, 15(9):2717--2746,
  2005.

\bibitem{Glass}
O.~Glass.
\newblock Exact boundary controllability of 3-{D} {E}uler equation.
\newblock {\em ESAIM Control Optim. Calc. Var.}, 5:1--44 (electronic), 2000.

\bibitem{GR}
O.~Glass and L.~Rosier.
\newblock On the control of the motion of a boat.
\newblock {\em Math. Models Methods Appl. Sci.}, 23(4):617--670, 2013.

\bibitem{Yudovich64}
V.~I. Judovi{\v{c}}.
\newblock A two-dimensional non-stationary problem on the flow of an ideal
  incompressible fluid through a given region.
\newblock {\em Mat. Sb. (N.S.)}, 64 (106):562--588, 1964.

\bibitem{Kazhikhov}
A.~V. Kazhikhov.
\newblock Note on the formulation of the problem of flow through a bounded
  region using equations of perfect fluid.
\newblock {\em Prikl. Matem. Mekhan.}, 44(5):947--950, 1980.

\bibitem{Kikuchi86}
K.~Kikuchi.
\newblock The existence and uniqueness of nonstationary ideal incompressible
  flow in exterior domains in {${\bf R}\sp 3$}.
\newblock {\em J. Math. Soc. Japan}, 38(4):575--598, 1986.

\bibitem{Lamb}
H.~Lamb.
\newblock {\em Hydrodynamics}.
\newblock Cambridge Mathematical Library. Cambridge University Press,
  Cambridge, sixth edition, 1993.
\newblock With a foreword by R. A. Caflisch [Russel E. Caflisch].

\bibitem{Leonard97}
N.~E. Leonard.
\newblock Stability of a bottom-heavy underwater vehicle.
\newblock {\em Automatica J. IFAC}, 33(3):331--346, 1997.

\bibitem{Leonard-Marsden}
N.~E. Leonard and J.~E. Marsden.
\newblock Stability and drift of underwater vehicle dynamics: mechanical
  systems with rigid motion symmetry.
\newblock {\em Phys. D}, 105(1-3):130--162, 1997.

\bibitem{NS}
S.~P. Novikov and I.~Shmel{\cprime}tser.
\newblock Periodic solutions of {K}irchhoff equations for the free motion of a
  rigid body in a fluid and the extended {L}yusternik-{S}hnirel\cprime
  man-{M}orse theory. {I}.
\newblock {\em Funktsional. Anal. i Prilozhen.}, 15(3):54--66, 1981.

\bibitem{ORT1}
J.~H. Ortega, L.~Rosier, and T.~Takahashi.
\newblock Classical solutions for the equations modelling the motion of a ball
  in a bidimensional incompressible perfect fluid.
\newblock {\em M2AN Math. Model. Numer. Anal.}, 39(1):79--108, 2005.

\bibitem{ORT2}
J.~H. Ortega, L.~Rosier, and T.~Takahashi.
\newblock On the motion of a rigid body immersed in a bidimensional
  incompressible perfect fluid.
\newblock {\em Ann. Inst. H. Poincar\'e Anal. Non Lin\'eaire}, 24(1):139--165,
  2007.

\bibitem{RR2008}
C.~Rosier and L.~Rosier.
\newblock Smooth solutions for the motion of a ball in an incompressible
  perfect fluid.
\newblock {\em J. Funct. Anal.}, 256(5):1618--1641, 2009.

\bibitem{sontag-book}
E.~D. Sontag.
\newblock {\em Mathematical control theory}, volume~6 of {\em Texts in Applied
  Mathematics}.
\newblock Springer-Verlag, New York, 1990.
\newblock Deterministic finite-dimensional systems.

\bibitem{ST}
B.~L. Stevens and F.~L. Lewis.
\newblock {\em Aircraft Control and Simulation}.
\newblock John Wiley $\&$ Sons, Inc., Hoboken, New Jersey, 2003.

\bibitem{WZ}
Y.~Wang and A.~Zang.
\newblock Smooth solutions for motion of a rigid body of general form in an
  incompressible perfect fluid.
\newblock {\em J. Differential Equations}, 252(7):4259--4288, 2012.

\end{thebibliography}
\end{document}